\documentclass[english]{amsart} 
 
\usepackage{enumerate}
\usepackage{amssymb}
\usepackage{array}
\usepackage{mathrsfs}
\usepackage{mathtools} 
\usepackage{ stmaryrd }
\usepackage{alphabeta}
\usepackage[normalem]{ulem}
 
\usepackage{color}	
\usepackage[dvipsnames,svgnames,table]{xcolor}
\usepackage[curve,all]{xy}
\xyoption{matrix}
 \xyoption{curve}
 \xyoption{color}
 \xyoption{line}
\xyoption{arc}

\usepackage[shortlabels]{enumitem}

\setlist[enumerate]{leftmargin=6mm,topsep=0pt,itemsep=-1ex,partopsep=1ex,parsep=1ex,label=\rm{(\arabic*)}}
\setlist[enumerate,2]{label=\rm({\it\roman*}), ref=\rm({\it\roman*})}
\setlist[itemize]{leftmargin=6mm,topsep=0pt,itemsep=-1ex,partopsep=1ex,parsep=1ex,label=\raisebox{0.25ex}{\tiny$\bullet$}}

\usepackage[backref, colorlinks, linktocpage, citecolor = blue, linkcolor = blue]{hyperref}

\theoremstyle{plain}
\newtheorem{theoremA}{Theorem}
\newtheorem{propositionA}[theoremA]{Proposition}

\newtheorem{theorem}{Theorem}[subsection]
\newtheorem*{theoremaux}{Theorem \theoremauxnum}
\gdef\theoremauxnum{1}

\newtheorem{proposition}[theorem]{Proposition}
\newtheorem*{propositionaux}{Proposition \propositionauxnum}
\gdef\propositionauxnum{1}

\newtheorem{lemma}[theorem]{Lemma}
\newtheorem*{lemmaaux}{Lemma \lemmaauxnum}
\gdef\lemmaauxnum{1}

\newtheorem{corollary}[theorem]{Corollary}

\theoremstyle{definition}
\newtheorem{definition}[theorem]{Definition}
\newtheorem{notation}[theorem]{Notation}
\newtheorem{example}[theorem]{Example}

\theoremstyle{remark}
\newtheorem{remark}[theorem]{Remark}

\newcommand\iso{\stackrel{\simeq}{\longrightarrow}}

\newcommand{\incl}[1][r]{\ar@<-0.2pc>@{^(-}[#1] \ar@<+0.2pc>@{-}[#1]}

\renewcommand{\P}{\mathbb{P}}

\renewcommand{\L}{\mathcal{L}}

\newcommand{\V}{\mathcal{V}}
\renewcommand{\SS}{\mathcal{S}}
\newcommand{\PP}{\mathcal{P}}
\newcommand{\FF}{\mathcal{F}}
\newcommand{\cO}{\mathcal{O}}

\newcommand{\TT}{\hat{\mathcal{S}}}
\newcommand{\p}{\mathbb{P}}
\newcommand{\s}[1]{s_{#1}}

\newcommand{\Id}{\mathrm{id}}

\renewcommand{\k}{\mathrm{k}}

\renewcommand{\Im}{\mathrm{Im}}

\newcommand{\A}{\mathbb{A}}
\newcommand{\T}{\mathcal{T}}
\newcommand{\C}{\mathbb{C}}
\newcommand{\F}{\mathbb{F}}
\newcommand{\N}{\mathbb{N}}

\newcommand{\M}{\mathcal{M}}
\newcommand{\U}{\mathcal{U}}
\newcommand{\Z}{\mathbb{Z}}

\newcommand{\G}{\mathbb{G}}

\renewcommand{\O}{\mathcal{O}}
\renewcommand{\V}{\mathcal{V}}

\newcommand{\OP}{\mathcal{O}_{\mathbb{P}^1}}
\newcommand{\OPP}{\mathcal{O}_{\mathbb{P}^2}}

\newcommand{\OFa}{\mathcal{O}_{\mathbb{F}_a}}

\DeclareMathOperator{\SL}{SL}
\DeclareMathOperator{\Aut}{Aut}
\DeclareMathOperator{\Hom}{Hom}

\DeclareMathOperator{\Autz}{Aut^{\circ}}
\DeclareMathOperator{\PGL}{PGL}

\DeclareMathOperator{\Pic}{Pic}

\DeclareMathOperator{\GL}{GL}
\DeclareMathOperator{\Bir}{Bir}

\DeclareMathOperator{\Ker}{Ker}

\makeatletter
\@namedef{subjclassname@2020}{%
  \textup{2020} Mathematics Subject Classification}
\makeatother

\title[Automorphisms of $\P^1$-bundles over surfaces in characteristic $p$]{Automorphisms of $\P^1$-bundles over rational surfaces in positive characteristic}
\date{\today}

\author{J\'er\'emy Blanc}
\address{J\'er\'emy Blanc, Universit\'e de Neuch\^atel,
	Rue Emile-Argand 11,
	CH-2000 Neuch\^atel,
	Switzerland}
\email{jeremy.blanc@unine.ch}

\author{Ronan Terpereau}
\address{Univ. Lille, CNRS, UMR 8524 - Laboratoire Paul Painlev\'e, F-59000 Lille, France}
\email{ronan.terpereau@univ-lille.fr}


\begin{document}

\begin{abstract}
We study $\mathbb{P}^1$-bundles over smooth projective rational surfaces over an algebraically closed field of characteristic $p>0$, together with the actions of their neutral automorphism groups.
After an equivariant square birational reduction to the projective plane or a Hirzebruch surface, we classify the possible induced actions on the base and identify the families relevant to relative maximality.
For $p\geq 3$, we prove that every such $\mathbb{P}^1$-bundle is equivariantly square birational to one whose neutral automorphism group is relatively maximal, and we classify all bundles with this property.
We also describe all equivariant square birational maps between these relatively maximal bundles and determine which are stiff or superstiff, the analogues of rigidity and superrigidity for $\mathbb{P}^1$-bundles.
\end{abstract}

\maketitle

\tableofcontents
 
 \vspace{-10mm}
 
\section{Introduction} 

\subsection{Aim and scope}
Throughout the article, $\k$ is an algebraically closed field of
characteristic $p>0$. We study Zariski locally trivial $\P^1$-bundles
$\pi\colon X\to S$ over smooth projective rational surfaces $S$, together
with the identity component $\Autz(X)$ of the automorphism group of $X$.
We describe these pairs up to equivariant birational maps that preserve
the $\P^1$-bundle structures (which we call square birational maps;
see the definition below), and classify the pairs for which $\Autz(X)$
is relatively maximal.
The classification of $\P^1$-bundles with relatively maximal
neutral automorphism group is obtained for $p\geq3$. 
In characteristic $2$, additional phenomena arise,
and the general classification of relatively maximal
neutral automorphism groups remains incomplete.

Our motivation comes in part from the classification of connected
algebraic subgroups of the Cremona group $\Bir(\P^3)$, initiated by
Enriques and Fano~\cite{EF1898} and completed over $\C$ by Umemura in
a series of papers~\cite{Ume80,Ume82a,Ume82b,Ume85}; see also
\cite{MU83,Ume88}. We discuss this broader motivation in
Section~\ref{sec: classification max conn subgroups of Cr3}.

In characteristic zero, $\P^1$-bundles over rational surfaces and their
neutral automorphism groups were studied in~\cite{BFT}. 
Let us emphasize that five of the eight infinite families of Mori fiber spaces appearing in the classification of maximal connected algebraic subgroups of $\Bir(\P^3)$ consist of $\P^1$-bundles; see~\cite[Theorem~E]{BFT-MFS}.
The present article extends this study to positive characteristic,
where new examples arise, including positive-dimensional families
of relatively maximal $\P^1$-bundles, a phenomenon absent in
characteristic zero. These families consist of Umemura
$\P^1$-bundles over $\F_a$ and their contractions to $\P^2$ when $a=1$.

\subsection{Main results}
We study connected automorphism groups through equivariant
birational maps that preserve the given $\P^1$-bundle structures.
Such maps may modify both the total space and the base.
We begin by fixing the terminology needed to compare these models.

\begin{definition}$($\cite[Definition~1.2.1]{BFT}$)$
	Let $\pi\colon X\to S$ and $\pi'\colon X'\to S'$ be $\P^1$-bundles over smooth projective rational surfaces.
	\begin{itemize}
		\item A birational map $\varphi\colon X\dasharrow X'$ is \emph{square} if there exists a birational map $\eta\colon S\dasharrow S'$ such that the diagram
		\[
		\xymatrix@R=4mm@C=2cm{
			X \ar@{-->}[r]^{\varphi} \ar[d]_{\pi} & X' \ar[d]^{\pi'} \\
			S \ar@{-->}[r]_{\eta} & S'
		}
		\]
		commutes. We then say that $\varphi$ is \emph{above} $\eta$. If both $\varphi$ and $\eta$ are isomorphisms, we call $\varphi$ a \emph{square isomorphism}; if moreover $X=X'$, $S=S'$, and $\pi=\pi'$, we call it a \emph{square automorphism}.
		
		\item If $S=S'$ and $\pi'\varphi=\pi$, we say that $\varphi$ is a \emph{birational map of $\P^1$-bundles}. If $\varphi$ is an isomorphism, we call it an \emph{isomorphism of $\P^1$-bundles}.
		
		\item Finally, a square birational map $\varphi\colon X\dasharrow X'$ is \emph{$\Autz(X)$-equivariant} if
		\[
		\varphi\Autz(X)\varphi^{-1}\subseteq\Autz(X').
		\]
	\end{itemize}
\end{definition}

We now recall the notions of relative maximality,
stiffness, and superstiffness.

\begin{definition}$($\cite[Definition~1.2.3]{BFT} and \cite[Definition~2.1]{Fong}$)$\\
	Let $\pi\colon X\to S$ be a $\P^1$-bundle over a smooth projective rational surface.
	\begin{itemize}
		\item We say that $\Autz(X)$ is \emph{relatively maximal} if, for every $\Autz(X)$-equivariant square birational map
		\[
		\varphi\colon (X,\pi)\dasharrow(X',\pi'),
		\]
		we have
		\[
		\varphi\Autz(X)\varphi^{-1}=\Autz(X').
		\]
		\item We say that $(X,\pi)$ is \emph{stiff} if moreover, for every
		such $\varphi$, the $\P^1$-bundles $(X,\pi)$ and $(X',\pi')$
		are square-isomorphic, and \emph{superstiff} if every such
		$\varphi$ is itself a square isomorphism.
	\end{itemize}
\end{definition}

\begin{remark}
	Relative maximality, stiffness, and superstiffness depend on the $\P^1$-bundle structure and not only on the variety $X$. For instance, let $X=\P^1\times\F_1$
	with its two natural $\P^1$-bundle structures
	\[
	\pi\colon X\to\P^1\times\P^1
	\qquad\text{and}\qquad
	\pi'\colon X\to\F_1.
	\]
	Then $\Autz(X)$ is relatively maximal with respect to $\pi$, but not with respect to $\pi'$.
\end{remark}

\begin{remark}
The adjective \emph{relatively} indicates that we restrict attention to $\P^1$-bundles: relative maximality concerns equivariant square birational maps to other $\P^1$-bundles. We use \emph{maximality} when allowing equivariant birational maps to arbitrary Mori fiber spaces. Among $\P^1$-bundles, stiffness and superstiffness correspond to rigidity and superrigidity, respectively.
\end{remark}

The induced action on the base plays an essential role in our study.
Already for Hirzebruch surfaces, connected automorphism groups in
positive characteristic exhibit a behavior absent in characteristic
zero. Laurent asked in \cite[Remarque~3.5.10]{Lau} whether every connected algebraic subgroup
\[H\subseteq\Aut(\F_a) \simeq \k[z_0,z_1]_a \rtimes(\GL_2/\mu_a)\] 
acting on $\F_a$ with exactly two orbits is conjugate to
a semidirect product of the form
$V\rtimes(\GL_2/\mu_a)$ or $V\rtimes S$, where
$S=\Im(\SL_2\to\GL_2/\mu_a)$ and
$V\subseteq\k[z_0,z_1]_a$ is a $\GL_2$-submodule. 
We answer this question negatively in
every positive characteristic: for $a=2p$, we construct such a group
$H$ for which the induced extension of $S$ is non-split.

\begin{propositionA}[{Proposition~\ref{prop: UniformCounterexampleNew}}]
	Let $\k$ be an algebraically closed field of characteristic $p>0$.
	There exists a connected algebraic subgroup
	$H\subseteq\Aut(\F_{2p})$ acting with exactly two orbits,
	$s_{-2p}$ and its complement, whose projection to
	$\GL_2/\mu_{2p}$ induces a non-split exact sequence
	\[
	1\longrightarrow\G_a^3\longrightarrow H
	\longrightarrow
	S=\Im(\SL_2\to\GL_2/\mu_{2p})\simeq\PGL_2
	\longrightarrow1.
	\]
	In particular, $H$ contains no conjugate of $S$ in
	$\Aut(\F_{2p})$.
\end{propositionA}

For $p\geq3$, the evenness of the Hirzebruch index is essential.
Indeed, if $a$ is odd, every connected algebraic subgroup
$H\subseteq\Aut(\F_a)$ surjecting onto $\PGL_2$ contains
a conjugate of $S$ by an element of
$U_a=\k[z_0,z_1]_a$
(Lemma~\ref{lem:H-maps-onto-PGL2-odd-a}).
In particular, if its image in $\GL_2/\mu_a$ is $S$,
the extension of $S$ by $H\cap U_a$ splits.

The counterexample concerns arbitrary connected algebraic subgroups
of $\Aut(\F_a)$. By contrast,
Corollary~\ref{cor:H-maps-onto-PGL2-even-a} shows that,
when $a\geq2$ is even, the image
$H$ of $\Autz(X)$ for a $\P^1$-bundle $X\to\F_a$
contains a conjugate of the standard subgroup
$\Im(\SL_2\to\GL_2/\mu_a)$ whenever its action on
the base of the ruling $\F_a\to\P^1$ is transitive.
Thus the phenomenon exhibited by the counterexample in Proposition~\ref{prop: UniformCounterexampleNew} cannot occur for these groups arising from $\P^1$-bundles over $\F_a$.

\smallskip

We now turn to $\P^1$-bundles over smooth projective rational
surfaces. The descent lemma (Lemma~\ref{lem:descent-lemma})
reduces the study to $\P^1$-bundles over minimal rational surfaces:
contracting the base yields an equivariant square birational
model over $\P^2$ or a Hirzebruch surface.
The surface $\F_1$ also plays a central role, since blowing up
a point of $\P^2$ and its fiber connects $\P^1$-bundles over these
two surfaces.

Over $\F_a$, with $a\geq1$, we introduce the Umemura
$\P^1$-bundles $\U_{a,\omega}^{b,c}$ in
Section~\ref{Sec:Umemurabundles}, extending the characteristic-zero
family considered in \cite[\S~10]{Ume88} and \cite[\S~3.6]{BFT}.
For fixed numerical invariants $(a,b,c)$, their isomorphism
classes over $\F_a$ are parametrized by a projective space
$\Omega_a^{b,c}$.
In characteristic zero, this space is empty or a point;
in positive characteristic, its dimension can be arbitrarily
large as $(a,b,c)$ varies
(Example~\ref{example: Omegaabc not bounded}).
Thus the parameter $\omega$ is an essential part of the
classification. Their neutral automorphism groups act
transitively on the base of the ruling, but their images in
$\Aut(\F_a)$ need not be the full group.

Despite these additional phenomena, including the nonsplitting
example above, the restrictions imposed by the $\P^1$-bundle structure
lead to the following trichotomy.

\begin{theoremA}[{Theorem~\ref{th:P1bundlesFa}}]
Let $a\geq0$, and let $\pi\colon X\to\F_a$ be a
$\P^1$-bundle.
Assume that $p\geq3$ or that $a$ is even.
	Then exactly one of the following holds:
	\begin{enumerate}
		\item
		$X\to\F_a$ is decomposable, hence isomorphic to
		$\FF_a^{b,c}\to\F_a$ for some $b,c\in\Z$
		$($Definition~$\ref{def Fabc})$;
		
		\item
		$a\geq1$ and $X\to\F_a$ is square-isomorphic to an Umemura
		$\P^1$-bundle $\U_{a,\omega}^{b,c}\to\F_a$, for some
		$b,c,\omega$ as in Definition~$\ref{def Umemura bundle}$;
		
		\item
		$X\to\F_a$ is indecomposable, and either $a=0$, or $a\geq1$
		and $\Autz(X)$ fixes a point on the base of the ruling
		$\F_a\to\P^1$.
		In both cases, there exists a decomposable $\P^1$-bundle
		$\FF_a^{b,c}\to\F_a$ and a commutative diagram
		\[
		\xymatrix@R=3mm@C=2cm{
			X \ar@{-->}[rr]^{\psi} \ar[rd]_{\pi}
			&& \FF_a^{b,c} \ar[ld] \\
			& \F_a
		}
		\]
		where $\psi$ is a birational map satisfying
		$\psi\Autz(X)\psi^{-1}\subsetneq\Autz(\FF_a^{b,c})$.
	\end{enumerate}
\end{theoremA}

The third case provides an explicit obstruction to relative
maximality: the $\P^1$-bundle admits an equivariant birational model
with a strictly larger neutral automorphism group.
Theorem~\ref{th:P1bundlesFa} thus reduces the study of relative
maximality over $\F_a$ to Umemura and decomposable $\P^1$-bundles
when $p\geq3$ or $a$ is even.
The case $(a,p)=(1,2)$ requires a different conclusion:
Lemma~\ref{lem: counter example for p=2} gives an indecomposable
$\P^1$-bundle over $\F_1$ which is not square-isomorphic
to an Umemura $\P^1$-bundle, although its neutral automorphism
group acts transitively on the base of the ruling.

\smallskip

Over the projective plane $\P^2$, two families account for the new phenomena.
The first consists of the $\P^1$-bundles $\L^{s,b}$, with
$s\geq0$ and $1\leq b\leq p^s$ (see Section~\ref{sec: Ls bundles}).
For $b<p^s$, their neutral automorphism groups induce
the full stabilizer of a line in $\PGL_3$, a possibility
excluded in characteristic zero by
\cite[Proposition~4.3.4]{BFT}.
At the endpoint $b=p^s$, the $\P^1$-bundle
$\L^s:=\L^{s,p^s}$ is the pullback of
$\P(T_{\P^2})\to\P^2$ by the $s$-th Frobenius morphism.
It is a homogeneous variety $\PGL_3/P_s$, where $P_s$
is a non-smooth parabolic subgroup scheme when $s>0$
(Remark~\ref{rk: Ps}).

The second family consists of generalized Schwarzenberger
$\P^1$-bundles $\SS^{b,c}$.
The classical Schwarzenberger $\P^1$-bundles considered in \cite[\S~4]{Ume88} and \cite[\S~4.2]{BFT}, indexed by a single integer $b$,
correspond to $c=b+2$; positive characteristic allows
additional values of $c$.
Their neutral automorphism groups are isomorphic to
$\PGL_2$ and preserve a smooth conic on the base,
except when $c=3b$.
In that case, $b=p^s$ for some $s\geq0$ and
$\SS^{b,3b}\simeq\L^s$, so the two families meet precisely
at these Frobenius pullbacks of the incidence $\P^1$-bundle.

Together with the decomposable $\P^1$-bundles, these families give
the following classification over $\P^2$.
Its fixed-point alternative returns us to $\F_1$, so the
classifications over Hirzebruch surfaces and over $\P^2$
can be used together.

\begin{theoremA}[{Theorem~\ref{th:P1bundlesP2}}]
Assume that $p \geq 3$.
Let $\pi\colon X\to\P^2$ be a $\P^1$-bundle and let $H:=\Im(\Autz(X)\to\Aut(\P^2))$. Then one of the following holds:
\begin{enumerate}
	\item  $H=\Aut(\P^2)$, and $\pi$ is either a decomposable $\P^1$-bundle or square-isomorphic to the $\P^1$-bundle $\L^{s} \to \P^2$ $($Definition~$\ref{def: P1 bundles Lsb})$, for some $s \geq 0$;
	\item  $H=\Aut(\P^2,\Gamma)\simeq\PGL_2$, where $\Gamma\subset\P^2$ is a smooth conic, and $\pi$ is square-isomorphic to a Schwarzenberger $\P^1$-bundle $\SS^{b,c} \to \P^2$, for some $b \geq 1$ and $c \geq 3$ as in Definition~$\ref{def:Schwarz bc}$ with $c \neq 3b$;
	\item  $H = \Aut(\P^2,L)\simeq \GL_2\ltimes\G_a^2$, where $L\subset\P^2$ is a line, and $\pi$ is square-isomorphic to the $\P^1$-bundle $\L^{s,b}\to\P^2$ (Definition~$\ref{def: P1 bundles Lsb}$), for some $s\ge1$ and $1\le b< p^s$; or
	\item there exists an $H$-fixed point $p_0\in\P^2$, and the blow-up $\eta\colon \hat X\to X$
	of the fiber $\pi^{-1}(p_0)$ yields a commutative diagram of $\Autz(X)$-equivariant maps
	\[
	\xymatrix@R=4mm@C=2cm{
		\hat X \ar[r]^{\eta} \ar[d]_{\hat\pi} & X \ar[d]^\pi \\
		\F_1 \ar[r]_{\tau} & \P^2
	}
	\]
	where $\hat{X}\to \F_1$ is a $\P^1$-bundle and $\tau$ is the blow-up of $p_0$.
\end{enumerate}
\end{theoremA}

These two classifications (Theorems~\ref{th:P1bundlesFa} and~\ref{th:P1bundlesP2}) identify the families that must
be considered before imposing relative maximality.
The next step is to compare their neutral automorphism groups
under equivariant birational modifications.

The automorphism groups appearing above are described
throughout the article.
For decomposable $\P^1$-bundles, see
Lemma~\ref{lem:aut-group-Fabc} over Hirzebruch surfaces
and Lemma~\ref{lemm:surjectiveDecP2} over $\P^2$.
For Umemura $\P^1$-bundles, Proposition~\ref{ActionSL2GL2Umemura}
and Remarks~\ref{rk: GL2-action on Umemura}
and~\ref{rk aut vert Umemura bundles} describe the lifted
actions and the subgroup acting trivially on the base.
The neutral automorphism groups of $\L^{s,b}$ and $\SS^{b,c}$
are determined in Propositions~\ref{prop: LsbAut}
and~\ref{prop:aut-orbits-Sbc}, respectively.

\smallskip 

Combining these descriptions with the study of equivariant
links yields our global classification.
The following theorem gives both the existence of a relatively
maximal model for every $\P^1$-bundle and the complete list of $\P^1$-bundles
that already have this property.

	\begin{theoremA}[{see Section~\ref{proof: classification rel max} for the proof}]\label{th: classification rel max}
		Assume that $p\geq3$.
		Let $\pi\colon X\to S$ be a $\P^1$-bundle over a smooth projective
		rational surface. Then there exists an $\Autz(X)$-equivariant square
		birational map
		\[
		(X,\pi)\dashrightarrow(X',\pi')
		\]
		such that $\Autz(X')$ is relatively maximal.
		Moreover, $\Autz(X)$ is relatively maximal if and only if
		$(X,\pi)$ is square-isomorphic to one of the following:
		\begin{enumerate}
			\renewcommand{\theenumi}{\alph{enumi}}
			\renewcommand{\labelenumi}{$(\theenumi)$}
			
			\item\label{item:relmax-Fabc}
			A decomposable $\P^1$-bundle $\FF_a^{b,c}\to\F_a$
			$($Definition~$\ref{def Fabc})$, where
			$a,b\geq0$, $a\neq1$, $c\in\Z$, and $c\leq0$ if $b=0$, with
			\[
			a=0,\qquad\text{or}\qquad b=c=0,
			\qquad\text{or}\qquad -a<c<ab.
			\]
			
			\item\label{item:relmax-Pb}
			A decomposable $\P^1$-bundle $\PP_b\to\P^2$
			$($Definition~$\ref{def Pb})$, with $b\geq0$.
			
			\item\label{item:relmax-Umemura}
			An Umemura $\P^1$-bundle
			$\U_{a,\omega}^{b,c}=Z_a^{b,c,P}\to\F_a$
			$($Definition~$\ref{def Umemura bundle})$,
			with $a,b\geq1$ and $c\geq2$, where
			\[
			P=\sum_{i\in I}\lambda_i y_0^iy_1^{b-i}z^{ai+p^{s_i}},
			\qquad ai+2p^{s_i}=c,
			\qquad \lambda_i\in\k^*.
			\]
			Writing $i_0=\min I$, $m=b-i_0$ and $r=p^{s_{i_0}}$,
			we require $m>0$ and
			\[
			\begin{cases}
				r\mid m\text{ and }m\geq2r,&a=1,\\
				a p^{\nu_p(m)}\geq r,&a\geq2.
			\end{cases}
			\]
			
			\item\label{item:relmax-Schwarzenberger}
			A Schwarzenberger $\P^1$-bundle $\SS^{b,c}\to\P^2$,
			with parameters as in Definition~$\ref{def:Schwarz bc}$,
			such that
			\[
			\frac{1}{2}(b+c)\notin\{p^e,2p^e\mid e\geq0\}.
			\]
			
			\item\label{item:relmax-Ls}
			The $\P^1$-bundle
			$\L^s=\P((F^s)^*T_{\P^2})\to\P^2$
			$($Definition~$\ref{def: P1 bundles Lsb})$, with $s\geq0$.
			
			\item\label{item:relmax-V}
			A $\P^1$-bundle $\V_{\omega}^{b,2r}\to\P^2$
			obtained by the contraction of
			Proposition~$\ref{prop:UmeF1Contraction}$ from
			$\U_{1,\omega}^{b,2r}\to\F_1$, where
			\[
			r=p^s,\qquad s\geq0,\qquad r\mid b,\qquad b\geq2r,
			\]
			and $\omega\in\Omega_1^{b,2r}$ has nonzero component
			of index $0$.
		\end{enumerate}
	\end{theoremA}

Relative maximality does not exclude nontrivial equivariant
birational maps that preserve the full neutral automorphism group.
Such maps may connect distinct $\P^1$-bundle structures; even for
a stiff $\P^1$-bundle, a birational self-map need not be biregular.
The next theorem determines which $\P^1$-bundles in the classification
are stiff or superstiff and describes all equivariant square
birational maps between them.

\begin{theoremA}[{see Section~\ref{proof: equivariant links} for the proof}]
	\label{th: equivariant links}
	Assume that $p\geq3$.
	Let $(X,\pi)$ be a $\P^1$-bundle in
	Theorem~$\ref{th: classification rel max}$. Then:
	\begin{itemize}
		\item $(X,\pi)$ is superstiff if and only if it belongs to
		case~\eqref{item:relmax-Pb} or~\eqref{item:relmax-Ls}, or to
		case~\eqref{item:relmax-Fabc} with $a=0$ or $b=c=0$;
		
		\item $(X,\pi)$ is stiff if and only if it is superstiff, or
		it belongs to case~\eqref{item:relmax-Schwarzenberger} with
		$p\nmid\frac{b+c}{2}$.
	\end{itemize}
	
	Every $\Autz(X)$-equivariant square birational map to another
	$\P^1$-bundle over a smooth projective rational surface is a
	composition of square isomorphisms and of the maps below and
	their inverses. Each of these maps conjugates the full neutral
	automorphism groups.
	\begin{enumerate}[$(L1)$]   
		\item
		The elementary links between decomposable $\P^1$-bundles over $\F_a$ described in
		Proposition~$\ref{prop:MaxDecBundlesFa}\ref{FFabcOnlyMax}.$
		
		\item
		The elementary links between Umemura $\P^1$-bundles described in
		Proposition~$\ref{prop: description of the equ links Umemura}$,
		items~$\ref{Blowl00egal}$ and~$\ref{Blowl10egal}.$
		
		\item
		The elementary links between Schwarzenberger $\P^1$-bundles above
		the invariant conic $\Gamma$, whose centers are the invariant
		sections described in
		Corollary~$\ref{cor: description of the orbits for Sbc}$,
		as in Lemma~$\ref{Lemm:DecElLinks}$.
		
		\item
		The birational morphisms from Umemura $\P^1$-bundles in
		case~\eqref{item:relmax-Umemura} to bundles in
		case~\eqref{item:relmax-V}, given by
		Proposition~$\ref{prop:UmeF1Contraction}$, with equality of neutral
		automorphism groups as in
		Lemma~$\ref{Lem:UmeF1P2-positive-char}$.
	\end{enumerate}
\end{theoremA}

\begin{remark}
	Relative maximality only tests equivariant square birational maps between
	$\P^1$-bundles, so it does not imply maximality in $\Bir(\P^3)$. Let us give two examples.
	\begin{itemize}
		\item In every characteristic, the decomposable $\P^1$-bundle $\PP_1\to\P^2$ is relatively maximal
		(indeed, superstiff; see Lemma~\ref{Lem:DecP2NoCurve}\ref{AutPPbSuperstiff}). However, if $\PP_1$ is viewed as the blow-up
		of $\P^3$ at $[0:0:0:1]$, the blow-down morphism
		$\varphi\colon\PP_1\to\P^3$ satisfies
		\[
		\varphi\Autz(\PP_1)\varphi^{-1}
		\subsetneq\Aut(\P^3)\simeq\PGL_4.
		\]
		
		\item Similarly, if $p\geq5$, then $\Autz(\SS^{2,4})$ is relatively maximal
		by Proposition~\ref{prop: relative max of Sbc}, yet there is a
		$\PGL_2$-equivariant birational map
		\[
		\varphi\colon\SS^{2,4}\dasharrow\P(V_3)\simeq\P^3,
		\]
		where $\PGL_2$ acts naturally on binary cubic forms, such that
		\[
		\varphi\Autz(\SS^{2,4})\varphi^{-1}
		\subsetneq\Aut(\P(V_3))\simeq\PGL_4.
		\]
	\end{itemize}
	In both examples, the target is the Mori fiber space $\P^3\to\mathrm{pt}$,
	which is not among the targets considered in relative maximality.
\end{remark}

\subsection{Brief overview of the proof}
Section~\ref{sec:preliminaries} recalls Blanchard's lemma and the
descent lemma, develops the representation theory of $\SL_2$ and
equivariant homomorphisms between its modules, and reviews Hirzebruch
surfaces and their automorphism groups. The descent lemma reduces the
study of a $\P^1$-bundle over a smooth projective rational surface to
bundles over $\P^2$ or a Hirzebruch surface $\F_a$.

Sections~\ref{sec:P1-bundles-over-Fa} and~\ref{section: P1 bundles over P2}
classify the $\P^1$-bundles over these surfaces. Over $\F_a$, we study
first decomposable $\P^1$-bundles and then indecomposable $\P^1$-bundles with no jumping fibers, including the Umemura $\P^1$-bundles. Over $\P^2$, we use the induced action of $\Autz(X)$ on the
base: if its image fixes a point, blowing up the corresponding fiber
reduces the problem to $\F_1$; otherwise, the possible actions lead to
the decomposable $\P^1$-bundles, the $\P^1$-bundles $\L^{s,b}$, and the generalized
Schwarzenberger $\P^1$-bundles.

Finally, Section~\ref{sec:max and non-max results} determines, case by case, the groups $\Autz(X)$ and their orbits on $X$. Using this orbit information and Lemma~\ref{Lemm:DecElLinks}, we identify the possible equivariant birational links between the $\P^1$-bundles. The final subsections then classify the relatively maximal $\P^1$-bundles and determine which of them are stiff or superstiff.

\subsection{Comparison with the classification in characteristic zero}
Theorems~\ref{th: classification rel max}
and~\ref{th: equivariant links} are the positive-characteristic
counterparts of \cite[Theorems~A and~B]{BFT}, respectively.
Our notion of \emph{relative maximality} is called
\emph{maximality} there: both concern equivariant square
birational maps between $\P^1$-bundles over smooth projective
rational surfaces. The decomposable cases have the same
numerical criteria and stiffness properties in both settings.

For Umemura $\P^1$-bundles in characteristic zero, the defining
polynomial has a single nonzero component, with $c=ai_0+2$.
Writing $m=b-i_0$, relative maximality is equivalent to $m>0$ for
$a\geq2$ and $m\geq2$ for $a=1$, or equivalently
\[
c-ab<2\quad\text{if }a\geq2,
\qquad
c-ab<1\quad\text{if }a=1.
\]
In positive characteristic, several components may occur,
with exponents involving powers of $p$.
Setting $r=p^{s_{i_0}}$, relative maximality is characterized
by Corollary~\ref{Cor:MaxUmeBundlesFa}:
\[
\begin{cases}
	r\mid m\text{ and }
	\bigl(m\geq2r\text{ or }(p=2,\ m=r,\ P_b\neq0)\bigr),
	&a=1,\\
	m>0\text{ and }a p^{\nu_p(m)}\geq r,
	&a\geq2.
\end{cases}
\]
For $r=1$, these conditions have the same numerical form
as in characteristic zero. The contractions $\V_{\omega}^{b,2r}$ are the analogues of the $\P^1$-bundles $\V_b$, which correspond to $r=1$ and
a single-component Umemura polynomial.

The classical Schwarzenberger $\P^1$-bundles correspond to
$\SS^{b,b+2}=\SS_b$, and the projectivized tangent bundle
$\L^0$ is $\SS_1=\P(T_{\P^2})$. 
Positive characteristic also gives
the superstiff Frobenius pullbacks
$\L^s=\P((F^s)^*T_{\P^2})$ for $s>0$ and a broader
Schwarzenberger family $\SS^{b,c}$. 
Writing $M=(b+c)/2$, the latter
is relatively maximal precisely when
$M\notin\{p^e,2p^e\mid e\geq0\}$,
and, among these relatively maximal $\P^1$-bundles, stiff
precisely when $p\nmid M$. Thus relative maximality
no longer implies stiffness for Schwarzenberger $\P^1$-bundles,
whereas all classical Schwarzenberger $\P^1$-bundles are stiff
in characteristic zero.
This accounts for the
six families in our classification, compared with
the five families of \cite[Theorem~A]{BFT}.

A further distinction concerns moduli. In characteristic
zero, all five families are indexed by at most three
integers. In positive characteristic, the Umemura $\P^1$-bundles over $\F_a$ 
and their contractions over $\P^2$ also depend on $\omega$ and can
form positive-dimensional families of pairwise
non-square-isomorphic relatively maximal $\P^1$-bundles with
fixed numerical invariants. Moreover, the parameter spaces
$\Omega_a^{b,c}$ can have arbitrarily large dimension
(see Example~\ref{example: Omegaabc not bounded}).

\subsection{Toward the classification of maximal connected algebraic subgroups of \texorpdfstring{$\Bir(\P^3)$}{Bir(P3)} in positive characteristic}\label{sec: classification max conn subgroups of Cr3}
Beyond its independent interest, the classification of relatively maximal
$\P^1$-bundles over smooth rational surfaces obtained here is motivated by
the broader goal of classifying maximal connected algebraic subgroups of
$\Bir(\P^3)$ in positive characteristic. 

A natural next step will be to adapt the methods and results
of~\cite{BFT-MFS} to positive characteristic and combine them
with the present classification. One important difficulty is that,
although resolution of
singularities is known for threefolds in positive characteristic, no
equivariant resolution theorem for general actions of
connected algebraic groups is currently available. Thus, starting with an
arbitrary connected algebraic subgroup of $\Bir(\P^3)$, we cannot in
general reduce its study to an action on a smooth projective threefold
and then run an MMP. 
Moreover, there is presently no general equivariant Sarkisov program for threefolds in positive characteristic. We therefore lack a general method for controlling equivariant birational maps between Mori fiber spaces and,
in particular, for establishing the maximality of their connected
automorphism groups in $\Bir(\P^3)$.

The positive-dimensional families of relatively maximal Umemura
$\P^1$-bundles over the Hirzebruch surfaces $\F_a$, together with the
corresponding families over $\P^2$ obtained by contraction, are expected
to yield new continuous families of conjugacy classes of maximal
connected algebraic subgroups of $\Bir(\P^3)$ in characteristic
$p\geq3$. This would be a striking new phenomenon: in characteristic
zero, the only continuous family of maximal connected algebraic
subgroups in the classification of~\cite{BFT-MFS} is the one arising
from the Umemura quadric fibrations studied in~\cite[\S~11]{Ume88} and~\cite[\S~4.4]{BFT-MFS}.
These fibrations exist in arbitrary characteristic, and the neutral
automorphism groups in this continuous family are isomorphic to
$\PGL_2$.

\subsection*{Use of generative AI}
Work on this project really began in spring 2021, during a three-month visit of the
second author to the University of Basel, and continued intermittently
over the following years. Most of the article was completed by spring
2025. 
One remaining issue was to understand when a connected algebraic
subgroup of $\Aut(\F_a)$ acting transitively on the base of the
ruling contains a conjugate of
$\Im(\SL_2\to\GL_2/\mu_a)$, and how the answer changes when
the subgroup arises from the automorphism group of a
$\P^1$-bundle.
Discussions with OpenAI's GPT models in recent
weeks helped us to resolve this issue. In particular, they helped to extend to arbitrary prime characteristic a counterexample previously known to the authors in characteristic $2$ (Proposition~\ref{prop: UniformCounterexampleNew}),
and contributed to the proofs of
Lemma~\ref{lem:H-maps-onto-PGL2-odd-a} and
Corollary~\ref{cor:H-maps-onto-PGL2-even-a}.
The authors take full responsibility for the contents of the article.

\subsection*{Acknowledgments} 
The first author acknowledges support of the grant ``Birational transformations of higher dimensional varieties'' $200020\_214999$ of the Swiss National Foundation.
The second author acknowledges the support of the CDP C2EMPI. He also thanks the French State, through the France-2030 programme, the University of Lille, the Initiative of Excellence of the University of Lille, and the European Metropolis of Lille for their funding and support of the R-CDP-24-004-C2EMPI project.

Both authors thank the University of Basel for its hospitality and the Swiss National Science Foundation grant ``Geometrically ruled surfaces'' $200020\_192217$ for its financial support during the second author's three-month visit to the first author in Basel in spring 2021.

\section{Preliminaries} \label{sec:preliminaries}

\subsection{Notation}
In this article, we work over a fixed algebraically closed field $\k$
of characteristic $p>0$. Any restriction on the characteristic is
stated explicitly.
For each nonzero integer $n$, we denote by $\nu_p(n)$
its $p$-adic valuation.

A \emph{variety} is an integral separated scheme of finite type
over $\k$; in particular, varieties are always irreducible. An \emph{algebraic group} is a smooth group scheme of finite type
over $\k$.
By an \emph{algebraic subgroup}, we mean a smooth closed subgroup
scheme.
We also consider subgroup schemes that are not necessarily smooth,
notably kernels and stabilizers; these are always understood
scheme-theoretically.
The identity component of an algebraic group $G$, denoted by
$G^\circ$, is the connected component containing the identity.
It is a normal subgroup scheme of $G$, and the quotient
$G/G^\circ$ is a finite étale group scheme.
 
 For a projective variety $X$, the automorphism group scheme is
 locally of finite type by \cite{MO67}.
 We always endow it with its reduced group scheme structure,
 and we write $\Autz(X)$ for its identity component, also endowed
 with the reduced structure.
 Images of homomorphisms of algebraic groups are understood
 scheme-theoretically.
 Since $\k$ is perfect, the scheme-theoretic image of a homomorphism
 of algebraic groups is again a smooth closed subgroup scheme;
 it is connected if the source is connected.
 In particular, any image of $\Autz(X)$ considered below is a
connected algebraic subgroup.
 
  Throughout this article, a $\P^n$-bundle is always assumed to be
  locally trivial in the Zariski topology.
  Unless otherwise specified, all varieties, rational maps, and algebraic groups are defined over a fixed algebraically closed field $\k$.

\subsection{Blanchard's lemma and descent lemma}
Here we recall two classical results, valid over any algebraically closed field, that will be used throughout this article.

\begin{proposition}\label{prop:blanchard-lemma} \emph{(Blanchard's lemma, \cite[Proposition~4.2.1]{BSU13})}\\ 
Let $f\colon X \to Y$ be a proper morphism between algebraic varieties such that $f_*(\O_X)=\O_Y$. If a connected algebraic group $G$ acts on $X$, then there exists a unique action of $G$ on $Y$ such that $f$ is $G$-equivariant.
\end{proposition}

\begin{lemma}\label{lem:descent-lemma} \emph{(Descent lemma, \cite[Lemma~2.3.2]{BFT})}\\ 
Let $\eta\colon\hat{S}\to S$ be a birational morphism between two smooth projective surfaces. Let $U\subseteq S$ and $\hat{U}\subseteq \hat{S}$ be two maximal open subsets such that $\eta$ induces an isomorphism $\hat U\iso U$, and $\Omega=S\setminus U$ is finite, and let $\hat\pi\colon \hat{X}\to \hat{S}$ be a $\P^1$-bundle.

Then, there exists a $\P^1$-bundle $\pi\colon X\to S$, and a birational map $\psi\colon \hat{X}\dasharrow X$, such that $\eta\hat\pi=\pi\psi$ $(\psi$ is a square birational map over $\eta)$ and such that $\psi$ induces an isomorphism $\hat\pi^{-1}(\hat{U})\iso \pi^{-1}(U)$. Moreover, $\psi$ is unique, up to composition by an isomorphism of $\P^1$-bundles at the target, and $\psi$ is $\Autz(\hat{X})$-equivariant, which means that $\psi \Autz(\hat{X}) \psi^{-1}$ is a subgroup of $\Autz(X)$.
\end{lemma}

\subsection{Equivariant homomorphisms between \texorpdfstring{$\SL_2$}{SL2}-modules}
We begin by describing explicitly the equivariant linear maps
between the classical $\SL_2$-modules $V_m$ and $V_n$ (Definition~\ref{def:Vn}).
In positive characteristic, such maps can exist even when $m\neq n$;
Proposition~\ref{prop: rep SL2 car p} gives a precise criterion
and formula, and Corollary~\ref{cor:case n=0} identifies the
invariant vectors that will be used in the construction of
Umemura $\P^1$-bundles.
These results are classical (see, for example,
\cite[\S II.2]{Jan03}), but we include direct proofs.

\begin{notation}
For each $u\in \Z$, $l\in \N$, we denote by $\binom{u}{l}$ the integer $\frac{\prod_{i=u-l+1}^u i}{l!}$, which is equal to $1$ if $l=0$ and is equal to $0$ if $l\ge 1$ and $u\in \{0,\dots,l-1\}$. 
\end{notation}
\begin{lemma}\label{Lemm:BinomialsEqualModp}
	Let $p$ be a prime number, let $s\geq1$, and let
	$u,t\in\Z$ and $0\leq l<p^s$ be integers.
	If $t\equiv0\pmod{p^s}$, then
	\[
	\binom{u+t}{l}\equiv\binom{u}{l}\pmod p.
	\]
\end{lemma}

\begin{proof}
	Set $q=p^s$ and write $t=aq$ with $a\in\Z$.
	In $\mathbb F_p[[T]]$, we have
	\[
	(1+T)^t=(1+T^q)^a\equiv1\pmod{T^q}.
	\]
	Here negative powers are well defined because the series involved
	have constant term $1$.
	Consequently,
	\[
	(1+T)^{u+t}\equiv(1+T)^u\pmod{T^q}.
	\]
	Comparing the coefficients of $T^l$ gives the result.
\end{proof}

\begin{remark}
	For $u,t\geq0$, the preceding congruence also follows
	from Lucas's theorem.
\end{remark}

\begin{lemma}\label{lem:simplebin}
\begin{enumerate}
    \item\label{sb1}
    For all $u \in \Z$ and $l \in \N$, we have
    \[
        \binom{u}{l} = (-1)^l \binom{l-u-1}{l}.
    \]
    
    \item\label{sb2}
    Let $p$ be a prime number. For all $u,s,l,v \in \Z$ with $1 \le s$, $0 \le l < p^s$, and
    $u + v + 1 - l \equiv 0 \pmod{p^s}$, we have
    \[
        \binom{u}{l} \equiv (-1)^l \binom{v}{l} \pmod{p}.
    \]
\end{enumerate}
\end{lemma}

\begin{proof}
\ref{sb1}: Follows from
\[
    l!\binom{u}{l} = \prod_{i=u-l+1}^{u} i
    = \prod_{i=-u}^{l-u-1} (-i)
    = (-1)^l \prod_{i=-u}^{l-u-1} i
    = (-1)^l \, l! \, \binom{l-u-1}{l}.
\]

\ref{sb2}:
By \ref{sb1}, we have
\[
    \binom{u}{l} = (-1)^l \binom{l-u-1}{l}.
\]
Moreover, since $v - (l-u-1) \equiv 0 \pmod{p^s}$, Lemma~\ref{Lemm:BinomialsEqualModp} implies
\[
    \binom{l-u-1}{l} \equiv \binom{v}{l} \pmod{p}.
\]
Combining these, we obtain
\[
    \binom{u}{l} \equiv (-1)^l \binom{v}{l} \pmod{p}.
\]
\end{proof}

\begin{lemma}\label{Lemm:ConsecutiveZeroModp}
Let $p$ be a prime number and let $l,s\ge 1$ be integers such that $p^{s-1}\le l<p^s$. For each integer $k\ge 1$, the following conditions are equivalent:
\begin{enumerate}
\item\label{keqiv0ps}
$k\equiv 0 \pmod{p^s}$.
\item\label{allzeroModp}
$\binom{i}{l}\equiv 0 \pmod{p}$ for each $i\in \{k,\dots,k+l-1\}$.\end{enumerate}
\end{lemma}
\begin{proof}
$\ref{keqiv0ps}\Rightarrow\ref{allzeroModp}$: If $k\equiv 0\pmod{p^s}$, then Lemma~\ref{Lemm:BinomialsEqualModp} yields $\binom{i}{l}\equiv \binom{i-k}{l}\pmod{p}$, for each integer $i\in \Z$. This implies that $\binom{i}{l}\equiv 0 \pmod{p}$ for each $i\in \{k,\dots,k+l-1\}$ since $i-k\in \{0,\dots,l-1\}$.

$\ref{allzeroModp}\Rightarrow \ref{keqiv0ps}$: Since $\binom{i}{l}\equiv\binom{i-ap^s}{l} \pmod{p}$ for each $a,i\in \Z$ (Lemma~\ref{Lemm:BinomialsEqualModp}), we can replace $k$ with $k-ap^s$, and assume then that $k\in \{0,\dots,p^s-1\}$. If $l\ge k$, then the fact that $\binom{l}{l}=1\not\equiv 0\pmod{p}$ implies that $l\not\in \{k,\dots,k+l-1\}$ (with \ref{allzeroModp}) and then implies that $l\ge l+k$, which yields $k=0$, as desired. It remains to see that $l<k$ yields a contradiction. Using the fact that $p^{s-1}\le l<k< p^{s}$, we can write
\[\begin{array}{rrrrrrrrr}
k&=& k_{s-1}p^{s-1}&+&\dots&+&k_1p&+&k_0;\ \text{and}\\
l&=& l_{s-1}p^{s-1}&+&\dots &+&l_1p&+&l_0\end{array}\]
where $k_i,l_i\in \{0,\dots,p-1\}$ for each $i=0,\dots,s-1$, and where $k_{s-1}\ge l_{s-1}\ge 1$. 
Set $q=(k_{s-1}+1)p^{s-1}-1$. Then
\[
0\leq q-k\leq p^{s-1}-1\leq l-1,
\]
so $q\in\{k,\dots,k+l-1\}$.
The leading base-$p$ digit of $q$ is
$k_{s-1}\geq l_{s-1}$, and all its lower digits are $p-1$.
Thus each base-$p$ digit of $q$ is at least the corresponding
digit of $l$. Lucas's theorem gives
$\binom{q}{l}\not\equiv0\pmod p$, contradicting
condition~\ref{allzeroModp}.
\end{proof}

\begin{lemma}\label{EqualityQyz}
	Let $K$ be a field of characteristic $p>0$, and let
	$b,l\geq0$ be integers.
	Assume that either $l=0$, or there exists $s\geq1$
	such that $l<p^s$ and $b+l+1\equiv0\pmod{p^s}$.
	Then
	\[
	\left(\sum_{i=0}^b\binom{l+i}{l}y^{b-i}z^i\right)
	(y-z)^{l+1}
	=
	y^{b+l+1}-z^{b+l+1}.
	\]
\end{lemma}

\begin{proof}
	For $l=0$, the assertion follows from the classical identity
		\[
		(y-z)\sum_{i=0}^b y^{b-i}z^i
		=
		y^{b+1}-z^{b+1}.
		\]
		Assume henceforth that $l>0$.
	Set $N=b+l+1$ and $q=p^s$. Since $q\mid N$, we have
	\[
	1-T^N=(1-T^{N/q})^q
	\]
	in $K[T]$. Thus $(1-T)^{l+1}$ divides $1-T^N$,
	because $l+1\leq q$.
	The quotient
	\[
	A(T):=\frac{1-T^N}{(1-T)^{l+1}}
	\]
	is a polynomial of degree $b$.
	On the other hand, in $K[[T]]$,
	\[
	(1-T)^{-l-1}
	=\sum_{i\geq0}\binom{l+i}{l}T^i.
	\]
	Since $N>b$, comparison of the coefficients of degrees at most
	$b$ gives
	\[
	A(T)=\sum_{i=0}^b\binom{l+i}{l}T^i.
	\]
	Homogenizing the identity
	$A(T)(1-T)^{l+1}=1-T^N$ yields the result.
\end{proof}

\begin{definition} \label{def:Vn}
	For each integer $m\geq0$, set $V_m:=\k[x,y]_m$.
	We equip this $(m+1)$-dimensional vector space with the
	following $\GL_2$-action:
 \[\left(\begin{bmatrix}
\alpha & \beta \\ \gamma & \delta
\end{bmatrix},P(x,y)\right)\mapsto P(\alpha x+\gamma y,\beta x+\delta y).\]
Denoting by $V=V_1$ the standard representation of $\GL_2$, we have $V_m=S^m(V)$.
\end{definition}

\begin{remark}\label{rk: def of dual Vm}
Let $m \geq 0$. There exists a unique linear action of $\GL_2$ on $\k[z]_{\le m}$,
which will play an important role later in this article, such that the map
\[
V \longrightarrow \k[z]_{\le m}, \qquad 
(u,v) \longmapsto \sum_{i=0}^m u^i v^{m-i}\, z^i
\]
is $\GL_2$-equivariant, where $V = V_1$ denotes the standard representation of $\GL_2$.

With this action, we obtain an isomorphism of $\GL_2$-modules
\[
\k[z]_{\le m} \simeq S^m(V^*)^*.
\]
In particular, since $V \simeq V^*$ as $\SL_2$-modules, it follows that
\[
\k[z]_{\le m} \simeq S^m(V)^* = V_m^* \qquad\text{as $\SL_2$-modules.}
\] 
\end{remark}

\begin{remark}\label{ExplicitActionSL2km}
Let us compute explicitly the action of Remark~\ref{rk: def of dual Vm},
which corresponds to a group homomorphism
\[
\begin{array}{rcl}
\GL_2 & \longrightarrow & \GL_{m+1}\\
\begin{bsmallmatrix}a & b\\ c & d\end{bsmallmatrix}
& \longmapsto &
R=\begin{bsmallmatrix}
r_{00} & \cdots & r_{0m}\\
\vdots & \ddots & \vdots\\
r_{m0} & \cdots & r_{mm}
\end{bsmallmatrix}
\end{array}, \quad \text{ where } \quad  
\begin{bsmallmatrix}a & b\\ c & d\end{bsmallmatrix}\cdot z^j
=\sum_{i=0}^m r_{ij}z^i.
\]
To compute the matrix $R$, we use the equivariant polynomial map
\[
F_m\colon V\longrightarrow\k[z]_{\leq m},
\qquad
(u,v)\longmapsto\sum_{j=0}^m u^jv^{m-j}z^j
\]
from Remark~\ref{rk: def of dual Vm}.
Its $\GL_2$-equivariance gives
\[
\sum_{i=0}^m(au+bv)^i(cu+dv)^{m-i}z^i
=
\begin{bsmallmatrix}a & b\\ c & d\end{bsmallmatrix}
\cdot
\left(\sum_{j=0}^m u^jv^{m-j}z^j\right)\\[2mm]
=
\sum_{i,j=0}^m u^jv^{m-j}r_{ij}z^i.
\]
Hence, for each $i\in\{0,\ldots,m\}$,
\[
\sum_{j=0}^m r_{ij}u^jv^{m-j}
=(au+bv)^i(cu+dv)^{m-i}.
\]
Comparing the coefficients of $u^jv^{m-j}$ gives
\[
r_{ij}
=
\sum_{e=\max(0,i+j-m)}^{\min(i,j)}
\binom{i}{e}\binom{m-i}{j-e}
a^e b^{i-e}c^{j-e}d^{m+e-i-j}.
\]

In particular, for $1\leq m\leq4$, the corresponding matrices $R$ are
\[
\begin{bsmallmatrix}
d & c\\
b & a
\end{bsmallmatrix},
\qquad
\begin{bsmallmatrix}
d^2 & 2cd & c^2\\
bd & ad+bc & ac\\
b^2 & 2ab & a^2
\end{bsmallmatrix}, \qquad
\begin{bsmallmatrix}
d^3 & 3cd^2 & 3c^2d & c^3\\
bd^2 & 2bcd+ad^2 & bc^2+2acd & ac^2\\
b^2d & b^2c+2abd & 2abc+a^2d & a^2c\\
b^3 & 3ab^2 & 3a^2b & a^3
\end{bsmallmatrix},
\]
and
\[
\begin{bsmallmatrix}
d^4 & 4cd^3 & 6c^2d^2 & 4c^3d & c^4\\
bd^3 & 3bcd^2+ad^3 & 3bc^2d+3acd^2 & bc^3+3ac^2d & ac^3\\
b^2d^2 & 2b^2cd+2abd^2 & b^2c^2+4abcd+a^2d^2
& 2abc^2+2a^2cd & a^2c^2\\
b^3d & b^3c+3ab^2d & 3ab^2c+3a^2bd
& 3a^2bc+a^3d & a^3c\\
b^4 & 4ab^3 & 6a^2b^2 & 4a^3b & a^4
\end{bsmallmatrix}.
\]

For later use, let us also record the action of some particular
elements. For each $m$ and each $i\in\{0,\ldots,m\}$, we have
\[
\begin{bsmallmatrix}a & 0\\ 0 & d\end{bsmallmatrix}\!\cdot z^i
   =a^id^{m-i}z^i,
\qquad
\begin{bsmallmatrix}0 & -1\\ 1 & 0\end{bsmallmatrix}\!\cdot z^i
   =(-z)^{m-i},
\qquad
\begin{bsmallmatrix}1 & 1\\ 0 & 1\end{bsmallmatrix}\!\cdot z^i
   =\sum_{j=i}^m\binom{j}{i}z^j.
\]
\end{remark}

\begin{proposition}\label{prop: rep SL2 car p}
	For all integers $m,n\geq0$, the $\k$-vector space
	$\Hom^{\SL_2}(V_m,V_n)$ has dimension at most $1$.
	More precisely:
	\begin{enumerate}
		\item
		If $m=n$, then $\Hom^{\SL_2}(V_m,V_n)$ is
		one-dimensional and generated by the identity.
		
		\item\label{mnp}
		If $m=n+2l$ with $l>0$, and if
		\[
		n+l+1\equiv0\pmod{p^s},
		\qquad
		p^{s-1}\leq l<p^s,
		\]
		then $\Hom^{\SL_2}(V_m,V_n)$ is one-dimensional
		and generated by the linear map
		$\chi\colon V_m\longrightarrow V_n$ defined by
		\[
		\chi(x^iy^{m-i})=
		\begin{cases}
			\binom{i}{l}x^{i-l}y^{n-i+l},
			& l\leq i\leq n+l,\\
			0, & \text{otherwise}.
		\end{cases}
		\]
		
		\item
		In all other cases, $\Hom^{\SL_2}(V_m,V_n)=\{0\}$.
	\end{enumerate}
\end{proposition}

\begin{proof}
	If $m=n$, the identity is a nonzero equivariant map.
	
	Suppose that the numerical conditions in~\ref{mnp} hold.
	Lemma~\ref{Lemm:ConsecutiveZeroModp} gives
	\[
	\binom{i}{l}=0\quad\text{in }\k
	\qquad\text{for }n+l+1\leq i\leq m.
	\]
	Moreover, $\binom{i}{l}=0$ for $0\leq i<l$.
	The map $\chi$ is nonzero, since
	$\chi(x^ly^{m-l})=y^n$.
	
	We check its equivariance on the elements
	\[
	d_\alpha=
	\begin{bmatrix}\alpha&0\\0&\alpha^{-1}\end{bmatrix}
	\quad(\alpha\in\k^\times),
	\qquad
	\sigma=\begin{bmatrix}0&-1\\1&0\end{bmatrix},
	\qquad
	\tau=\begin{bmatrix}1&0\\1&1\end{bmatrix}.
	\]
	These generate $\SL_2(\k)$: diagonal conjugates of $\tau$
	give all lower unipotent elements, since $\k$ is algebraically
	closed, and conjugation by $\sigma$ gives all upper
	unipotent elements.
	
	The map $\chi$ commutes with $d_\alpha$ because it preserves
	torus weights:
	\[
	2(i-l)-n=2i-m.
	\]
	For $l\leq i\leq n+l$, we have
	\[
	\begin{aligned}
		\chi\bigl(\sigma(x^iy^{m-i})\bigr)
		&=(-1)^{m-i}\binom{m-i}{l}x^{m-i-l}y^{i-l},\\
		\sigma\bigl(\chi(x^iy^{m-i})\bigr)
		&=(-1)^{m-i-l}\binom{i}{l}x^{m-i-l}y^{i-l}.
	\end{aligned}
	\]
	These expressions agree by
	Lemma~\ref{lem:simplebin}\,\ref{sb2}, since
	$m+1-l=n+l+1\equiv0\pmod{p^s}$.
	For the remaining indices, both expressions vanish.
	Thus $\chi$ commutes with $\sigma$.
	
	For $\tau$, we compute
	\[
	\chi\bigl(\tau(x^iy^{m-i})\bigr)
	=
	\sum_{j=l}^{\min(i,n+l)}
	\binom{i}{j}\binom{j}{l}
	x^{j-l}y^{n-j+l}.
	\]
	If $i<l$, this sum is empty and both sides of the desired
	equivariance identity vanish.
	For $i\geq l$, we use the integer identity
	\[
	\binom{i}{j}\binom{j}{l}
	=
	\binom{i}{l}\binom{i-l}{j-l}
	\qquad(l\leq j\leq i).
	\]
	If $i>n+l$, then $\binom{i}{l}=0$ in $\k$, so both sides
	again vanish.
	If $l\leq i\leq n+l$, the binomial theorem gives
	\[
	\chi\bigl(\tau(x^iy^{m-i})\bigr)
	=
	\binom{i}{l}(x+y)^{i-l}y^{n-i+l}
	=
	\tau\bigl(\chi(x^iy^{m-i})\bigr).
	\]
	This proves that $\chi$ is $\SL_2$-equivariant.
	
	It remains to classify all nonzero equivariant maps.
	Let $\varphi\colon V_m\longrightarrow V_n$ be such a map.
	The weights of $x^iy^{m-i}$ and $x^jy^{n-j}$ are
	$2i-m$ and $2j-n$, respectively.
	Since $\varphi$ preserves weights, we must have
	$m\equiv n\pmod2$. Set $l=(m-n)/2$.
	For each $0\leq i\leq m$, there is a scalar $\lambda_i$
	such that
	\[
	\varphi(x^iy^{m-i})
	=
	\lambda_i x^{i-l}y^{m-i-l}
	\]
	whenever the exponents are nonnegative.
	If either exponent is negative, the image is zero,
	and we set $\lambda_i=0$.
	
	Suppose first that $l<0$.
	Every monomial occurring in
	$\varphi\bigl(\tau(x^iy^{m-i})\bigr)$
	has positive $x$-degree, whereas the coefficient of $y^n$
	in $\tau\bigl(\varphi(x^iy^{m-i})\bigr)$ is $\lambda_i$.
	Equivariance therefore gives $\lambda_i=0$ for every $i$,
	a contradiction. Hence $l\geq0$.
	
	Comparing the coefficients of $y^n$ in the same identity
	now gives
	\[
	\lambda_i=\lambda_l\binom{i}{l}
	\qquad(0\leq i\leq m).
	\]
	In particular, $\lambda_l\neq0$.
	If $l=0$, then $m=n$ and $\varphi$ is a scalar multiple
	of the identity.
	
	If $l>0$, then $\lambda_i=0$ for $n+l+1\leq i\leq m$,
	so
	\[
	\binom{i}{l}=0\quad\text{in }\k
	\qquad(n+l+1\leq i\leq m).
	\]
	Let $s\geq1$ be the integer such that
	$p^{s-1}\leq l<p^s$. 
	Lemma~\ref{Lemm:ConsecutiveZeroModp} implies
	$n+l+1\equiv0\pmod{p^s}$.
	Thus the conditions in~\ref{mnp} hold, and
	$\varphi=\lambda_l\chi$.
\end{proof}

\begin{corollary}\label{cor:case n=0}
	For each integer $m\geq0$, we have
	\[
	\dim (V_m^*)^{\SL_2}
	=
	\begin{cases}
		1 & \text{if }m=2(p^s-1)\text{ for some integer }s\geq0,\\
		0 & \text{otherwise}.
	\end{cases}
	\]
	If $m=2l$ with $l=p^s-1$, the invariant line
	$(V_m^*)^{\SL_2}$ is generated by 
	\[
	\varepsilon_l\colon V_{2l}\rightarrow\k,
	\qquad
	\sum_{i=0}^{2l}a_i x^iy^{2l-i}\mapsto a_l.
	\]
	For the action on $\k[z]_{\leq m}$ defined in
	Remark~$\ref{rk: def of dual Vm}$, we have
	\[
	\bigl(\k[z]_{\leq m}\bigr)^{\SL_2}
	=
	\begin{cases}
		\k\cdot z^{p^s-1}
		& \text{if }m=2(p^s-1)\text{ for some integer }s\geq0,\\
		\{0\} & \text{otherwise}.
	\end{cases}
	\]
\end{corollary}

\begin{proof}
	We identify $(V_m^*)^{\SL_2}$ with
	$\Hom^{\SL_2}(V_m,V_0)$.
	For $m=0$, the assertion is immediate.
	For $m>0$, Proposition~\ref{prop: rep SL2 car p}
	shows that this space is nonzero precisely when
	$m=2l$ and
	\[
	l+1\equiv0\pmod{p^s},
	\qquad p^{s-1}\leq l<p^s.
	\]
	Since $0<l+1\leq p^s$, this is equivalent to
	$l=p^s-1$.
	The generator given by the proposition is then
	$\varepsilon_l$, because its only nonzero value
	on the monomial basis occurs at $i=l$, where
	$\binom{l}{l}=1$.
	
	Finally, Remark~\ref{rk: def of dual Vm} gives an
	$\SL_2$-equivariant isomorphism
	$\k[z]_{\leq m}\simeq V_m^*$.
	When $m=2l$, the diagonal torus acts on $z^i$
	with weight $2i-2l$, so its zero-weight space is
	$\k\cdot z^l$.
	Whenever the $\SL_2$-invariant subspace is nonzero,
	it is therefore equal to this line.
\end{proof}

\subsection{Hirzebruch surfaces and their automorphism groups}
This section recalls some basic facts about Hirzebruch surfaces and their automorphism groups. We also answer a question raised by Laurent in his thesis (see \cite[Remarque~3.5.10]{Lau}) concerning the structure of connected subgroups of $\Aut(\F_n)$ acting transitively on the base $\P^1$ of the fibration $\F_n\to\P^1$; see Proposition~\ref{prop: UniformCounterexampleNew} and Lemma~\ref{lem:H-maps-onto-PGL2-odd-a}.

\begin{definition}
Let $n\in \Z$. The \emph{$n$-th Hirzebruch surface} $\F_n$ is defined to be the quotient of $(\A^2\setminus \{0\})^2$ by the action of $\G_m^2$ given by
\[\begin{array}{ccc}
(\G_m)^2 \times (\A^2\setminus \{0\})^2 & \to & (\A^2\setminus \{0\})^2\\
((\mu,\rho), (y_0,y_1,z_0,z_1))&\mapsto& (\mu\rho^{-n} y_0,\mu y_1,\rho z_0,\rho z_1)\end{array}\]
The class of $(y_0,y_1,z_0,z_1)$ will be written $[y_0:y_1;z_0:z_1]$. The projection 
\[\tau_n\colon\F_n\to\P^1, \ \ [y_0:y_1;z_0:z_1]\mapsto [z_0:z_1]\]
identifies $\F_n$ with $\P(\OP(n) \oplus \OP)$ as a $\P^1$-bundle over $\P^1$.

The disjoint sections $\s{-n},\s{n}\subseteq \F_n$ given by $y_0=0$ and $y_1=0$ have self-intersection $-n$ and $n$ respectively. The fibers $f\subseteq \F_n$ given by $z_0=0$ and $z_1=0$ are linearly equivalent and of self-intersection $0$. Furthermore we have $\Pic(\F_n)=\Z f\bigoplus \Z \s{-n}= \Z f\bigoplus \Z \s{n}$ since $\s{n}\sim\s{-n}+nf$.

As the surface $\F_n$ is naturally isomorphic to $\F_{-n}$ via $[y_0:y_1;z_0:z_1]\mapsto [y_1:y_0;z_0:z_1]$, we will most of the time assume that $n \ge 0$.
\end{definition}

We now recall the description of the automorphism groups of Hirzebruch surfaces.
\begin{itemize}
\item 
If $n=0$, then $\F_n\simeq \F_0$ and the natural action of $(\GL_2)^2$ on $(\A^2\setminus \{0\})^2$ yields a surjective group homomorphism $(\GL_2)^2\to \Autz(\F_0)\simeq (\PGL_2)^2$. In coordinates, this corresponds to the action given by
\[ \left( \left( \begin{bmatrix}
\alpha' & \beta' \\ \gamma' & \delta'
\end{bmatrix}, \begin{bmatrix}
\alpha & \beta \\ \gamma & \delta
\end{bmatrix}\right),[y_0:y_1;z_0:z_1] \right) \mapsto [\alpha' y_0+\beta' y_1:\gamma' y_0+\delta' y_1;\alpha z_0+\beta z_1:\gamma z_0+\delta z_1].\]
The action of $\Autz(\F_0)$ on $\F_0$ is then homogeneous (one orbit). Moreover, $\Aut(\F_0)=\Autz(\F_0)\rtimes \langle \iota\rangle$, with $\iota\colon[y_0:y_1;z_0:z_1]\mapsto [z_0:z_1;y_0:y_1]$. 
\item 
If $n\ge 1$, then the curve $\s{-n}$ given by $y_0=0$ is the unique section of negative self-intersection of $\F_n\to\P^1$ and is thus invariant. Let $\k[z_0,z_1]_{n}\subseteq \k[z_0,z_1]$ be the vector space of homogeneous polynomials of degree $n$. There is an action of $ \k[z_0,z_1]_{n}\rtimes \GL_2$ on $\F_n$ given by
\[ 
\left( \left( P, \begin{bmatrix}
\alpha & \beta \\ \gamma & \delta
\end{bmatrix} \right),
[y_0:y_1;z_0:z_1] \right) \mapsto [y_0:y_1+y_0P(\alpha z_0+\beta z_1,\gamma z_0+\delta z_1);\alpha z_0+\beta z_1:\gamma z_0+\delta z_1],\]
which yields an exact sequence of group schemes
\[1\rightarrow \mu_n \to \k[z_0,z_1]_{n}\rtimes \GL_2\rightarrow \Autz(\F_n)\rightarrow 1,\]
where $\mu_n\subseteq\GL_2$ is the scalar subgroup scheme
defined, for each $\k$-algebra $R$, by
\[
\mu_n(R)=\{\lambda I_2\mid \lambda\in R^\times,\ \lambda^n=1\}.
\]
Moreover, $\Aut(\F_n)=\Autz(\F_n)$ acts on $\F_n$ with two orbits, namely $\s{-n}$ and its complement.
\end{itemize}

\smallskip

It follows from the description of $\Autz(\F_n)$ above that, for each $n\ge 0$, the morphism $\tau_n\colon \F_n\to \P^1$ yields a surjective group homomorphism
\[\Autz(\F_n)\twoheadrightarrow\Aut(\P^1)\simeq \PGL_2.\]
In particular, $\Autz(\F_n)$ acts transitively on the set of fibers of $\tau_n\colon \F_n\to \P^1$.

\begin{remark}\label{StdActionGL2}
In \( \Aut(\F_n) \) for \( n \geq 1 \), the linear \( \GL_2 \)-action on \( \k[z_0, z_1]_n \) is given by
\[
\begin{array}{rcl}
\GL_2 \times \k[z_0, z_1]_n & \longrightarrow & \k[z_0, z_1]_n \\
\left(\begin{bmatrix}
\alpha & \beta \\
\gamma & \delta
\end{bmatrix}^{-1},\, P\right) & \longmapsto & P(\alpha z_0 + \beta z_1,\ \gamma z_0 + \delta z_1).
\end{array}
\]
Note that, for this $\GL_2$-action, we have
\[
\k[z_0, z_1]_n \simeq S^n(V^*) \text{\ \ \ as\ } \text{$\GL_2$-modules,}
\]
 and hence
\[
\k[z_0, z_1]_n \simeq V_n \text{\ \ \ as\ } \text{$\SL_2$-modules}
\]
(see Definition~\ref{def:Vn} for the definition of \( V_n \)). 
An explicit isomorphism of $\SL_2$-modules is given by
\[
\begin{array}{rcl}
\k[z_0, z_1]_n & \longrightarrow & V_n \\
P(z_0, z_1) & \longmapsto & P(-y, x),
\end{array}
\]
which follows from the fact that 
$\begin{bsmallmatrix}
 0 & -1 \\ 1 & 0 \end{bsmallmatrix} \in \SL_2$ conjugates every matrix in $\SL_2$ to the inverse of its transpose.

\textbf{This $\GL_2$-action on \( \k[z_0, z_1]_n \) should not be confused with the \( \GL_2 \)-action on \( \k[z]_{\leq n} \) introduced in Remark \ref{rk: def of dual Vm}.}
\end{remark}

\begin{proposition}[{\cite[Proposition~3.5.7 and Remark~3.5.10(2)]{Lau}}]
	\label{prop:trichotomy for algebraic subgroups of AUt(Fa)}\ \\
	Let $a\geq1$, and assume that either $a\geq2$ or $p\geq 3$.
	Let $H\subseteq\Aut(\F_a)$ be a connected algebraic subgroup.
	Then exactly one of the following holds:
	\begin{enumerate}
		\item
		$H$ is solvable and acts on $\F_a$ with a fixed point.
		\item
		After conjugation in $\Aut(\F_a)$, the subgroup $H$ is equal
		to the image of $\SL_2$ in $\GL_2/\mu_a$ or to
		$\GL_2/\mu_a$. Its three orbits are then the sections
		$s_a$, $s_{-a}$ and their open complement.
		\item
		$H$ acts on $\F_a$ with two orbits:
		the section $s_{-a}$ and its open complement.
	\end{enumerate}
\end{proposition}

\begin{remark}
	The exclusion of $(a,p)=(1,2)$ is necessary.
	Indeed, in characteristic $2$, the stabilizer
	$H=\Aut(\P^2,\Gamma)\simeq\PGL_2$ of the smooth conic
	$\Gamma=\{xz=y^2\}$ fixes its nucleus $q=[0:1:0]\notin\Gamma$.
	Its action therefore lifts to $\F_1=\mathrm{Bl}_q(\P^2)$, where it has
	three orbits: the exceptional section $s_{-1}$, the strict
	transform $\widetilde\Gamma$, and their open complement.
	Since $\widetilde\Gamma$ is a bisection of self-intersection $4$,
	this subgroup is not conjugate to the standard $\SL_2$ or $\GL_2$.
	Thus none of the three cases above holds.
	\end{remark}

In \cite[Remarque~3.5.10]{Lau}, Laurent asked whether every connected
algebraic subgroup of $\Aut(\F_a)$ acting on $\F_a$ with two orbits is
conjugate to a semidirect product of the form
\[
V\rtimes (\GL_2/\mu_a)
\qquad\text{or}\qquad
V\rtimes S,
\]
where
$S=\Im(\SL_2\to\GL_2/\mu_a)$
and $V\subseteq \k[x,y]_a$ is a $\GL_2$-submodule.

Although we will not use the following result later in the article, we
show that the answer to this question is negative in every positive
characteristic. More precisely, for $a=2p$ we construct a connected
subgroup $H\subseteq\Aut(\F_a)$ whose projection onto $S\simeq\PGL_2$
is a non-split extension (Proposition~\ref{prop: UniformCounterexampleNew}). 
In contrast, when $p\geq3$ and $a$ is odd,
every connected subgroup of $\Aut(\F_a)$ surjecting onto $\PGL_2$
contains a conjugate of the standard subgroup $S$ (Lemma~\ref{lem:H-maps-onto-PGL2-odd-a}). 
Both assumptions are essential: in characteristic $p=2$, an additional counterexample with
$a=1$ will be given in Lemma~\ref{lem: counter example for p=2}. 
We first record two preliminary lemmas. The torus-lifting
lemma will be used in
Lemma~\ref{lem:H-maps-onto-PGL2-odd-a} and
Corollary~\ref{cor:H-maps-onto-PGL2-even-a}.

\begin{lemma}\label{SL2OverPGL2}
	Let $a\geq 1$ be an integer, and let
	$H\subseteq \GL_2/\mu_a$ be a connected algebraic subgroup.
	If the standard projection
	$q\colon \GL_2/\mu_a\to\PGL_2$ restricts to a surjective
	homomorphism on $H$, then either $H=\GL_2/\mu_a$ or $H$ is
	the image of the natural homomorphism
	$\SL_2\to\GL_2/\mu_a$.
\end{lemma}

\begin{proof}
	Set $G=\GL_2/\mu_a$, and let $S\subseteq G$ be the image
	of $\SL_2$. Since $q(H)=\PGL_2$, we have
	\[
	q([H,H])=[\PGL_2,\PGL_2]=\PGL_2,
	\qquad
	[H,H]\subseteq[G,G]=S.
	\]
	Thus $[H,H]$ is a closed subgroup of the connected
	three-dimensional group $S$ and has dimension at least
	$\dim\PGL_2=3$. Consequently, $[H,H]=S$, so $S\subseteq H$.
	Since $\dim S=3$ and $\dim G=4$, connectedness of $H$
	and $G$ now implies that either $H=S$ or $H=G$.
\end{proof}

\begin{lemma}\label{lem:lifttorus}
Let $a\geq 1$ be an integer. Write
\[
\Aut(\F_a)\simeq U_a\rtimes(\GL_2/\mu_a),
\qquad
U_a:=\k[z_0,z_1]_a,
\]
and let $H\subseteq\Aut(\F_a)$ be a connected algebraic subgroup whose
image in $\GL_2/\mu_a$ contains the diagonal torus of
$S=\operatorname{Im}\bigl(\SL_2\to\GL_2/\mu_a\bigr)$.
Then there exists $u\in U_a$ such that $uHu^{-1}$ contains the diagonal
torus of $S$.
\end{lemma}

\begin{proof}
	Let $T\subseteq S$ be the standard maximal torus, and let
	$
	\rho\colon H\longrightarrow\GL_2/\mu_a
	$
	be the restriction of the second projection.
	Set
	$
	C:=\bigl(\rho^{-1}(T)_{\mathrm{red}}\bigr)^\circ.
	$
	Since $\rho^{-1}(T)\to T$ is surjective and $T$ is connected,
	the restriction $C\to T$ is also surjective.
	Its kernel is contained in $U_a$, so $C$ is solvable.
	
	A maximal torus $\widetilde T\subseteq C$ maps onto $T$.
	Moreover, the scheme-theoretic intersection
	$\widetilde T\cap U_a$ is trivial, since it is both of
	multiplicative type and a subgroup scheme of a vector group.
	Thus $\widetilde T\to T$ is an isomorphism.
	
	The tori $\widetilde T$ and $T$ are maximal tori of the
	connected solvable group $U_a\rtimes T$.
	By conjugacy of maximal tori, there exists $u\in U_a$ such that
	$u\widetilde T u^{-1}=T$.
	Since $\widetilde T\subseteq H$, we obtain
	$T\subseteq uHu^{-1}$.
\end{proof}

\begin{proposition}\label{prop: UniformCounterexampleNew}
Let $\k$ be an algebraically closed field of characteristic $p>0$. Consider the following $\GL_2$-modules, with the action of Remark~$\ref{StdActionGL2}$:
\[
V=\langle z_1^{2p},\,z_0^pz_1^p,\,z_0^{2p}\rangle
\subseteq \k[z_0,z_1]_{2p}
\]
and
\[
E=\langle z_0^iz_1^{2p-2-i}\mid 0\leq i\leq 2p-2,\ i\neq p-1\rangle
\subseteq \k[z_0,z_1]_{2p-2}.
\]
Then the following hold.
\begin{enumerate}
\item\label{kVEiso}
The map  
\[
\begin{array}{rccc}
\pi\colon &\k[z_0,z_1]_{2p}&\longrightarrow&E\\
&\displaystyle\sum_{i=0}^{2p}a_i z_0^iz_1^{2p-i}
&\longmapsto&
\displaystyle\sum_{i=1}^{2p-1} i a_i
z_0^{i-1}z_1^{2p-i-1}
\end{array}
\]

is a surjective morphism of $\SL_2$-modules, but not of $\GL_2$-modules, with kernel~$V$. In particular,
$\k[z_0,z_1]_{2p}/V\simeq E$
as $\SL_2$-modules.

\item\label{ConjtildeS}
Let $h=(z_0z_1)^{p-1}\in\k[z_0,z_1]_{2p-2}$.
The conjugate
\[
\widetilde S=h\SL_2h^{-1}
\subseteq \k[z_0,z_1]_{2p-2}\rtimes\SL_2
\]
of the standard $\SL_2$ is contained in $E\rtimes\SL_2$ and is not
conjugate to the standard $\SL_2$ in $E\rtimes\GL_2$.

\item\label{NonSplitGL2}
Let $H\subseteq \Aut(\F_{2p})
=\k[z_0,z_1]_{2p}\rtimes(\GL_2/\mu_{2p})$ be the image of the inverse
image of $\widetilde S$ under
\[
\pi\times\Id\colon
\k[z_0,z_1]_{2p}\rtimes\SL_2
\longrightarrow E\rtimes\SL_2.
\]
Then $H$ is a connected algebraic subgroup of $\Aut(\F_{2p})$, and its
projection to $\GL_2/\mu_{2p}$ yields a non-split exact sequence
\[
1\longrightarrow V
\longrightarrow H
\longrightarrow
S=\operatorname{Im}\bigl(\SL_2\to\GL_2/\mu_{2p}\bigr)
\simeq\SL_2/\mu_2\simeq\PGL_2
\longrightarrow1.
\] 
In particular, $H$ contains no subgroup conjugate in
$\Aut(\F_{2p})$ to $S$.
Moreover, $H$ acts on $\F_{2p}$ with exactly two orbits:
$s_{-2p}$ and its complement.
\end{enumerate}
\end{proposition}

\begin{proof}Recall that $\GL_2$ acts on $\k[z_0,z_1]_{2p}$ and
$\k[z_0,z_1]_{2p-2}$ by
\[
(M\cdot P)(z)=P(M^{-1}z).
\]
The subspace $V\subset\k[z_0,z_1]_{2p}$ is $\GL_2$-invariant, since
$V=\{f^p\mid f\in\k[z_0,z_1]_2\}$.

Consider the linear form
\[
\begin{array}{rccc}
\chi\colon&\k[z_0,z_1]_{2p-2}&\longrightarrow&\k\\
&\displaystyle\sum_{i=0}^{2p-2}a_i z_0^iz_1^{2p-2-i}
&\longmapsto&a_{p-1},
\end{array}
\]
so that $E=\Ker(\chi)$. We claim that
\[
\chi(M\cdot P)=\det(M)^{1-p}\chi(P)
\qquad
\text{for all }M\in\GL_2.
\]
In particular, $E$ is $\GL_2$-invariant, and $\chi$ is
$\SL_2$-invariant.

It is enough to check the formula on the diagonal torus, on
$\begin{bsmallmatrix}0&1\\1&0\end{bsmallmatrix}$, and on the upper
unipotent subgroup. If
$M=\begin{bsmallmatrix}a&0\\0&d\end{bsmallmatrix}$, then
\[
M\cdot (z_0z_1)^{p-1}
=(ad)^{1-p}(z_0z_1)^{p-1},
\]
and hence
\[
\chi(M\cdot P)=\det(M)^{1-p}\chi(P).
\]
For
$w=\begin{bsmallmatrix}0&1\\1&0\end{bsmallmatrix}$, we have
$\det(w)^{1-p}=(-1)^{p-1}=1$, and $\chi(w\cdot P)=\chi(P)$.

Finally, let $t=\begin{bsmallmatrix}1&-\lambda\\0&1\end{bsmallmatrix}$.
Then
\[
t\cdot(z_0^iz_1^{2p-2-i})
=(z_0+\lambda z_1)^iz_1^{2p-2-i}.
\]
If $i<p-1$, the coefficient of $z_0^{p-1}z_1^{p-1}$ is zero; if
$i=p-1$, it is $1$; and if $p-1<i\leq2p-2$, it is
$\lambda^{i-p+1}\binom{i}{p-1}=0$
by Lucas' theorem. Thus $\chi(t\cdot P)=\chi(P)$.
Since these elements generate $\GL_2$, the claim follows.

\ref{kVEiso}
By definition,
\[
\pi(z_0^iz_1^{2p-i})=
\begin{cases}
	i\,z_0^{i-1}z_1^{2p-i-1}, & 1\leq i\leq2p-1,\\
	0, & i=0\text{ or }i=2p.
\end{cases}
\]
Since $i=0$ in $\k$ precisely for $i=0,p,2p$, the map $\pi$ is
surjective onto $E$ and
\[
\Ker(\pi)
=\langle z_1^{2p},z_0^pz_1^p,z_0^{2p}\rangle
=V.
\]

We now prove that $\pi$ is $\SL_2$-equivariant.
Since $V$ is invariant and $\pi$ vanishes on $V$, equivariance
already holds on $V$.
In the computations below, we therefore take
$1\leq i\leq2p-1$.
Equivariance under the diagonal torus is immediate from the weights. For
$s=\begin{bsmallmatrix}0&1\\-1&0\end{bsmallmatrix}$, we have
$s\cdot(z_0^iz_1^{2p-i})
=(-1)^iz_0^{2p-i}z_1^i$,
and therefore
\[
\begin{aligned}
\pi\bigl(s\cdot(z_0^iz_1^{2p-i})\bigr)
&=(2p-i)(-1)^iz_0^{2p-i-1}z_1^{i-1}\\
&=i(-1)^{i-1}z_0^{2p-i-1}z_1^{i-1}\\
&=s\cdot\pi(z_0^iz_1^{2p-i}).
\end{aligned}
\]
For
$t=\begin{bsmallmatrix}1&-\lambda\\0&1\end{bsmallmatrix}$, we have
\[
t\cdot(z_0^iz_1^{2p-i})
=\sum_{j=0}^i\binom{i}{j}\lambda^{i-j}
z_0^jz_1^{2p-j},
\]
and hence
\[
\pi\bigl(t\cdot(z_0^iz_1^{2p-i})\bigr)
=
\sum_{j=1}^i
j\binom{i}{j}\lambda^{i-j}
z_0^{j-1}z_1^{2p-j-1}.
\]
On the other hand,
\[
t\cdot\pi(z_0^iz_1^{2p-i})
=
\sum_{r=0}^{i-1}
i\binom{i-1}{r}\lambda^{i-1-r}
z_0^rz_1^{2p-r-2}.
\]
These expressions coincide after setting $j=r+1$, since
\[
(r+1)\binom{i}{r+1}=i\binom{i-1}{r}.
\]
Thus $\pi$ is $\SL_2$-equivariant, and consequently
\[
\k[z_0,z_1]_{2p}/V\simeq E
\]
as $\SL_2$-modules. It is not $\GL_2$-equivariant, since the scalar
matrices act with different weights on
$\k[z_0,z_1]_{2p}$ and $\k[z_0,z_1]_{2p-2}$.

\ref{ConjtildeS} Put $h=(z_0z_1)^{p-1}$. For every
$\alpha\in\SL_2$, we have
\[
(h,\Id)(0,\alpha)(-h,\Id)
=(h-\alpha(h),\alpha).
\]
Since $\chi$ is $\SL_2$-invariant,
$\chi(h-\alpha(h))=0$,
and hence $h-\alpha(h)\in E$. Therefore
\[
\widetilde S=h\SL_2h^{-1}\subseteq E\rtimes\SL_2.
\]

Suppose that $\widetilde S$ is conjugate to the standard $\SL_2$ in
$E\rtimes\GL_2$. Since $\GL_2$ normalizes $\SL_2$, we may assume that
the conjugating element belongs to $E$, say $e\in E$. Then
$h-e$ normalizes the standard $\SL_2$, which is equivalent to
\[
\alpha(h-e)=h-e
\qquad\text{for all }\alpha\in\SL_2.
\]
Invariance under the diagonal torus gives
$h-e\in\k \left \langle h \right \rangle$.
Since $e\in E$ and $\chi(h)=1$, we obtain $h-e=h$, hence $e=0$. This
would imply that $h$ is $\SL_2$-invariant, which is false, for instance
by applying a nontrivial upper unipotent element. Thus $\widetilde S$
is not conjugate to the standard $\SL_2$ in $E\rtimes\GL_2$.

\ref{NonSplitGL2} Let $H'$ be the inverse image of $\widetilde S$ under
\[
\pi\times\Id\colon
\k[z_0,z_1]_{2p}\rtimes\SL_2
\longrightarrow E\rtimes\SL_2.
\]
Then $H'$ is a closed connected subgroup containing
$V=\Ker(\pi)$ and surjecting onto $\SL_2$. Its image
$H\subset\Aut(\F_{2p})$ is therefore a closed connected subgroup, and
the projection to $\GL_2/\mu_{2p}$ yields an exact sequence
\[
1\longrightarrow V
\longrightarrow H
\longrightarrow
S=\operatorname{Im}\bigl(\SL_2\to\GL_2/\mu_{2p}\bigr)
\simeq\PGL_2
\longrightarrow1.
\]

Suppose that this sequence splits. Pulling a section back
along $\SL_2\to S$ gives a regular cocycle
\[
b\colon\SL_2\longrightarrow W:=\k[z_0,z_1]_{2p}
\]
satisfying
\[
b(gg')=b(g)+g\cdot b(g'),
\qquad
\pi(b(g))=h-g\cdot h.
\]
Let $T\subseteq\SL_2$ be the diagonal torus.
Since $T$ fixes $h$, the restriction $b|_T$ takes values
in $V$.
As $T$ is linearly reductive, conjugating the section
by an element of $V$ allows us to assume that $b|_T=0$.

Put
$u_\lambda=
\begin{bsmallmatrix}1&-\lambda\\0&1\end{bsmallmatrix}$.
The relation
\[
\begin{bsmallmatrix}t&0\\0&t^{-1}\end{bsmallmatrix}
u_\lambda
\begin{bsmallmatrix}t^{-1}&0\\0&t\end{bsmallmatrix}
=u_{t^2\lambda}
\]
and $b|_T=0$ imply that
\[
b(u_\lambda)
=
\sum_{i=0}^{p-1}
c_i\lambda^{p-i}z_0^iz_1^{2p-i}
\]
for some $c_i\in\k$.
Comparing the coefficient of $z_0^{p-2}z_1^p$ in
$\pi(b(u_\lambda))=h-u_\lambda\cdot h$ gives
$c_{p-1}=-1$.

On the other hand, the coefficient of $z_1^{2p}$ in
the cocycle identity
\[
b(u_{\lambda+\mu})
=b(u_\lambda)+u_\lambda\cdot b(u_\mu)
\]
gives
\[
c_0(\lambda+\mu)^p
=
c_0\lambda^p+
\sum_{i=0}^{p-1}c_i\mu^{p-i}\lambda^i.
\]
Since $(\lambda+\mu)^p=\lambda^p+\mu^p$, comparison
of coefficients yields $c_i=0$ for $1\leq i\leq p-1$,
a contradiction.
Thus the exact sequence does not split.

Finally, $S$ acts transitively on the base $\P^1$,
and $V$ acts transitively on the complement of the
negative section in each fiber: evaluation of $V$
at any point of $\P^1$ is nonzero, since $V$ contains
$z_0^{2p}$ and $z_1^{2p}$.
Consequently, $H$ has exactly two orbits, namely
$s_{-2p}$ and its complement.
\end{proof}

For $p\geq3$, the fact that the preceding construction occurs
for an even Hirzebruch index is essential.
Indeed, when $a$ is odd, the opposite phenomenon holds:
Lemma~\ref{lem:H-maps-onto-PGL2-odd-a} below shows that every connected algebraic subgroup of $\Aut(\F_a)$
surjecting onto $\PGL_2$ contains, up to conjugation by the unipotent
radical $U_a$, the standard subgroup $S$. 
In particular, if the image of such a subgroup $H$ in
$\GL_2/\mu_a$ is $S$, then the exact sequence
$1\longrightarrow H\cap U_a\longrightarrow H
\longrightarrow S\longrightarrow 1$
splits.

\begin{lemma}\label{lem:H-maps-onto-PGL2-odd-a}
Assume that $p\geq3$. Let $a\geq1$ be odd.
Write
\[
\Aut(\F_a)\simeq U_a\rtimes(\GL_2/\mu_a),
\qquad
U_a:=\k[z_0,z_1]_a,
\]
and let $q\colon\Aut(\F_a)\to\PGL_2$
be the  standard projection $\Aut(\F_a)\to \GL_2/\mu_a\to \PGL_2$. Set
$S:=\Im\bigl(\SL_2\to\GL_2/\mu_a\bigr)$.
Let $H\subseteq\Aut(\F_a)$ be a connected algebraic subgroup such that
$q(H)=\PGL_2$. Then there exists $u\in U_a$ such that
$uSu^{-1}\subseteq H$.
\end{lemma}

\begin{proof}
Let $L:=\Im\bigl(H\to\GL_2/\mu_a\bigr).$ 
Since $q(L)=\PGL_2$, Lemma~\ref{SL2OverPGL2} yields
$S\subseteq L$.
Let $T\subseteq S$ be the standard maximal torus. By Lemma~\ref{lem:lifttorus}, we may replace $H$ with a conjugate by an element of $U_a$ and  assume that
$T\subseteq H$. Set
\[
J:=H\cap(U_a\rtimes S).
\]
Then $J\to S$ is surjective and $T\subseteq J$. Since $a$ is odd, the
central element $-I\in T$ acts on $U_a=\k[z_0,z_1]_a$ as multiplication
by $-1$. Thus, for every $(R,g)\in J$, using
$[x,y]=xyx^{-1}y^{-1}$, we obtain
\[
[(R,g),(0,-I)]=(2R,I)\in J\cap U_a.
\]
As $p\neq 2$, choose $m\in\Z$ such that $2m\equiv1\pmod p$. Since
$J\cap U_a$ is a subgroup of the additive group $U_a$, it follows that
\[
(R,I)=m(2R,I)\in J\cap U_a.
\]
Consequently,
\[
(0,g)=(R,I)^{-1}(R,g)\in J.
\]
Since $J\to S$ is surjective, this proves
$S\subseteq J\subseteq H$.
Undoing the initial conjugation gives $uSu^{-1}\subseteq H$ for some
$u\in U_a$.
\end{proof}

\section{\texorpdfstring{$\P^1$}{P1}-bundles over Hirzebruch surfaces}\label{sec:P1-bundles-over-Fa}
Throughout this section, we work in positive characteristic;
any additional restriction is stated explicitly.
	
We study $\P^1$-bundles over Hirzebruch surfaces and the induced
actions of their neutral automorphism groups on the base.
After recalling the decomposable case and the parameter spaces
for $\P^1$-bundles with no jumping fibers, we treat $\F_0$ and introduce
the families $\T^{b,c} \to \P^1 \times \P^1$ and $\U_{a,\omega}^{b,c} \to \F_a$.
Our main result in this section is the trichotomy of Theorem~\ref{th:P1bundlesFa}. 

\subsection{Decomposable \texorpdfstring{$\P^1$}{P1}-bundles over Hirzebruch surfaces}
In this section, we recall the definition of decomposable $\P^1$-bundles over Hirzebruch surfaces and describe their automorphism groups.

\begin{definition} \label{def Fabc}
Let $a,b,c\in \Z$. We define $\FF_a^{b,c}$ to be the quotient of $(\A^2\setminus \{0\})^3$ by the action of $(\G_m)^3$ given by
\[\begin{array}{ccc}
(\G_m)^3 \times (\A^2\setminus \{0\})^3 & \to & (\A^2\setminus \{0\})^3\\
((\lambda,\mu,\rho), (x_0,x_1,y_0,y_1,z_0,z_1))&\mapsto& 
(\lambda\mu^{-b} x_0, \lambda\rho^{-c} x_1,\mu\rho^{-a} y_0,\mu y_1,\rho z_0,\rho z_1)\end{array}\]
The class of $(x_0,x_1,y_0,y_1,z_0,z_1)$ will be written $[x_0:x_1;y_0:y_1;z_0:z_1]$. The projection 
\[\FF_a^{b,c}\to \F_{a}, \ \ [x_0:x_1;y_0:y_1;z_0:z_1]\mapsto [y_0:y_1;z_0:z_1]\]
identifies $\FF^{b,c}_a$ with 
\[\P(\OFa(b \s{a}) \oplus \OFa(c f))=\P(\OFa \oplus \OFa(-b\s{a}+cf))\]
as a $\P^1$-bundle over $\F_a$, where $\s{a},f\subseteq \F_a$ are given by $y_1=0$ and $z_1=0$.

Moreover, every fiber of the composed morphism
$\FF^{b,c}_a\to\F_a\to\P^1$ given by the $z$-projection
is isomorphic to $\F_b$, and the restrictions of
$\FF^{b,c}_a$ to the curves $\s{-a}$ and $\s{a}$,
given by $y_0=0$ and $y_1=0$, are isomorphic to
$\F_c$ and $\F_{c-ab}$, respectively.

Exchanging $y_0$ and $y_1$ yields an isomorphism
\[
\FF_a^{b,c}\iso\FF_{-a}^{b,c-ab}
\]
covering the corresponding isomorphism $\F_a\iso\F_{-a}$.
Thus one can assume that $a\geq0$ without changing the
isomorphism class.
We then observe that the exchange of $x_0$ and $x_1$ yields an isomorphism $\FF^{b,c}_a\iso \FF^{-b,-c}_a$. 
Henceforth, we will usually assume that $a,b\ge 0$. If $b=0$, we can moreover assume that $c\le 0$.
\end{definition}

\begin{lemma}\label{lem:aut-group-Fabc}\emph{(\cite[Lemma~3.1.5 and Remark~3.1.6]{BFT})}\\
For each $i,j \in \Z$, we denote by $\k[y_0,y_1,z_0,z_1]_{i,j}\subseteq \k[y_0,y_1,z_0,z_1]$ the space of homogeneous polynomials of bidegree $(i,j)$, where the variables $y_0,y_1,z_0,z_1$ are of bidegree $(1,-a)$, $(1,0)$, $(0,1)$, and $(0,1)$, respectively. Let $a,b\ge 0$ and $c\in \Z$.
There is an exact sequence of algebraic groups
\[
1 \to \Autz(\FF_a^{b,c})_{\F_a} \to \Autz(\FF_a^{b,c}) \to \Autz(\F_a) \to 1,
\]
with
\[
\Autz(\FF_a^{b,c})_{\F_a}
=\left\{ \begin{bmatrix}
 p_{1,0,0} & p_{2,-b,c} \\
 p_{3,b,-c} & p_{4,0,0}
\end{bmatrix}
\in \PGL_2(\k[y_0,y_1,z_0,z_1])
\ \middle|\
\begin{array}{l}
 p_{k,i,j}\in \k[y_0,y_1,z_0,z_1]_{i,j} \\
 \text{for } k=1,\ldots,4
\end{array}
\right\},
\]
acting on $\FF_{a}^{b,c}$ as follows:
\[
[x_0:x_1; y_0:y_1; z_0:z_1]
\longmapsto
[x_0 p_1+x_1 p_2 : x_0 p_3+x_1 p_4; y_0:y_1; z_0:z_1].
\]
\end{lemma}
The next result explains why we will subsequently focus on $\P^1$-bundles over Hirzebruch surfaces whose automorphism groups act transitively on the base $\P^1$ of the fibration $\F_a\to\P^1$.

\begin{lemma}\label{Lemm:fiberInvariant}
	\emph{(\cite[Lemma~3.7.1]{BFT})}
	Let $\pi\colon X\to\F_a$ be a $\P^1$-bundle, and let
	$\tau_a\colon\F_a\to\P^1$ be the standard projection.
	If there is a point $p\in\P^1$ such that the surface
	$(\tau_a\circ\pi)^{-1}(p)\subseteq X$ is invariant under
	$\Autz(X)$, then there exist a decomposable $\P^1$-bundle
	$\FF_a^{b,c}\to\F_a$ and a commutative diagram
\[\xymatrix@R=3mm@C=2cm{
    X \ar@{-->}[rr]^{\psi} \ar[rd]_{\pi}  && \FF_a^{b,c} \ar[ld] \\
    & \F_a
  }\]
  where $\psi$ is a birational map satisfying $\psi\Autz(X)\psi^{-1}\subsetneq \Autz(\FF_a^{b,c})$.
\end{lemma}

\subsection{\texorpdfstring{$\P^1$}{P1}-bundles over Hirzebruch surfaces with no jumping fibers}\label{sec:moduli P1 bundles over Fa}
In this section we recall certain results obtained in \cite[\S~3]{BFT}, valid over any algebraically closed base field, that we will need later.

\begin{definition}
Let $a,b,c\in \Z$ such that $a,b\ge 0$ and $c\le0$ if $b=0$.
We say that a $\P^1$-bundle $\pi\colon X\to \F_a$ \emph{has numerical invariants $(a,b,c)$} if it is the projectivization of a rank two vector bundle
$\mathcal{E}$ which fits into a short exact sequence
\[ 0 \to \O_{\F_a} \to \mathcal{E} \to \O_{\F_a}(-b\s{a}+cf) \to 0.\]
\end{definition}

\begin{example}
The $\P^1$-bundle $\FF_a^{b,c} \to \F_a$ has numerical invariants $(a,b,c)$.
\end{example}

Let $\pi\colon X\to \F_a$ be a $\P^1$-bundle; it is said to \emph{have no jumping fiber} if there exists $b \in \Z$ such that $(\tau_a \circ \pi)^{-1}(p) \simeq \F_b$ for every point $p \in \P^1$. Examples of $\P^1$-bundles over $\F_a$ having no jumping fibers are given by the $\P^1$-bundles $\FF_a^{b,c} \to \F_a$.

It follows from \cite[Proposition~3.3.1]{BFT} that a $\P^1$-bundle $\pi\colon X\to \F_a$ with numerical invariants $(a,b,c)$ has no jumping fiber, and that every $\P^1$-bundle $\pi\colon X\to \F_a$ with no jumping fiber has numerical invariants $(a,b,c)$, for some uniquely determined integers $b,c\in \Z$, with $b\ge 0$, and $c\le0$ if $b=0$. In addition, these numerical invariants are invariant under isomorphisms of $\P^1$-bundles. Let us also mention that, if $b=0$ or $c \leq 1$, every $\P^1$-bundle with numerical invariants $(a,b,c)$ is a decomposable $\P^1$-bundle by \cite[Corollary~3.3.7]{BFT}.

\begin{definition}\label{Not:ZabcP}
Let $a,b,c\in \Z$, with $a,b\ge 0$ and $c\le 0$ if $b=0$. For each 
\[
P\in\k\left[y_0,y_1,z, \frac{1}{z}\right]_b:=\left\{f\in \k\left[y_0,y_1,z,\frac{1}{z}\right]\text{ homogeneous of degree $b$ in }y_0,y_1\right\}
\] 
we denote by \[Z_{a}^{b,c,P}\to \F_a\] the $\P^1$-bundle, with numerical invariants $(a,b,c)$, obtained by gluing two copies of $\F_b\times \A^1$  along $\F_b \times \A^1 \setminus \{0\}$ through the automorphism $\nu_{c,P}\in \Aut(\F_b \times \A^1 \setminus \{0\})$ defined by
\[
\nu_{c,P}\colon([x_0:x_1;y_0:y_1],z) \mapsto \left([x_0:x_1z^{c}+x_0P(y_0,y_1,z);y_0z^a:y_1],\frac{1}{z}\right).\] 
The $\P^1$-bundle $Z_{a}^{b,c,P}\to \F_a$ sends  $([x_0:x_1;y_0:y_1],z)\in \F_b\times \A^1$ onto, respectively, $[y_0:y_1;1:z]\in \F_a$ on the first chart and $[y_0:y_1;z:1]\in \F_a$ on the second chart.
\end{definition}

In the following classification, isomorphisms of $\P^1$-bundles
are understood to induce the identity on $\F_a$.
We set $\k[z]_{\leq r}=\{0\}$ and
$\k[z_0,z_1]_r=\{0\}$ for $r<0$.

\begin{proposition}$($\cite[Corollaries~3.3.7 and~3.3.8]{BFT}$)$ \label{prop: paramter space for non-dceomposable P1 bundles}
Let $a,b,c\in \Z$, with $a\ge 0$, $b\ge1$ and $c\ge 2$.
The isomorphism classes of indecomposable $\P^1$-bundles $X \to \F_a$, with numerical invariants $(a,b,c)$, are parametrized by the projective space
\begin{equation*}
\M_{a}^{b,c}=\P \left( \bigoplus\limits_{i=0}^{b} y_0^i y_1^{b-i}\cdot\k[z]_{\leq c-2-ai} \right).
\end{equation*}
(For a finite-dimensional vector space $W$, the notation
$\P(W)$ below denotes the space
of one-dimensional subspaces of $W$.)
Indeed, every $\P^1$-bundle with numerical invariants $(a,b,c)$ is isomorphic to $Z_{a}^{b,c,P}\to \F_a$, where
\begin{equation}\label{eq:def of P}
P(y_0,y_1,z)=\sum_{i=0}^b  {y_0}^i{y_1}^{b-i}P_i(z)z^{ai+1},
\ P_i(z)\in \k[z]_{\le c-2-ai} \text{ for } i=0,\ldots,b.
\end{equation}
Moreover, if $P$ and $Q$ are both in the normal form
\eqref{eq:def of P}, then the $\P^1$-bundles
$Z_a^{b,c,P}\to\F_a$ and $Z_a^{b,c,Q}\to\F_a$
are isomorphic if and only if $Q=\lambda P$ for some
$\lambda\in\k^*$.
For $P$ in this normal form, the $\P^1$-bundle
$Z_a^{b,c,P}\to\F_a$ is decomposable if and only if $P=0$.
\end{proposition}

\begin{remark}\label{rem:normal-form-reduction}
	For a Laurent polynomial
	\[
	A(y_0,y_1,z)
	=
	\sum_{i=0}^b\sum_{j\in\Z}
	a_{ij}y_0^iy_1^{b-i}z^j,
	\]
	define
	\[
	\operatorname{red}_{a,c}(A)
	:=
	\sum_{i=0}^b
	\sum_{\substack{j\in\Z\\ ai<j<c}}
	a_{ij}y_0^iy_1^{b-i}z^j.
	\]
	Then
	\[
	Z_a^{b,c,A}\simeq
	Z_a^{b,c,\operatorname{red}_{a,c}(A)}
	\]
	as $\P^1$-bundles over $\F_a$.
	Indeed, changes of trivialization on the two charts allow
	one to add
	\[
	z^cA_1(y_0,y_1,z)+A_2(y_0z^a,y_1,z^{-1}),
	\]
	where $A_1,A_2\in\k[y_0,y_1,z]$ are homogeneous of degree
	$b$ in $y_0,y_1$; see \cite[Lemma~3.3.6]{BFT}.
	These changes remove precisely the terms with $j\geq c$
	or $j\leq ai$.
	
	Equivalently, if
	\[
	A=\sum_{i=0}^b y_0^iy_1^{b-i}z^{ai+1}A_i(z),
	\qquad A_i\in\k[z,z^{-1}],
	\]
	then reduction consists in retaining in $A_i$ only the
	powers $z^j$ with $0\leq j\leq c-2-ai$.
	Thus all negative powers and all powers exceeding
	$c-2-ai$ are discarded.
\end{remark}

The following lemma generalizes \cite[Remark~3.6.4]{BFT}.
\begin{lemma}\label{AutZaF1}
Let $a,b,c\in \Z$ with $a\ge 0$, $b\ge 1$, and $c\ge 2$, and let $Z_a^{b,c,P}\to\F_a$ be an indecomposable $\P^1$-bundle
with numerical invariants $(a,b,c)$, where $P$ is a
nonzero polynomial in the normal form \eqref{eq:def of P}.

The group of automorphisms of the $\P^1$-bundle $Z_{a}^{b,c,P}\to \F_a$ identifies with the vector group
$\bigoplus_{i=1}^b \k[z_0,z_1]_{ai-c}$, whose action on the first chart is given by
\[
([x_0:x_1;y_0:y_1],z)\longmapsto
\bigl([x_0:x_1+\sum_{i=1}^b x_0\,y_0^i y_1^{\,b-i} F_i(z,1);\,y_0:y_1],z\bigr),
\]
where $F=(F_i)_{i=1,\ldots,b}\in \bigoplus_{i=1}^b \k[z_0,z_1]_{ai-c}$.
\end{lemma} 

\begin{proof}
We write $P(y_0,y_1,z)=\sum_{i=0}^b P_i(z)\,z^{ai+1} y_0^{i}y_1^{b-i}$,
where $P_i(z)\in \k[z]_{\le c-2-ai}$ as in Proposition~\ref{prop: paramter space for non-dceomposable P1 bundles}. The gluing function is given by
\[
\nu_{c,P}\colon([x_0:x_1;y_0:y_1],z)\longmapsto
\bigl([x_0:x_1z^{c}+x_0P(y_0,y_1,z);\,y_0z^a:y_1],\tfrac{1}{z}\bigr).
\]
As $b\ge 1$, automorphisms of $\P^1$-bundles of $Z_{a}^{b,c,P}\to \F_a$ correspond, on the two charts, to elements $\chi_1,\chi_2\in \Aut(\F_b\times \A^1)$ given by
\[
\begin{array}{rccc}
\chi_1\colon&([x_0:x_1;y_0:y_1],z)&\longmapsto&
([x_0:\alpha x_1+\sum_{i=0}^b x_0 y_0^i y_1^{b-i} f_i(z);\,y_0:y_1],z),\\
\chi_2\colon&([x_0:x_1;y_0:y_1],z)&\longmapsto&
([x_0:\beta x_1+\sum_{i=0}^b x_0 y_0^i y_1^{b-i} g_i(z);\,y_0:y_1],z),
\end{array}
\]
for some $\alpha,\beta\in \k^*$ and $f_0,\ldots,f_b,g_0,\ldots,g_b\in \k[z]$. These satisfy $\nu_{c,P}\circ \chi_1=\chi_2\circ \nu_{c,P}$, which is equivalent to 
\[  \hspace{-7mm}
(\alpha x_1+\sum_{i=0}^b x_0 y_0^i y_1^{b-i} f_i(z))z^c + x_0P(y_0,y_1,z)
= \beta (x_1 z^{c}+x_0P(y_0,y_1,z)) + \sum_{i=0}^b x_0 (y_0z^a)^i y_1^{b-i} g_i\!\left(\tfrac{1}{z}\right).
\] 
The coefficient of $x_1$ yields $\beta=\alpha$. The coefficient of $x_0y_0^iy_1^{b-i}$ gives
\begin{equation}\label{eqfigi}
f_i(z)z^c+P_i(z)z^{ai+1}
=\alpha P_i(z) z^{ai+1}+ z^{ai}g_i\!\left(\tfrac{1}{z}\right).
\end{equation} 
As $P\neq 0$, there exists $i\in \{0,\ldots,b\}$ such that $P_i\neq 0$. For such $i$, we have $0\neq P_i(z)\in \k[z]_{\le c-2-ai}$, hence $P_i(z)z^{ai+1}$ is a sum of monomials $\lambda_j z^j$ with $ai+1\le j\le c-1$. 
Equation~\eqref{eqfigi} then implies $\alpha=1$ and $f_i=g_i=0$.
Now let $i\in \{0,\ldots,b\}$ with $P_i=0$. Then \eqref{eqfigi} becomes $f_i(z)z^c=z^{ai} g_i\!\left(\tfrac{1}{z}\right)$. If $ai<c$, this again gives $f_i=g_i=0$. If $ai\ge c$, then $f_i\in \k[z]_{\le ai-c}$, and every such $f_i$ occurs. This yields the result.
\end{proof}

We now introduce group actions on $\M_{a}^{b,c}$. 
Let $\pi\colon X \to \F_a$ be a $\P^1$-bundle with numerical invariants $(a,b,c)$. Then, for every $ \varphi \in \Autz(\F_a)$, the $\P^1$-bundle $\varphi \circ \pi\colon X\to \F_a$ also has numerical invariants $(a,b,c)$.
 Furthermore, two $\P^1$-bundles $\pi_1$ and $\pi_2$ are isomorphic if and only if $\varphi \circ \pi_1$ and $\varphi \circ \pi_2$ are isomorphic.
Hence, there is a natural $\Autz(\F_a)$-action on $\M_{a}^{b,c}$ given by
\begin{equation}\label{eq:Aut(Fa) action on Mabc}
\Autz(\F_a) \times \M_{a}^{b,c} \to \M_{a}^{b,c},\ 
(\varphi, [\pi] ) \mapsto  [\varphi \circ \pi].
\end{equation}
Also, composing this action with the homomorphism 
\[\GL_2 \to \Autz(\F_a), \ \begin{bmatrix}
\alpha & \beta \\ \gamma & \delta
\end{bmatrix} \mapsto \left( [y_0:y_1;z_0:z_1] \mapsto [y_0:y_1;\alpha z_0+\beta z_1: \gamma z_0+\delta z_1] \right)
\]
yields a $\GL_2$-action on $\M_{a}^{b,c}$. 

On the other hand, for each $r \geq 0$, we defined in Remark~\ref{rk: def of dual Vm} a linear $\GL_2$-action on $\k[z]_{\leq r} \simeq S^r(V^*)^*$. This induces another $\GL_2$-action on
\begin{equation*}
	\M_a^{b,c}
	=
	\P\left(
	\bigoplus_{i=0}^b
	y_0^iy_1^{b-i}\cdot\k[z]_{\leq c-2-ai}
	\right)
	\simeq
	\P\left(
	\bigoplus_{\substack{0\leq i\leq b\\ ai\leq c-2}}
	S^{c-2-ai}(V^*)^*
	\right),
\end{equation*}
by letting $\GL_2$ act trivially on $y_0$ and $y_1$.

\smallskip

The next lemma will be useful later to compute explicitly the $\Autz(\F_a)$-action on $\M_{a}^{b,c}$ defined by \eqref{eq:Aut(Fa) action on Mabc}. It will also imply that the two $\GL_2$-actions on $\M_{a}^{b,c}$ defined above coincide.

\begin{lemma}\label{Lemm:ActionAutFaParSpace} \emph{(\cite[Lemma~3.4.5]{BFT})}
Let $a \ge 0$, $b\ge 1$, and $c\ge 2$. Let $\pi\colon Z_{a}^{b,c,P}\to \F_a$ be the $\P^1$-bundle induced by
\[
P(y_0,y_1,z)=\sum_{i=0}^b y_0^i y_1^{b-i} P_i(z) z^{ai+1}, \qquad P_i\in \k[z]_{\le c-2-ai}
\]
$($see Definition~$\ref{Not:ZabcP})$.
For each $\varphi\in \Autz(\F_a)$, the $\P^1$-bundle $\varphi \circ \pi$ is isomorphic to
$Z_{a}^{b,c,\hat{P}}\to \F_a$, where
\[
\hat{P}(y_0,y_1,z)=\sum_{i=0}^b y_0^i y_1^{b-i} \hat{P}_i(z) z^{ai+1}
\]
is defined as follows.
\begin{enumerate}
\item\label{ActionGL2}
If
$\varphi([y_0:y_1;z_0:z_1])=[y_0:y_1;\alpha z_0+\beta z_1:\gamma z_0+\delta z_1]$
for some $\sigma=\begin{bmatrix}\alpha&\beta\\ \gamma&\delta\end{bmatrix}\in \GL_2$,
then $\hat{P}_i=\sigma \cdot P_i$ for each $i$, where the action of $\GL_2$ on the vector space
$\k[z]_{\le c-2-ai}$ is the one given in Remark~$\ref{rk: def of dual Vm}$.
\item\label{ActionGL20}
If $a=0$ and
$\varphi([y_0:y_1;z_0:z_1])=[\alpha y_0+\beta y_1:\gamma y_0+\delta y_1:z_0:z_1]$
for some $\alpha,\beta,\gamma,\delta\in \k$ with $\alpha\delta-\beta\gamma\not=0$,
then $\hat{P}(y_0,y_1,z)$ satisfies
\[
P(y_0,y_1,z)=\hat{P}(\alpha y_0+\beta y_1,\gamma y_0+\delta y_1,z).
\]
\item\label{ActionFap}    
If $a\geq1$ and $\varphi([y_0:y_1;z_0:z_1])
=
[y_0:y_1+y_0R(z_0,z_1);z_0:z_1]$
for some $R\in\k[z_0,z_1]_a$, then
\[
\hat{P}
=
\operatorname{red}_{a,c}
\bigl(P(y_0,y_1-y_0R(1,z),z)\bigr),
\]
where $\operatorname{red}_{a,c}(-)$ is the normal-form reduction
defined in Remark~$\ref{rem:normal-form-reduction}.$
\end{enumerate}
\end{lemma}

In case~\ref{ActionFap}, the automorphism $\varphi$ acts
on the first chart by
\[
(y_0,y_1,z)\longmapsto
(y_0,y_1+y_0R(1,z),z),
\]
and on the second chart by
\[
(y_0,y_1,z)\longmapsto
(y_0,y_1+y_0R(z,1),z).
\]
Since $z^aR(z^{-1},1)=R(1,z)$, compatibility with the
gluing gives the transition polynomial
$P(y_0,y_1-y_0R(1,z),z)$.
Remark~\ref{rem:normal-form-reduction} then gives its
normal representative.

Let us now focus on the case $a=0$. Then there is another $\GL_2$-action on $\M_0^{b,c}$ obtained by composing the action defined by \eqref{eq:Aut(Fa) action on Mabc} with the homomorphism 
\[\GL_2 \to \Autz(\F_a), \ \begin{bmatrix}
\alpha' & \beta' \\ \gamma' & \delta'
\end{bmatrix} \mapsto \left( [y_0:y_1;z_0:z_1] \mapsto [\alpha' y_0+\beta' y_1: \gamma' y_0+\delta' y_1;z_0:z_1] \right).
\]
In fact, by Lemma~\ref{Lemm:ActionAutFaParSpace}\ref{ActionGL20}, this $\GL_2$-action coincides with the $\GL_2$-action obtained by letting $\GL_2$ act linearly on $\k[y_0,y_1]_b$ through
 \[
 \left(\begin{bmatrix}
\alpha' & \beta' \\ \gamma' & \delta'
\end{bmatrix}^{-1},Q(y_0,y_1)\right)\mapsto Q(\alpha' y_0+\beta' y_1,\gamma' y_0+\delta' y_1).
\]
Then $\k[y_0,y_1]_b \simeq S^b(V^*)\simeq V_b$ as $\SL_2$-modules and 
$\M_{0}^{b,c} \simeq \P(V_b^{(1)} \otimes V_{c-2}^{*(2)})$ as $\SL_2^{(1)} \times \SL_2^{(2)}$-varieties (here the superscripts $(1)$ and $(2)$ label the two group
factors and indicate which factor acts on each representation;
they do not denote Frobenius twists). In particular, denoting by \[\Delta:=\{ (g,g),\ g \in \SL_2\} \subseteq \SL_2 \times \SL_2\] the diagonal group (isomorphic to $\SL_2$), we have
\begin{equation}\label{eq:Mobc as Delta variety}
\M_{0}^{b,c} \simeq \P(V_b \otimes V_{c-2}^{*}) \simeq \P(\Hom(V_{c-2},V_b))\ \text{as $\Delta$-varieties.}
\end{equation}

\subsection{Non-maximality of indecomposable \texorpdfstring{$\P^1$}{P1}-bundles over \texorpdfstring{$\F_0$}{F0}}
The main result of this subsection is
Corollary~\ref{cor: non-maximality of indecomposable},
which states that the identity component of the automorphism
group of an indecomposable $\P^1$-bundle over
$\F_0=\P^1\times\P^1$ is birationally conjugate to a proper
subgroup of the identity component of the automorphism group
of some decomposable $\P^1$-bundle over $\F_0$.

\begin{lemma}\label{Lemm:S1S2}
Let $\pi \colon X\to \F_0$ be a $\P^1$-bundle, let $b_1,b_2\ge 0$ be such that for $i=1,2$, every fiber of $\mathrm{pr}_i\circ \pi\colon X\to \P^1$ is isomorphic to $\F_{b_i}$, where  $\mathrm{pr}_i\colon \F_0\to \P^1$ is the $i$-th projection $($no jumping fibers with respect to either ruling$)$. Then, one of the following holds:
\begin{enumerate}
\item
$\pi$ is a decomposable $\P^1$-bundle.
\item
We have $b_1,b_2\geq1$, and for $i=1,2$, the union of the $(-b_{i})$-curves of the fibers of $\mathrm{pr}_i\circ \pi$ is a section $S_i\subseteq X$ of $\pi$. Moreover, $S_1\not=S_2$.
\end{enumerate}
\end{lemma}
\begin{proof}
	Suppose that $\pi$ is indecomposable.
	If $b_i=0$ for either ruling, the classification of $\P^1$-bundles
	with no jumping fibers implies that $\pi$ is decomposable
	by \cite[Corollary~3.3.7]{BFT}. Thus $b_1,b_2\geq1$.
	
	Applying the classification with respect to the first ruling,
	we obtain numerical invariants $(0,b_1,c)$ for some $c\geq2$.
	In the corresponding presentation
	$X\simeq Z_0^{b_1,c,P}$, the equations $x_0=0$ on the two
	charts define a section $S_1$ of $\pi$.
	On every fiber of $\mathrm{pr}_1\circ\pi$, this is the
	unique section of self-intersection $-b_1$.
	Applying the same argument to the second ruling gives
	a section $S_2$ whose restriction to each fiber of
	$\mathrm{pr}_2\circ\pi$ has self-intersection $-b_2$.
	
	To distinguish these sections, restrict the presentation
	$Z_0^{b_1,c,P}$ to a fiber of the second ruling.
	Near $S_1$, with local coordinate $t=x_0/x_1$, the
	transition function is
	\[
	t\longmapsto
	\frac{t}{z^c+tP(y_0,y_1,z)}
	=
	z^{-c}t\pmod{t^2}.
	\]
	Consequently, the normal line bundle of the restricted section
	is $\O_{\P^1}(c)$, so its self-intersection is $c>0$.
	It cannot coincide with the section of self-intersection
	$-b_2<0$.
	Hence $S_1\neq S_2$.
\end{proof} 

The following proposition distinguishes three cases, according to whether the $\P^1$-bundle is decomposable, has no jumping fibers with respect to either ruling of $\F_0$, or has an $\Autz(X)$-invariant fiber over one of the two rulings.

\begin{proposition}\label{prop:three case P1 bundles over F0}
Let $\pi\colon X\to \F_0$ be a $\P^1$-bundle.
Then, one of the following holds.
\begin{enumerate}
\item\label{P1P1dec}
The $\P^1$-bundle $\pi\colon X\to \F_0$ is a decomposable $\P^1$-bundle $\FF^{b,c}_0 \to \F_0$.
\item\label{P1P1sch}
The $\P^1$-bundle $\pi$ is indecomposable, and there exist
$b_1,b_2\geq1$ such that, for $i=1,2$, every fiber of
$\mathrm{pr}_i\circ\pi$ is isomorphic to $\F_{b_i}$.
The union of the $(-b_i)$-curves in these fibers is a
section $S_i\subseteq X$ of $\pi$, and $S_1\neq S_2$.
Moreover, viewed as an effective Cartier divisor on $S_1$,
the scheme-theoretic intersection $S_1\cap S_2$ is equal
to $mC$ for some integer $m\geq1$ and some smooth curve
$C\simeq\P^1$.
\item\label{P1P1Fixfiber}
There exist a closed point $p\in\P^1$ and
$i\in\{1,2\}$ such that
$(\mathrm{pr}_i\circ\pi)^{-1}(p)$ is invariant under
$\Autz(X)$.
\end{enumerate}
Moreover, in cases~\ref{P1P1sch} and~\ref{P1P1Fixfiber}, there is a decomposable $\P^1$-bundle $\FF_0^{b,c}\to \F_0$ and a commutative diagram
\[\xymatrix@R=3mm@C=2cm{
    X \ar@{-->}[rr]^{\psi} \ar[rd]_{\pi} & & \FF_0^{b,c} \ar[ld] \\
    & \F_0
  }\]
  where $\psi$ is a birational map satisfying $\psi\Autz(X)\psi^{-1}\subsetneq \Autz(\FF_0^{b,c})$. Also, in Case~\ref{P1P1sch}, the strict transforms $S_1'$ and $S_2'$ of $S_1$ and $S_2$ respectively are two disjoint sections of $\FF_0^{b,c} \to \F_0$.
\end{proposition}

\begin{proof}
	Set $G:=\Autz(X)$.
	By Blanchard's lemma
	(Proposition~\ref{prop:blanchard-lemma}), the action of $G$
	descends to $\F_0$ and to each factor $\P^1$.
	
	Suppose that neither~\ref{P1P1dec} nor~\ref{P1P1Fixfiber}
	holds.
	A proper connected algebraic subgroup of $\PGL_2$ fixes
	a point of $\P^1$.
	Thus $G$ acts transitively on each factor of $\F_0$.
	In particular, all fibers of each morphism
	$\mathrm{pr}_i\circ\pi$ are isomorphic.
	Lemma~\ref{Lemm:S1S2} therefore gives $b_1,b_2\geq1$
	and distinct sections $S_1,S_2$.
	Both sections are $G$-invariant, since the negative
	section of each Hirzebruch fiber is unique.
	
	Let $D\subseteq\F_0$ be the scheme-theoretic image of
	$S_1\cap S_2$ under the isomorphism
	$\pi|_{S_1}\colon S_1\to\F_0$.
	Since $S_1$ and $S_2$ are distinct sections, $D$ is an
	effective Cartier divisor.
	It is nonzero: otherwise the two sections would be
	disjoint, and $\pi$ would be decomposable.
	Moreover, $D$ is $G$-invariant.
	
	Since $G$ is connected, it preserves every irreducible
	component of $D_{\mathrm{red}}$.
	Each such component dominates both factors of $\F_0$;
	otherwise its image in one factor would be a fixed point.
	The singular locus of the reduced curve
	$D_{\mathrm{red}}$ is finite and $G$-invariant.
	It must be empty, since a nonempty finite invariant
	subset would yield a fixed point on each factor.
	Thus $D_{\mathrm{red}}$ is smooth.
	
	Distinct components of a smooth curve are disjoint.
	However, two curves in $\P^1\times\P^1$ that dominate
	both factors have positive intersection number.
	Hence $D_{\mathrm{red}}$ is irreducible.
	
	Let $H$ be the image of $G$ in
	$\Autz(\F_0)=\PGL_2\times\PGL_2$.
	Every $H$-orbit in $D_{\mathrm{red}}$ maps surjectively
	onto each factor $\P^1$, so every such orbit is
	one-dimensional and therefore open in $D_{\mathrm{red}}$.
	Irreducibility implies that there is only one orbit.
	Since $H$ is a connected linear algebraic group,
	$D_{\mathrm{red}}$ is unirational.
	Being smooth and projective, it is therefore isomorphic
	to $\P^1$.
	Consequently, $D=mD_{\mathrm{red}}$ for some $m\geq1$,
	which proves~\ref{P1P1sch}.
	
	We now prove the final assertion.
	In Case~\ref{P1P1Fixfiber}, it follows from
	Lemma~\ref{Lemm:fiberInvariant}, applied to the relevant
	ruling.
	
	Suppose that~\ref{P1P1sch} holds, and write
	$S_1\cap S_2=mC$ as in the statement.
	The sections $S_1,S_2$, the curve $C$, and its image
	$D_0:=\pi(C)$ are $G$-invariant.
	Perform the elementary transformation obtained by blowing
	up $C$ and contracting the strict transform of
	$\pi^{-1}(D_0)$.
	This gives a $G$-equivariant birational map to another
	$\P^1$-bundle over $\F_0$.
	
	This operation decreases the intersection multiplicity
	of the two sections along $D_0$ by one.
	Indeed, locally near $C$, let $f=0$ define $D_0$ and
	choose a fiber coordinate $t$ such that the sections
	are given by
	\[
	t=0,\qquad t=uf^m,
	\]
	where $u$ is a unit.
	The elementary transformation replaces $t$ by $t/f$,
	so the transformed sections are given by
	\[
	t/f=0,\qquad t/f=uf^{m-1}.
	\]
	After $m$ such transformations, their strict transforms
	are disjoint sections.
	The resulting $\P^1$-bundle is therefore decomposable,
	and hence isomorphic to $\FF_0^{b,c}\to\F_0$ for some
	$b,c$.
	We obtain a birational map over $\F_0$ satisfying
	\[
	\psi G\psi^{-1}\subseteq\Autz(\FF_0^{b,c}).
	\]
	
	Finally, the image of $G$ on $\F_0$ preserves the
	nonempty proper curve $D_0$.
	In contrast, $\Autz(\FF_0^{b,c})$ surjects onto
	$\Autz(\F_0)$ and therefore acts transitively on $\F_0$
	by Lemma~\ref{lem:aut-group-Fabc}.
	Since $\psi$ is a birational map over $\F_0$, the
	inclusion above is strict.
\end{proof}

The next result is a partial reformulation of Proposition \ref{prop:three case P1 bundles over F0}.

\begin{corollary}\label{cor: non-maximality of indecomposable}
Let $\pi\colon X\to \F_0$ be an indecomposable $\P^1$-bundle. Then there is a decomposable $\P^1$-bundle $\FF_0^{b,c}\to \F_0$ and a commutative diagram
\[\xymatrix@R=3mm@C=2cm{
    X \ar@{-->}[rr]^{\psi} \ar[rd]_{\pi} & & \FF_0^{b,c} \ar[ld] \\
    & \F_0
  }\]
  where $\psi$ is a birational map satisfying $\psi\Autz(X)\psi^{-1}\subsetneq \Autz(\FF_0^{b,c})$.
\end{corollary}

\begin{proof}
	Since $\pi$ is indecomposable, Case~\ref{P1P1dec} of
	Proposition~\ref{prop:three case P1 bundles over F0}
	does not occur. The conclusion follows from the final
	assertion of that proposition.
\end{proof}

\subsection{The \texorpdfstring{$\P^1$}{P1}-bundles \texorpdfstring{$\T^{b,c} \to \F_0$}{P1 x P1}}\label{sec: Tbc}
We study a family of indecomposable $\P^1$-bundles over $\F_0 = \P^1 \times \P^1$, whose identity component of the automorphism group is not maximal in $\Bir(\P^3)$ (see Corollary~\ref{cor: non-maximality of indecomposable}), but which will play an important role in Section~\ref{section: P1 bundles over P2} when studying $\P^1$-bundles over $\P^2$.

\begin{definition}\label{def hat Schwrzenberger}
Let $b \geq 1$ and $c \geq 3$ such that $c=b+2(l+1)$ with $l=0$ or satisfying $p^{s-1} \leq l <p^s$ and $b+l+1 \equiv 0\ (\text{mod }p^s)$ for some $s \geq 1$.
We define $\T^{b,c}\to \F_0$ to be the $\P^1$-bundle $Z_{0}^{b,c,P}\to \F_0$, where \[
P(y_0,y_1,z)=\sum_{i=0}^b y_0^iy_1^{b-i} P_i(z)z \text{ \ \  and \ \ } P_i(z)=\binom{l+i}{l} z^{l+i}.\]
\end{definition}

\begin{remark}
	Let us use the notation of Definition~\ref{def hat Schwrzenberger}.
\begin{itemize}
\item The $\P^1$-bundle  $\T^{b,c}\to \F_0$ has numerical invariants $(0,b,c)$, since it is equal to $Z_{0}^{b,c,P}$ for some polynomial $P$. Moreover, it is indecomposable, since $P$ is in normal form
and $P_0(z)=z^l\neq0$.
\item If $l=0$, then we recover the $\P^1$-bundle $\TT_b \to \F_0$ studied in \cite[\S~3.5]{BFT}.
\item
For each integer $i$ with $0\leq i\leq b$, we have
\[
\binom{l+i}{l}
\equiv(-1)^l\binom{b+l-i}{l}\pmod p.
\]
For $l=0$, this is immediate. For $l>0$, it follows from
Lemma~\ref{lem:simplebin}\,\ref{sb2}, since
$b+l+1\equiv0\pmod{p^s}$.
In particular $P_b(z)=(-1)^lz^{b+l}$ and $P_0(z)=z^l$.
\end{itemize}
\end{remark}

The next lemma characterizes the $\P^1$-bundles $\T^{b,c}$
by the action of their automorphism groups on the base.
As before, $[X\to\F_0]$ denotes the isomorphism class
of the $\P^1$-bundle over $\F_0$.

\begin{lemma}\label{lemma:horizontal_part_Tbc} 
Let $X \to \F_0$ be a $\P^1$-bundle, and set 
\[
H := \Im(\Autz(X) \to \Autz(\F_0)).
\]  
Then $H$ contains 
\[
\Delta' := \{ (g,g) \mid g \in \Aut(\P^1) \} \simeq \PGL_2
\] 
if and only if either
\begin{enumerate}
    \item $X \to \F_0$ is decomposable, in which case $H = \Autz(\F_0)$, or
    \item $[X \to \F_0] = [\T^{b,c} \to \F_0]$ for some $b,c \in \N$ satisfying the numerical conditions of Definition~$\ref{def hat Schwrzenberger}$, in which case $H = \Delta'$.
\end{enumerate}
\end{lemma}

\begin{proof}
Assume that $\Delta' \subseteq H$. Then $\Autz(X)$ acts transitively on the set of fibers of each projection $\F_0 \to \P^1$, so $X \to \F_0$ has no jumping fiber. Hence $X \to \F_0$ is either decomposable or isomorphic to $Z_0^{b,c,P} \to \F_0$ for some $b \ge 1$, $c \ge 2$, and $P \in \k[y_0,y_1,z]$ as in \eqref{eq:def of P}.  

If $[X \to \F_0] = [Z_0^{b,c,P} \to \F_0]$, then 
\[
\Delta' \subseteq H \quad \Leftrightarrow \quad \Delta' \cdot [Z_0^{b,c,P} \to \F_0] = [Z_0^{b,c,P} \to \F_0],
\] 
i.e., $[Z_0^{b,c,P} \to \F_0]$ is a fixed point for the $\Delta$-action on $\M_0^{b,c}$ (the $\Delta$-action factorizes through $\Delta'$). 
Since $\SL_2$ has no nontrivial characters, a line in a
rational $\SL_2$-module is stable if and only if it is
generated by an invariant vector.
Thus the fixed points under consideration correspond to
nonzero elements of
$\Hom^{\SL_2}(V_{c-2},V_b)$, up to scalar multiplication.
By Proposition~\ref{prop: rep SL2 car p} and the isomorphism \eqref{eq:Mobc as Delta variety}, the fixed-point set $(\M_0^{b,c})^\Delta$ is either the singleton $[\T^{b,c} \to \F_0]$ (when $b,c$ satisfy Definition~\ref{def hat Schwrzenberger}) or empty. Hence $\Delta' \subseteq H$ implies that $X \to \F_0$ is either decomposable or $[X \to \F_0] = [\T^{b,c} \to \F_0]$.  

Conversely, for a decomposable $\P^1$-bundle, we have
$H=\Autz(\F_0)$ by Lemma~\ref{lem:aut-group-Fabc}.
For $\T^{b,c}$, the preceding computation shows that its
parameter is fixed by $\Delta'$.
Thus $\Delta'\subseteq H$.
Here we use the fact that the image of $\Autz(X)$ is the
reduced identity component of the stabilizer of the
isomorphism class $[X\to\F_0]$.

It remains to determine $H$ in the indecomposable case.
By Lemma~\ref{Lemm:S1S2}, there are two distinct
$H$-invariant sections $S_1,S_2$.
Their intersection is nonempty, since two disjoint sections
would imply that the $\P^1$-bundle is decomposable.
Its image in $\F_0$ is therefore a nonempty proper
$H$-invariant curve.
The diagonal subgroup $\Delta'$ has exactly two orbits on
$\F_0$: the diagonal $D$ and its complement.
Consequently, the reduced image of $S_1\cap S_2$ is $D$.
Since the stabilizer of $D$ in
$\Autz(\F_0)=\PGL_2\times\PGL_2$ is precisely $\Delta'$,
we obtain $H=\Delta'$.
\end{proof}

\begin{proposition}\label{prop:properties of the Tbc}
Let $b\ge1$ and $c\ge 3$ such that $c-2=b+2l$, where either $l=0$ or there exists $s\ge1$ with
$p^{s-1}\le l<p^s$ and $b+l+1\equiv0\pmod{p^s}$. Then:
\begin{enumerate}
\item\label{item vert part trivial}
$\Aut(\T^{b,c})_{\F_0}$ is trivial and $\Autz(\T^{b,c})\simeq\PGL_2$.
\item\label{item open orbit Tbc}
The variety $\T^{b,c}$ contains an open orbit isomorphic to $\PGL_2/H$, where
\[
H=
\left\{
\left[\begin{smallmatrix}t&0\\0&1\end{smallmatrix}\right]
\in\PGL_2
\ \middle|\ t^{b+l+1}=1
\right\}
\simeq\mu_{b+l+1},
\]
as a subgroup scheme.
\item\label{item unique involution Tbc}
There exists a unique involution $\sigma\in\Aut(\T^{b,c})$ such that
\[
\xymatrix@R=3mm{
\T^{b,c}\ar[r]^{\sigma}\ar[d]_{\pi} & \T^{b,c}\ar[d]^{\pi}\\
\F_0\ar[r]^{\tau} & \F_0
}
\]
commutes, where $\tau([y_0:y_1],[z_0:z_1])=([z_0:z_1],[y_0:y_1])$.
\item\label{item: open orbit in the quotient}
The quotient morphism $\T^{b,c}\to Z:=\T^{b,c}/\langle\sigma\rangle$ is
$\Autz(\T^{b,c})$-equivariant, and $Z$ contains an open orbit isomorphic to
$\PGL_2/H'$, where
\[
H'=\left\langle H,\begin{bmatrix}0& 1\\ -1&0\end{bmatrix}\right\rangle,  \text{ as group schemes},
\qquad \text{and}\qquad |H':H|=2.
\]
\end{enumerate}
\end{proposition}

\begin{proof} 
Applying Lemma~\ref{AutZaF1} with $a=0$, we obtain
\[
\Aut(\T^{b,c})_{\F_0}
\simeq
\bigoplus_{i=1}^b\k[z_0,z_1]_{-c}
=\{0\}.
\]
The chart computation in that lemma also shows that the
scheme-theoretic kernel is trivial.
Together with Lemma~\ref{lemma:horizontal_part_Tbc}, this gives
\[
\Autz(\T^{b,c})\simeq\Delta'\simeq\PGL_2,
\]
proving~\ref{item vert part trivial}. 

\smallskip

Let
\[
K=\left\{\begin{bmatrix}\alpha&0\\0&\alpha^{-1}\end{bmatrix}\mid\alpha\in\k^*\right\}
\simeq\G_m\subset\PGL_2.
\]
Its action on $\F_0$ is given by
\begin{equation}\label{eq:K-action on the basis}
\begin{bmatrix}\alpha&0\\0&\alpha^{-1}\end{bmatrix}\!\cdot\!
([y_0:y_1],[z_0:z_1])
=([\alpha y_0:\alpha^{-1}y_1],[\alpha z_0:\alpha^{-1}z_1]).
\end{equation}
In particular, $K$ stabilizes the two charts $\F_b \times \A^1$ introduced in Definition~\ref{Not:ZabcP} to define the $\P^1$-bundle $\T^{b,c} \to \F_0$. One can check, by using the transition function $\nu_{c,P}$ that defines the gluing between these two charts, that the regular $K$-action defined on the first chart by 
\begin{equation*} \begin{bmatrix} \alpha & 0 \\ 0 & \alpha^{-1} \end{bmatrix} \cdot ([x_0:x_1;y_0:y_1],z)=([\alpha^{-c} x_0:x_1;\alpha y_0:\alpha^{-1} y_1],\alpha^{-2}z).
\end{equation*} 
extends to a regular $K$-action on $\T^{b,c}$. Since $\Aut(\T^{b,c})_{\F_0}$ is trivial, this $K$-action on $\T^{b,c}$ is the unique lift of the $K$-action on the base $\F_0$ defined by \eqref{eq:K-action on the basis}.

Let $p_0=([1:1;0:1],0)$. Its stabilizer $H$ is contained in $K$ and consists of
those elements fixing $p_0$, hence
\[
H=
\operatorname{Im}\left(
\mu_{2(b+l+1)}
\rightarrow\PGL_2,\quad
\alpha\mapsto
\left[\begin{smallmatrix}\alpha&0\\0&\alpha^{-1}\end{smallmatrix}\right]
\right)
\simeq\mu_{b+l+1}.
\]
Since $\dim(\PGL_2\cdot p_0)=3$, this orbit is open, proving~\ref{item open orbit Tbc}.

\smallskip

We now consider the $\P^1$-bundle $X\to\F_0$ obtained by composing
$\pi\colon\T^{b,c}\to\F_0$ with the involution $\tau\colon\F_0\to\F_0$.
Since the actions of $\tau$ and $\Autz(\T^{b,c})\simeq\PGL_2$ on $\F_0$
commute, we have
\[
\Im\bigl(\Autz(X)\to\Autz(\F_0)\bigr)=\Delta'.
\]
By Lemma~\ref{lemma:horizontal_part_Tbc}, it follows that
$X\simeq\T^{b',c'}$ for some $b'\ge1$ and $c'\ge2$ satisfying
$c'-2=b'+2l'$, where either $l'=0$ or there exists $s'\ge1$ such that
$p^{s'-1}\le l'<p^{s'}$ and $b'+l'+1\equiv0\pmod{p^{s'}}$.

We claim that $b'=b$ and $c'=c$.
To compute the Hirzebruch index for the other ruling,
restrict to $[y_0:y_1]=[1:0]$.
The transition function of the resulting ruled surface is
\[
([x_0:x_1],z)\longmapsto
\left(
[x_0:x_1z^c+(-1)^lx_0z^{b+l+1}],
\frac1z
\right),
\]
since $P_b(z)=(-1)^lz^{b+l}$.
It is therefore induced by the matrix
$\begin{bsmallmatrix}
1 & 0 \\
(-1)^l z^{b+l+1} & z^c
\end{bsmallmatrix}$.
The equality
\[
\begin{bsmallmatrix}
0 & (-1)^l \\
-(-1)^l & z^{-b-l-1}
\end{bsmallmatrix}
\begin{bsmallmatrix}
1 & 0 \\
(-1)^l z^{b+l+1} & z^c
\end{bsmallmatrix}
\begin{bsmallmatrix}
1 & -(-1)^l z^{l+1} \\
0 & 1
\end{bsmallmatrix}
=
\begin{bsmallmatrix}
z^{b+l+1} & 0 \\
0 & z^{l+1}
\end{bsmallmatrix}
\]
then shows that $b' = b$.

Moreover, exchanging the two rulings does not change the
$\PGL_2$-action on the total space, since $\tau$ commutes
with the diagonal action.
The stabilizers of points in the open orbit therefore have
the same scheme-theoretic length.
By~\ref{item open orbit Tbc}, this gives
$b+l+1=b'+l'+1$.
Since $b'=b$, we obtain $l'=l$ and hence $c'=c$.

Consequently, there exists an automorphism $\sigma\in\Aut(\T^{b,c})$
such that the diagram
\begin{equation}\label{eq:diagram-commutes-with-sigma}
\xymatrix@R=3mm{
\T^{b,c}\ar[r]^{\sigma}\ar[d]_{\pi} &
\T^{b,c}\ar[d]^{\pi}\\
\F_0\ar[r]^{\tau} & \F_0
}
\end{equation}
commutes.
Since $\sigma^2\in\Aut(\T^{b,c})_{\F_0}=\{\Id\}$, the automorphism $\sigma$
is an involution.
If $\sigma'\in\Aut(\T^{b,c})$ is another involution making
\eqref{eq:diagram-commutes-with-sigma} commute, then
$\sigma\circ\sigma'\in\Aut(\T^{b,c})_{\F_0}=\{\Id\}$, hence $\sigma=\sigma'$.
This proves~\ref{item unique involution Tbc}.

\smallskip

The involution $\sigma$ from~\ref{item unique involution Tbc} defines a
$\Z/2\Z$-action on the projective variety $\T^{b,c}$, hence there exists a
geometric quotient
\[
q\colon\T^{b,c}\to Z:=\T^{b,c}/\langle\sigma\rangle,
\]
where $Z$ is projective. For every $\phi\in\Autz(\T^{b,c})$, we have
$\sigma\phi\sigma\phi^{-1}\in\Aut(\T^{b,c})_{\F_0}=\{\Id\}$; therefore the
actions of $\Autz(\T^{b,c})$ and $\langle\sigma\rangle$ commute. By the
universal property of geometric quotients, $\Autz(\T^{b,c})$ acts on $Z$ and
$q$ is $\Autz(\T^{b,c})$-equivariant.

By~\ref{item open orbit Tbc}, $\T^{b,c}$ contains an open $\PGL_2$-orbit
isomorphic to $\PGL_2/H$. Its image in $Z$ is an open $\PGL_2$-orbit
isomorphic to $\PGL_2/H'$, 
where, up to conjugacy in $\PGL_2$,
\[
H'=
\left\langle
H,
\left[\begin{smallmatrix}0&1\\-1&0\end{smallmatrix}\right]
\right\rangle,
\qquad
H'/H\simeq\Z/2\Z.
\]
Here $H'$ denotes the subgroup scheme generated by the
displayed subgroups.
More explicitly, $H'$ is the stabilizer of the unordered
pair $\{p_0,p_1\}$, where $p_1:=\sigma(p_0)$. As $p_0$ and $p_1$ have the same
stabilizer $H$, we have $H'=\langle H,g\rangle$ for some $g\in\PGL_2$ with
$g\cdot p_0=p_1$.
Since $p_0$ lies over $([0:1],[1:0])$ and $p_1$ over
$([1:0],[0:1])$, the element $g$ exchanges these two points
of the base. Thus
$g=
\left[\begin{smallmatrix}
	0&\beta\\
	-\beta^{-1}&0
\end{smallmatrix}\right]$
for some $\beta\in\k^\times$.
Conjugating by a suitable element of the diagonal torus
transforms $g$ into
$\left[\begin{smallmatrix}0&1\\-1&0\end{smallmatrix}\right]$
and preserves $H$.
This proves~\ref{item: open orbit in the quotient}.
\end{proof}

\begin{lemma}\label{TbcInvolutionExplicit}
Let $b\geq1$ and $l\geq0$ be integers, and set
$c=b+2(l+1)$.
Assume that either $l=0$, or there exists $s\geq1$
such that
\[
p^{s-1}\leq l<p^s,
\qquad b+l+1\equiv0\pmod{p^s}.
\]
We consider the $\P^1$-bundle $\T^{b,c}\to \F_0$ of Definition~$\ref{def hat Schwrzenberger}$. Restricting to the two affine open subsets of $\F_0$ given by \[(y,z)\mapsto [1:y;1:z]\text{ and }(y,z)\mapsto [y:1;z:1]\]
we obtain trivial $\P^1$-bundles, with the transition function $\nu_{b,c}\in \Aut(\P^1 \times (\A^1\setminus \{0\})^2)$ given by
\[
\nu_{b,c}\colon ([x_0:x_1],(y,z))\mapsto 
\Big([y^bx_0:x_1z^{c}+x_0\sum_{i=0}^b y^{b-i} \binom{l+i}{l} z^{l+i+1}],
(\tfrac{1}{y},\tfrac{1}{z})\Big).
\]
The unique involution $\sigma \in \Aut(\T^{b,c})$ of Proposition~$\ref{prop:properties of the Tbc}\ref{item unique involution Tbc}$ corresponds to 
\[
\begin{array}{rcl}
\sigma_1\colon ([x_0:x_1],(y,z)) &\mapsto& 
([x_0-(y-z)^{l+1}x_1 : (-1)^l x_1], (z,y)),\\[1mm]
\sigma_2\colon ([x_0:x_1],(y,z)) &\mapsto& 
([x_0-(z-y)^{l+1}x_1 : (-1)^l x_1], (z,y)).
\end{array}
\]
on both charts. In particular, $\T^{b,c}/\langle\sigma\rangle$ is smooth
	if $p=2$ or $l$ is even.
	If $p\geq 3$ and $l$ is odd, its singular locus consists
	of two disjoint smooth rational curves.
\end{lemma}

\begin{proof}
Recall that $\T^{b,c}\to \F_0$ is obtained by gluing two copies of $\F_b\times \A^1$ along 
$\F_b\times(\A^1\setminus\{0\})$ via the automorphism 
$\nu_{c,P}\in \Aut(\F_b\times(\A^1\setminus\{0\}))$ defined by 
\[
\begin{array}{c}
\nu_{c,P}\colon([x_0:x_1;y_0:y_1],z) \mapsto 
\left([x_0:x_1z^{c}+x_0P(y_0,y_1,z);y_0:y_1],\frac{1}{z}\right), \text{ where }\\[0.3em]
P(y_0,y_1,z)=\sum_{i=0}^b y_0^iy_1^{b-i} P_i(z)z 
\quad \text{ and } \quad 
P_i(z)=\binom{l+i}{l} z^{l+i}.
\end{array}
\]

We choose two embeddings of $\A^2$ into $\F_0=\P^1\times \P^1$, namely
$(y,z)\mapsto [1:y;1:z]$ and $(y,z)\mapsto [y:1;z:1]$, and obtain two copies of 
$\P^1\times \A^2$ inside $\T^{b,c}$.
The first one is given by the embedding into $\F_b\times \A^1$ on the first chart,
\[
([x_0:x_1],(y,z))\mapsto ([x_0:x_1;1:y],z),
\]
and the second is given by
\[
([x_0:x_1],(y,z))\mapsto ([x_0:x_1;y:1],z)
\]
on the second chart.
As $\nu_{c,P}$ sends $([x_0:x_1;1:y],z)$ onto
\[
([x_0:x_1z^{c}+x_0P(1,y,z);1:y],\tfrac{1}{z})
=([y^bx_0:x_1z^{c}+x_0P(1,y,z);\tfrac{1}{y}:1],\tfrac{1}{z}),
\]
we obtain two copies of $\P^1\times \A^2$, together with the gluing function
$\nu_{b,c}\in \Aut(\P^1 \times (\A^1\setminus \{0\})^2)$ of the statement. To prove that the unique involution $\sigma \in \Aut(\T^{b,c})$ from
Proposition~\ref{prop:properties of the Tbc}\ref{item unique involution Tbc} 
corresponds to the two involutions 
$\sigma_1, \sigma_2 \in \Aut(\P^1 \times \A^2)$ of the statement, one needs to prove  that
\[
\sigma_2 \circ \nu_{b,c} = \nu_{b,c} \circ \sigma_1.
\]
Let
\[
Q(y,z)=\sum_{i=0}^b \binom{l+i}{l} y^{b-i} z^i \in \k[y,z].
\]
By Lemma~\ref{EqualityQyz},
\[
Q(y,z)=\frac{y^{b+l+1}-z^{b+l+1}}{(y-z)^{l+1}},
\]
hence in particular $Q(z,y)=(-1)^l Q(y,z)$.
We obtain
\[
\nu_{b,c}([x_0:x_1],(y,z))
=
\bigl([y^b x_0 : z^{l+1}Q(y,z)x_0 + z^c x_1],(\tfrac{1}{y},\tfrac{1}{z})\bigr),
\]
and therefore
\[
\sigma_2 \circ \nu_{b,c}([x_0:x_1],(y,z))
=
([A:B],(\tfrac{1}{z},\tfrac{1}{y})),
\quad
\nu_{b,c} \circ \sigma_1([x_0:x_1],(y,z))
=
([C:D],(\tfrac{1}{z},\tfrac{1}{y})),
\]
for some $A,B,C,D \in \k[x_0,x_1](y,z)$. We check that $[A:B]=[C:D]$:
\[
\begin{array}{rcl}
[A:B]
&=&
\bigl[
y^b x_0 - \left(\tfrac{1}{z}-\tfrac{1}{y}\right)^{l+1}
\bigl(z^{l+1}Q(y,z)x_0 + z^c x_1\bigr)
:
(-1)^l\bigl(z^{l+1}Q(y,z)x_0 + z^c x_1\bigr)
\bigr]
\\[4pt]
&=&
\bigl[
\frac{z^{b+l+1}}{y^{l+1}} x_0
-
\frac{(y-z)^{l+1} z^{b+l+1}}{y^{l+1}} x_1
:
(-1)^l z^{l+1}Q(y,z)x_0 + (-1)^l z^c x_1
\bigr],
\\[6pt]

[C:D]&=&[z^b (x_0 - (y-z)^{l+1} x_1) : y^{l+1}Q(z,y)(x_0-(y-z)^{l+1}x_1)+y^c(-1)^lx_1]\\ &=&[z^bx_0-(y-z)^{l+1}z^bx_1:(-1)^ly^{l+1} Q(y,z)x_0+(-1)^ly^{l+1} z^{b+l+1} x_1].
\end{array}
\]
Multiplying the displayed coordinates $A,B$ by
$(y/z)^{l+1}$ gives $C,D$, respectively.

The two base charts cover $\F_0$ except for
$([1:0],[0:1])$ and $([0:1],[1:0])$.
By \cite[Lemma~2.3.1]{BFT}, the compatible formulas
extend uniquely to an automorphism of $\T^{b,c}$,
which is the lift $\sigma$.

Suppose first that $p=2$, and put $u=y+z$.
On both charts, the involution is given by
\[
\sigma([x_0:x_1],(y,z))
=
([x_0+u^{l+1}x_1:x_1],(z,y)).
\]
Consider the automorphism
\[
\tau([x_0:x_1],(y,z))
=
([x_0+y u^l x_1:x_1],(y,z)).
\]
Since $y u^l+z u^l=u^{l+1}$, we obtain
\[
\tau^{-1}\sigma\tau([x_0:x_1],(y,z))
=
([x_0:x_1],(z,y)).
\]
Consequently, the quotient on each chart is
$\P^1\times\A^2$, with quotient morphism, in these
new coordinates, given by
\[
([x_0:x_1],(y,z))
\longmapsto
([x_0:x_1],(y+z,yz)).
\]
Thus the quotient is smooth near the image of
$\pi^{-1}(D)$, where $D\subseteq\F_0$ is the diagonal.
Every fixed point lies over $D$, and away from the
fixed locus the quotient morphism is étale.
Hence $\T^{b,c}/\langle\sigma\rangle$ is smooth.

We now assume that $p\geq 3$.
Every fixed point lies over the diagonal
$D\subseteq\F_0$, which is covered by the two charts.
If $l$ is even, the fixed locus is $\pi^{-1}(D)$;
locally at a fixed point, the involution has the form
\[
(s,t,w)\longmapsto(s,-t,w).
\]
Its quotient is therefore smooth.

If $l$ is odd, the fixed locus consists locally of
$y=z$ together with either $x_0=0$ or $x_1=0$.
These define two disjoint sections over $D$:
indeed, $Q(y,y)=0$, so the transition matrix is diagonal
along $D$.
At each fixed point, the involution has local form
\[
(s,t,w)\longmapsto(s,-t,-w),
\]
and the completed local ring of the quotient is
\[
\k[[s,A,B,C]]/(AC-B^2).
\]
Away from the fixed locus, the quotient morphism is
étale. Thus the singular locus consists of the images
of the two fixed sections, which are disjoint smooth
rational curves.
\end{proof}

For $p\geq3$ and even $l$, the following proposition realizes
$\T^{b,c}/\langle\sigma\rangle$ as a $\P^1$-bundle
over $\P^2$. These quotients will be studied in
Section~\ref{sec: Schwarzenbergger bundles}.

\begin{proposition}\label{prop: def of the Sbc}
	Assume that $p\geq 3$.
	Let $b,c,l$ be as in Lemma~$\ref{TbcInvolutionExplicit}$,
	and assume that $l$ is even.
	Then the quotient $\T^{b,c}/\langle\sigma\rangle$ is a
	$\P^1$-bundle over $\P^2$, fitting into the commutative diagram
	\[
	\xymatrix@R=3mm{
		\T^{b,c}\ar[r]\ar[d] &
		\T^{b,c}/\langle\sigma\rangle\ar[d]\\
		\F_0\ar[r]^{\kappa} & \P^2,
	}
	\]
	where $\kappa$ is the double covering
	\[
	\begin{array}{rccc}
		\kappa\colon & \F_0 & \to & \P^2\\
		& ([y_0:y_1],[z_0:z_1])
		& \mapsto &
		[y_0z_0:y_0z_1+y_1z_0:y_1z_1].
	\end{array}
	\]
	
	Set
	\[
	U_0=\{[1:u:v]\mid (u,v)\in\A^2\},
	\qquad
	U_2=\{[v:u:1]\mid (u,v)\in\A^2\}.
	\]
	The restriction of this $\P^1$-bundle to $U_0\cup U_2$ is obtained
	by gluing $\P^1\times U_0$ and $\P^1\times U_2$ over
	$U_0\cap U_2$ via
	\[
	\begin{aligned}
		\nu_{02}\colon\quad
		\P^1\times(U_0\cap U_2)
		&\longrightarrow
		\P^1\times(U_0\cap U_2),\\
		([x_0:x_1],[1:u:v])
		&\longmapsto
		\left(
		[\beta_{11}(u,v)x_0+\beta_{12}(u,v)x_1:
		\beta_{21}(u,v)x_0+\beta_{22}(u,v)x_1],
		[v^{-1}:uv^{-1}:1]
		\right),
	\end{aligned}
	\]
	where $\beta_{ij}\in\k[u,v]$ are uniquely determined by
	\[
	\begin{bmatrix}
		\beta_{11}(y+z,yz)&\beta_{12}(y+z,yz)\\
		\beta_{21}(y+z,yz)&\beta_{22}(y+z,yz)
	\end{bmatrix}
	=
	\begin{bmatrix}
		y^{b+l+1}+z^{b+l+1}
		&
		(y-z)^{2(l+1)}Q(y,z)
		\\
		(yz)^{l+1}Q(y,z)
		&
		(yz)^{l+1}(y^{b+l+1}+z^{b+l+1})
	\end{bmatrix},
	\]
	with
	\[
	Q(y,z)=\sum_{i=0}^b
	\binom{l+i}{l}y^{b-i}z^i.
	\]
	All four entries are symmetric polynomials in $y,z$, and
	\[
	\det(\beta_{ij}(u,v))=4v^c.
	\]
	In particular, this matrix is invertible on $U_0\cap U_2$.
\end{proposition}

\begin{proof}
	Write
	\[
	h=l+1,\qquad m=b+l+1,
	\]
	so that $h$ is odd and $c=m+h$.
	By Lemma~\ref{EqualityQyz}, we have
	\[
	(y-z)^hQ(y,z)=y^m-z^m.
	\]
	Since $h$ is odd, this identity also gives
	$Q(z,y)=Q(y,z)$.
	
	Consider the two affine open subsets of $\F_0$ parametrized by
	\[
	(y,z)\longmapsto[1:y;1:z],
	\qquad
	(y,z)\longmapsto[y:1;z:1].
	\]
	The $\P^1$-bundle $\T^{b,c}\to\F_0$ is trivial over these open
	subsets, with transition function
	\[
	\nu_{b,c}([x_0:x_1],(y,z))
	=
	\left(
	[y^bx_0:z^hQ(y,z)x_0+z^cx_1],
	(y^{-1},z^{-1})
	\right).
	\]
	By Lemma~\ref{TbcInvolutionExplicit}, the involution
	$\sigma$ is given on these charts by
	\[
	\begin{aligned}
		\sigma_1([x_0:x_1],(y,z))
		&=
		([x_0-(y-z)^hx_1:x_1],(z,y)),\\
		\sigma_2([x_0:x_1],(y,z))
		&=
		([x_0+(y-z)^hx_1:x_1],(z,y)).
	\end{aligned}
	\]
	
	Define $\tau_1,\tau_2\in\Aut(\P^1\times\A^2)$ by
	\[
	\begin{aligned}
		\tau_1([x_0:x_1],(y,z))
		&=
		([x_0+(y-z)^hx_1:2x_1],(y,z)),\\
		\tau_2([x_0:x_1],(y,z))
		&=
		([x_0-(y-z)^hx_1:2x_1],(y,z)).
	\end{aligned}
	\]
	A direct computation gives
	\[
	\tau_1^{-1}\sigma_1\tau_1
	=
	\tau_2^{-1}\sigma_2\tau_2
	=
	\hat{\sigma},
	\qquad
	\hat{\sigma}([x_0:x_1],(y,z))
	=
	([x_0:x_1],(z,y)).
	\]
	Thus the quotient on each chart is $\P^1\times\A^2$,
	with quotient morphism
	\[
	([x_0:x_1],(y,z))
	\longmapsto
	([x_0:x_1],(y+z,yz)).
	\]
	On the base, these are precisely the restrictions of
	$\kappa$ over $U_0$ and $U_2$.
	Consequently, the induced morphism
	\[
	q\colon \T^{b,c}/\langle\sigma\rangle\longrightarrow\P^2
	\]
	is a trivial $\P^1$-bundle over each of these two open subsets.
	
	We now compute the transition function in these
	trivializations. Put $d=(y-z)^h$.
	On the overlap, we have
	\[
	(y^{-1}-z^{-1})^h=-\frac{d}{(yz)^h}.
	\]
	Hence the fiber matrix of
	$\hat{\nu}=\tau_2^{-1}\nu_{b,c}\tau_1$ is
	\[
	\begin{bmatrix}
		1&-\dfrac{d}{2(yz)^h}\\
		0&\dfrac12
	\end{bmatrix}
	\begin{bmatrix}
		y^b&0\\
		z^hQ(y,z)&z^c
	\end{bmatrix}
	\begin{bmatrix}
		1&d\\
		0&2
	\end{bmatrix}
	=
	\frac{1}{2y^h}
	\begin{bmatrix}
		y^m+z^m&d^2Q(y,z)\\
		(yz)^hQ(y,z)&(yz)^h(y^m+z^m)
	\end{bmatrix}.
	\]
	Here we used $dQ(y,z)=y^m-z^m$.
	Since $2y^h$ is invertible on the overlap, the matrix
	\[
	M(y,z)=
	\begin{bmatrix}
		y^m+z^m&d^2Q(y,z)\\
		(yz)^hQ(y,z)&(yz)^h(y^m+z^m)
	\end{bmatrix}
	\]
	represents the same projective transformation.
	
	All entries of $M(y,z)$ are symmetric polynomials.
	Since $\k[y,z]^{\mathfrak S_2}=\k[y+z,yz]$, they uniquely
	determine polynomials $\beta_{ij}(u,v)\in\k[u,v]$.
	Moreover,
	\[
	\begin{aligned}
		\det M(y,z)
		&=(yz)^h
		\bigl((y^m+z^m)^2-d^2Q(y,z)^2\bigr)\\
		&=(yz)^h
		\bigl((y^m+z^m)^2-(y^m-z^m)^2\bigr)\\
		&=4(yz)^{m+h}
		=4(yz)^c.
	\end{aligned}
	\]
	Therefore $\det(\beta_{ij}(u,v))=4v^c$.
	The inversion $(y,z)\mapsto(y^{-1},z^{-1})$ induces
	\[
	(u,v)\longmapsto(u/v,1/v),
	\]
	which gives the stated transition function.
	
	Finally, we verify local triviality over all of $\P^2$.
	The diagonal $\PGL_2$-action on $\F_0$ lifts to
	$\T^{b,c}$ and commutes with $\sigma$, by the previously
	established properties of $\T^{b,c}$.
	It therefore descends to the quotient, making $q$
	equivariant for the action on $\P^2$ induced by the diagonal action
	on $\F_0$ via $\kappa$.
	The $\PGL_2$-translates of $U_0$ cover $\P^2$:
	any unordered pair of points of $\P^1$ can be moved
	into the affine chart where the first homogeneous
	coordinate is nonzero.
	Transporting the trivialization over $U_0$ by this
	action yields a trivialization over each such translate.
	Thus $q$ is a $\P^1$-bundle over all of $\P^2$.
\end{proof}

\begin{remark}
	If $l=0$, then $c=b+2$, and the transition matrix
	$M_b$ of Proposition~\ref{prop: def of the Sbc} satisfies
	\[
	\begin{bmatrix}
		-\frac12 & \frac{y+z}{2yz}\\
		0&1
	\end{bmatrix}
	M_b
	\begin{bmatrix}
		1&\frac{y+z}{2}\\
		0&-\frac12
	\end{bmatrix}
	=
	\frac{yz}{y-z}
	\begin{bmatrix}
		y^b-z^b & yz(y^{b-1}-z^{b-1})\\
		y^{b+1}-z^{b+1} & yz(y^b-z^b)
	\end{bmatrix}.
	\]
	The left and right factors define changes of trivialization
	over $U_2$ and $U_0$, respectively. Up to the invertible
	scalar $yz$, the resulting matrix is that of
	\cite[Lemma~4.2.1]{BFT}, where the quotients by $y-z$
	are polynomials.
	Thus the two $\P^1$-bundles are isomorphic away from $[0:1:0]$.
	This isomorphism extends uniquely across the missing point
	by \cite[Lemma~2.3.1]{BFT}.
	 Hence
	$\T^{b,b+2}/\langle\sigma\rangle\simeq\SS_b$
	as $\P^1$-bundles over $\P^2$.
\end{remark}

\begin{corollary}
	Let $b,c,l$ be as in Lemma~$\ref{TbcInvolutionExplicit}.$
	The morphism
	\[
	\T^{b,c}/\langle\sigma\rangle\longrightarrow\P^2
	\]
	induced by the quotient
	$\kappa\colon\F_0\to\P^2$
	is a $\P^1$-bundle if and only if $p=2$ or $l$ is even.
\end{corollary}

\begin{proof}
	For $p\geq3$, this follows from
	Lemma~\ref{TbcInvolutionExplicit} and
	Proposition~\ref{prop: def of the Sbc}.
	
	For $p=2$, the changes of trivialization in the proof of
	Lemma~\ref{TbcInvolutionExplicit} give local trivializations
	over the two standard symmetric charts.
	The $\SL_2$-action on $\T^{b,c}$ commutes with $\sigma$,
	so it descends to the quotient, and the induced morphism
	to $\P^2$ is equivariant.
	Every effective divisor of degree $2$ on $\P^1$
	is contained in the complement of some point.
	Hence the $\SL_2$-translates of these charts cover
	$\operatorname{Sym}^2(\P^1)\simeq\P^2$.
	Transporting the local trivializations by this action
	proves that the induced morphism is a $\P^1$-bundle.
\end{proof}

\subsection{The Umemura \texorpdfstring{$\P^1$}{P1}-bundles over \texorpdfstring{$\F_a$}{FA}}\label{Sec:Umemurabundles}
In this section, we study a family of indecomposable $\P^1$-bundles over $\F_a$ ($a \geq 1$) whose automorphism group acts transitively on the base $\P^1$, the so-called \emph{Umemura $\P^1$-bundles} (see Definition \ref{def Umemura bundle}).

\smallskip

Let $a \geq 1$, $b \geq 1$, and $c \geq 2$. We define 
\[\Omega_{a}^{b,c}:=\P\left(\bigoplus_{i=0}^{b} y_0^i y_1^{b-i} \cdot \k[z]^{\SL_2}_{\leq c-2-ai}\right) \subseteq  \M_{a}^{b,c}\] 
which is a (possibly empty) linear subspace of $\M_{a}^{b,c}$.

Since $\k[z]_{\leq c-2-ai} \simeq V_{c-2-ai}^{*}$ as $\SL_2$-modules, it follows from Corollary~\ref{cor:case n=0} that
\[
\dim(\k[z]_{\leq c-2-ai}^{\SL_2})= \dim((V_{c-2-ai}^{*})^{\SL_2}) \leq 1.
\] 
However, the dimension of $\Omega_{a}^{b,c}$ can be arbitrarily large, depending on $a$, $b$, and $c$. We now give some examples. Fix integers $a, b, c$, and consider the integers $i \in \{0, \ldots, b\}$ such that $\dim\bigl((V_{c-2-ai}^{*})^{\SL_2}\bigr) = 1$. These integers will be called the \emph{Umemura indices}. By Corollary~\ref{cor:case n=0}, this is equivalent to requiring $c - a i = 2 p^s$ for some integer $s \ge 0$.

\begin{example}\label{example: Omegaabc not bounded}
Fix an arbitrary integer $N \ge 1$. 
\begin{enumerate}
\item
Take $a = 1$, $b \ge 2 p^N - 2$, and $c = 2 p^N$. The Umemura indices are the integers $i \in \{0, \ldots, b\}$ such that $i = c - 2 p^s$. The possible values for $s$ are in $\{0, \ldots, N\}$, so $\dim(\Omega_a^{b,c}) = N$.

\item 
Take $a = 2$, $b \ge p^N - 1$, and $c = 2 p^N$. The indices are 
\[
i \in \{0, \ldots, b\} \quad \text{with} \quad i = \frac{c - 2 p^s}{2} = p^N - p^s,
\]
where $s \in \{0, \ldots, N\}$. Again, $\dim(\Omega_a^{b,c}) = N$.

\item 
Take $a=2p$, $b\geq p^N-1$, and $c=2p^{N+1}$.
The Umemura indices are
\[
i=\frac{c-2p^s}{2p}=p^N-p^{s-1},
\qquad s\in\{1,\ldots,N+1\}.
\]
Thus $\dim(\Omega_a^{b,c})=N$.
\end{enumerate}
\end{example}

\smallskip

Assume now that $\Omega_a^{b,c}\neq\varnothing$, and denote by
$I_{a,b,c}$ the set of Umemura indices:
\[
I_{a,b,c}=\{i\in\{0,\ldots,b\}\mid c-ai=2p^s
\text{ for some integer }s\geq0\}.
\]
By Corollary~\ref{cor:case n=0}, for each $i\in I_{a,b,c}$ there is
a unique integer $s_i\geq0$ such that $c-ai=2p^{s_i}$. Set $l_i=p^{s_i}-1$, so that
\begin{equation}\label{eq: relation betweenc and lj}
	c-2-ai=2l_i,
	\qquad l_i=p^{s_i}-1.
\end{equation}
Then $\k[z]_{\leq c-2-ai}^{\SL_2}$ is generated by $z^{l_i}$.
The linear isomorphism
\[
\k^{I_{a,b,c}}\longrightarrow
\bigoplus_{i\in I_{a,b,c}}
y_0^iy_1^{b-i}\cdot\k[z]_{\leq c-2-ai}^{\SL_2},
\qquad
(\lambda_i)_{i\in I_{a,b,c}}\longmapsto
\sum_{i\in I_{a,b,c}}\lambda_i y_0^iy_1^{b-i}z^{l_i}
\]
induces an identification
$\P(\k^{I_{a,b,c}})\simeq\Omega_a^{b,c}$, which we use throughout.
For each $i\in I_{a,b,c}$, denote by $e_i$ the coordinate point
whose only nonzero coordinate is the one indexed by $i$.

\begin{definition}[Umemura $\P^1$-bundles]  \label{def Umemura bundle}
Let $a,b\ge 1, c\ge 2$ such that $\Omega_{a}^{b,c} \neq \varnothing$.
We keep the previous notation. Let $\omega=[\lambda_i]_{i\in I_{a,b,c}} \in \Omega_{a}^{b,c}$.
We call \emph{Umemura $\P^1$-bundle} of type $(a,b,c,\omega)$ the $\P^1$-bundle
defined by 
\[\U_{a,\omega}^{b,c}:=Z_{a}^{b,c,P}\to \F_a \ \text{ with }\ P(y_0,y_1,z)=\sum_{i\in I_{a,b,c}} y_0^{i}y_1^{b-i}\lambda_i z^{c-1-l_i}.\] 
As $c-1-l_i=c-p^{s_i}$, the gluing function
$\nu_{c,P}\in \Aut(\F_b \times (\A^1 \setminus \{0\}))$ is given by 
\begin{center}$
([x_0:x_1;y_0:y_1],z) \mapsto ([x_0:x_1z^{c}+x_0\sum_{i\in I_{a,b,c}} y_0^{i}y_1^{b-i}\lambda_i z^{c-p^{s_i}};y_0z^a:y_1],\frac{1}{z}).$
\end{center} 
For a given parameter $\omega$, we write
\[
I=\{i\in I_{a,b,c}\mid\lambda_i\neq0\}
\]
for its support. Thus $I$ is nonempty, and the sums defining $P$
and the gluing function may equivalently be taken over $I$.
In the notation of \eqref{eq:def of P}, we have
\[
P_i(z)=
\begin{cases}
\lambda_i z^{l_i},&i\in I_{a,b,c},\\
0,&i\notin I_{a,b,c}.
\end{cases}
\]
\end{definition}

\begin{remark}
In characteristic zero, $\Omega_{a}^{b,c}$ is either empty or a point, while in positive characteristic the projective space $\Omega_{a}^{b,c}$ can have arbitrarily large dimension, depending on $(a,b,c)$.
Hence, for a given triplet $(a,b,c)$, there is at most one Umemura $\P^1$-bundle (up to isomorphism over $\F_a$) in characteristic zero, which is generally far from being the case in positive characteristic.
\end{remark}

\begin{proposition}\label{ActionSL2GL2Umemura}
Let $a,b\ge 1$, $c\ge 2$, and $\omega=[\lambda_i]_{i\in I_{a,b,c}}\in\Omega_a^{b,c}\not=\varnothing$. 
\begin{enumerate}\item
\label{SL2actionUmemura}
There is an $\SL_2$-action on the Umemura $\P^1$-bundle
\[\U_{a,\omega}^{b,c}=Z_{a}^{b,c,P}\to \F_a \ \text{ with }\ P(y_0,y_1,z)=\sum_{i\in I_{a,b,c}} y_0^{i}y_1^{b-i}\lambda_i z^{c-p^{s_i}}.\] 
The action of $ \begin{bsmallmatrix} \alpha & \beta \\ \gamma & \delta\end{bsmallmatrix}\in\SL_2$ 
is given on the two charts by 
{\small \[ \hspace{-15mm} \begin{array}{llcccll}
\theta_1\colon ([x_0:x_1;y_0:y_1],z)&\mapsto& ([ x_0:x_1 (\alpha+\beta z)^{c}+x_0\sum\limits_{i\in I_{a,b,c}} y_0^{i}y_1^{b-i}\lambda_i\beta^{p^{s_i}}( \alpha +\beta z)^{c-p^{s_i}};y_0(\alpha +\beta z)^a:y_1],  \frac{ \gamma+\delta z}{\alpha +\beta z} ),\\
\theta_2\colon ([x_0:x_1;y_0:y_1],z)&\mapsto& ([ x_0:x_1 (\gamma z +\delta)^{c}-x_0\sum\limits_{i\in I_{a,b,c}} y_0^{i}y_1^{b-i}\lambda_i\gamma^{p^{s_i}}( \gamma z +\delta)^{c-p^{s_i}};y_0(\gamma z +\delta)^a:y_1],  \frac{ \alpha z +\beta }{\gamma z +\delta} ).
\end{array}
\]}  
\item\label{GL2actionUmemura}
If moreover only one of the coefficients $\lambda_i$ is nonzero, say $\lambda_i$, then the whole group $\GL_2$ acts on $\U_{a,\omega}^{b,c}$. The action of $\begin{bsmallmatrix} \alpha & 0 \\ 0 & \delta \end{bsmallmatrix}\in \GL_2$ is given on the two charts by
\[
\begin{array}{rcl} {
\tilde{\theta}_1}\colon ([x_0:x_1;y_0:y_1],z)
&\mapsto&
([x_0:x_1 \frac{\alpha^{c-p^{s_i}}}{\delta^{p^{s_i}}};\;y_0\alpha^a:y_1],\;\tfrac{\delta}{\alpha}z),\\[6pt]
{\tilde{\theta}_2}\colon ([x_0:x_1;y_0:y_1],z)
&\mapsto&
([x_0:x_1\frac{\delta^{c-p^{s_i}}}{\alpha^{p^{s_i}}};\;y_0\delta^a:y_1],\;\tfrac{\alpha}{\delta}z).
\end{array}
\]
In both cases, the action of $\SL_2$ $($respectively $\GL_2)$ is compatible with the standard action on $\F_a$ via the morphism $\Autz(\U_{a,\omega}^{b,c}) \to \Aut(\F_a)$.
\end{enumerate}
\end{proposition}

\begin{proof}
	According to Definition~\ref{Not:ZabcP}, the transition function
	$\nu_{c,P}\in\Aut(\F_b\times(\A^1\setminus\{0\}))$ is defined by
	\[
	\nu_{c,P}\colon([x_0:x_1;y_0:y_1],z)\longmapsto
	\left(
	[x_0:x_1z^c+x_0P(y_0,y_1,z);y_0z^a:y_1],\frac1z
	\right).
	\]
	The $\P^1$-bundle $Z_a^{b,c,P}\to\F_a$ sends
	$([x_0:x_1;y_0:y_1],z)\in\F_b\times\A^1$ to
	$[y_0:y_1;1:z]\in\F_a$ on the first chart and to
	$[y_0:y_1;z:1]\in\F_a$ on the second chart.
	
	The group $\GL_2$ acts on $\F_a$ as
	\begin{equation}\label{GL2Faexpl}
		\begin{bsmallmatrix}\alpha&\beta\\\gamma&\delta\end{bsmallmatrix}
		\cdot[y_0:y_1;z_0:z_1]
		=
		[y_0:y_1;\alpha z_0+\beta z_1:\gamma z_0+\delta z_1].
	\end{equation}
	On the two charts this becomes, respectively,
	\[
	\begin{aligned}
		[y_0:y_1;1:z]
		&\longmapsto
		[y_0:y_1;\alpha+\beta z:\gamma+\delta z]\\
		&=
		[(\alpha+\beta z)^ay_0:y_1;
		1:\tfrac{\gamma+\delta z}{\alpha+\beta z}],\\[6pt]
		[y_0:y_1;z:1]
		&\longmapsto
		[y_0:y_1;\alpha z+\beta:\gamma z+\delta]\\
		&=
		[(\gamma z+\delta)^ay_0:y_1;
		\tfrac{\alpha z+\beta}{\gamma z+\delta}:1].
	\end{aligned}
	\]
	
	\ref{SL2actionUmemura}:
	We choose
	$\begin{bsmallmatrix}\alpha&\beta\\\gamma&\delta\end{bsmallmatrix}
	\in\SL_2$
	and verify that the maps $\theta_1,\theta_2$ define an action
	compatible with that on $\F_a$.
	The map $\theta_1$ is defined where $\alpha+\beta z\neq0$,
	and $\theta_2$ where $\gamma z+\delta\neq0$.
	These formulas describe the action when the source and
	target lie in the same chart. We will also check regularity
	where the image lies in the other chart.
	
	We first check that
	\[
	\theta_2\nu_{c,P}=\nu_{c,P}\theta_1.
	\]
	Indeed, the compositions give
	\[
	\begin{aligned}
		\nu_{c,P}(\theta_1([x_0:x_1;y_0:y_1],z))
		&=
		\left(
		[x_0:x_1(\gamma+\delta z)^c+x_0R_1;
		y_0(\gamma+\delta z)^a:y_1],
		\frac{\alpha+\beta z}{\gamma+\delta z}
		\right),\\[6pt]
		\theta_2(\nu_{c,P}([x_0:x_1;y_0:y_1],z))
		&=
		\left(
		[x_0:x_1(\gamma+\delta z)^c+x_0R_2;
		y_0(\gamma+\delta z)^a:y_1],
		\frac{\alpha+\beta z}{\gamma+\delta z}
		\right),
	\end{aligned}
	\]
	where
	\[
	\begin{aligned}
		R_1
		&=
		\left(
		\sum_{i\in I_{a,b,c}} y_0^{i}y_1^{b-i}
		\lambda_i\beta^{p^{s_i}}
		(\alpha+\beta z)^{c-p^{s_i}}
		\right)
		\left(\frac{\gamma+\delta z}{\alpha+\beta z}\right)^c\\
		&\qquad+
		P\left(
		y_0(\alpha+\beta z)^a,y_1,
		\frac{\gamma+\delta z}{\alpha+\beta z}
		\right),\\[6pt]
		R_2
		&=
		\left(\frac{\gamma}{z}+\delta\right)^cP(y_0,y_1,z)
		-
		\sum_{i\in I_{a,b,c}}
		(y_0z^a)^{i}y_1^{b-i}
		\lambda_i\gamma^{p^{s_i}}
		\left(\frac{\gamma}{z}+\delta\right)^{c-p^{s_i}}.
	\end{aligned}
	\]
	
	Recall that
	\[
	P(y_0,y_1,z)
	=
	\sum_{i\in I_{a,b,c}}
	y_0^{i}y_1^{b-i}\lambda_i z^{c-p^{s_i}}.
	\]
	Using $ai=c-2p^{s_i}$ and $\alpha\delta=\beta\gamma+1$,
	the coefficient of $y_0^{i}y_1^{b-i}$ in $R_1$ is
	\[
	\begin{aligned}
		&\lambda_i\left(
		\beta^{p^{s_i}}(\alpha+\beta z)^{-p^{s_i}}
		(\gamma+\delta z)^c
		+
		(\alpha+\beta z)^{ai}
		\left(\frac{\gamma+\delta z}{\alpha+\beta z}\right)^{c-p^{s_i}}
		\right)\\
		&\qquad=
		\lambda_i(\alpha+\beta z)^{-p^{s_i}}
		(\gamma+\delta z)^{c-p^{s_i}}
		\left(\beta^{p^{s_i}}(\gamma+\delta z)^{p^{s_i}}+1\right)\\
		&\qquad=
		\lambda_i\delta^{p^{s_i}}
		(\gamma+\delta z)^{c-p^{s_i}}.
	\end{aligned}
	\]
	Similarly, the coefficient of $y_0^{i}y_1^{b-i}$ in $R_2$ is
	\[
	\begin{aligned}
		&\lambda_i\left(
		\left(\frac{\gamma}{z}+\delta\right)^cz^{c-p^{s_i}}
		-
		z^{ai}\gamma^{p^{s_i}}
		\left(\frac{\gamma}{z}+\delta\right)^{c-p^{s_i}}
		\right)\\
		&\qquad=
		\lambda_i z^{-p^{s_i}}(\gamma+\delta z)^{c-p^{s_i}}
		\left((\gamma+\delta z)^{p^{s_i}}-\gamma^{p^{s_i}}\right)\\
		&\qquad=
		\lambda_i\delta^{p^{s_i}}
		(\gamma+\delta z)^{c-p^{s_i}}.
	\end{aligned}
	\]
	Thus $\theta_2\nu_{c,P}=\nu_{c,P}\theta_1$.
	
	To check regularity everywhere, put
	\[
	A=\alpha+\beta z,\qquad B=\gamma+\delta z.
	\]
	For a point in the first chart with $B\neq0$, the image
	in the second chart is
	\[
	\left(
	\left[
	x_0:
	x_1B^c+x_0\sum_{i\in I_{a,b,c}}
	y_0^{i}y_1^{b-i}\lambda_i
	\delta^{p^{s_i}}B^{c-p^{s_i}};
	y_0B^a:y_1
	\right],
	\frac AB
	\right).
	\]
	Similarly, putting $E=\alpha z+\beta$ and $C=\gamma z+\delta$,
	the image of a point in the second chart with $E\neq0$
	is given in the first chart by
	\[
	\left(
	\left[
	x_0:
	x_1E^c-x_0\sum_{i\in I_{a,b,c}}
	y_0^{i}y_1^{b-i}\lambda_i
	\alpha^{p^{s_i}}E^{c-p^{s_i}};
	y_0E^a:y_1
	\right],
	\frac CE
	\right).
	\]
	These expressions follow by composing with $\nu_{c,P}$
	or its inverse.
	Since $A,B$ cannot vanish simultaneously, and neither
	can $C,E$, the four formulas cover all points.
	Moreover, $c-p^{s_i}=ai+p^{s_i}\geq1$.
	Thus they define a morphism
	$\SL_2\times\U_{a,\omega}^{b,c}\to\U_{a,\omega}^{b,c}$.
	
	It remains to check that, for
	\[
	g=\begin{bsmallmatrix}\alpha&\beta\\\gamma&\delta\end{bsmallmatrix},
	\qquad
	g'=\begin{bsmallmatrix}\alpha'&\beta'\\\gamma'&\delta'\end{bsmallmatrix}
	\in\SL_2,
	\]
	the composition $\theta_1\circ\theta_1'$ equals
	$\theta_1''$ associated with the product
	\[
	gg'
	=
	\begin{bsmallmatrix}
		\alpha\alpha'+\beta\gamma'&\alpha\beta'+\beta\delta'\\
		\gamma\alpha'+\delta\gamma'&\gamma\beta'+\delta\delta'
	\end{bsmallmatrix}
	=
	\begin{bsmallmatrix}
		\alpha''&\beta''\\\gamma''&\delta''
	\end{bsmallmatrix}.
	\]
	The last coordinates of
	$\theta_1(\theta_1'([x_0:x_1;y_0:y_1],z))$ are
	\[
	\begin{aligned}
		&\left(
		\left[\cdots;
		y_0(\alpha'+\beta'z)^a
		\left(\alpha+\beta\frac{\gamma'+\delta'z}{\alpha'+\beta'z}\right)^a
		:y_1\right],
		\frac{\gamma+\delta\frac{\gamma'+\delta'z}{\alpha'+\beta'z}}
		{\alpha+\beta\frac{\gamma'+\delta'z}{\alpha'+\beta'z}}
		\right)\\
		&\qquad=
		\left(
		[\cdots;y_0(\alpha''+\beta''z)^a:y_1],
		\frac{\gamma''+\delta''z}{\alpha''+\beta''z}
		\right).
	\end{aligned}
	\]
	The first coordinate is $x_0$, and the second one is
	\[
	\begin{aligned}
		&
		\left(
		x_1(\alpha'+\beta'z)^c
		+x_0\sum_{i\in I_{a,b,c}} y_0^{i}y_1^{b-i}
		\lambda_i(\beta')^{p^{s_i}}
		(\alpha'+\beta'z)^{c-p^{s_i}}
		\right)
		\left(
		\alpha+\beta\frac{\gamma'+\delta'z}{\alpha'+\beta'z}
		\right)^c\\
		&\quad+
		x_0\sum_{i\in I_{a,b,c}}
		\bigl(y_0(\alpha'+\beta'z)^a\bigr)^{i}
		y_1^{b-i}\lambda_i\beta^{p^{s_i}}
		\left(
		\alpha+\beta\frac{\gamma'+\delta'z}{\alpha'+\beta'z}
		\right)^{c-p^{s_i}}\\
		&\qquad=
		x_1(\alpha''+\beta''z)^c
		+x_0\sum_{i\in I_{a,b,c}}
		y_0^{i}y_1^{b-i}\lambda_i\Delta_i,
	\end{aligned}
	\]
	where, using $ai=c-2p^{s_i}$,
	\[
	\begin{aligned}
		\Delta_i
		&=
		(\beta')^{p^{s_i}}(\alpha'+\beta'z)^{c-p^{s_i}}
		\left(
		\alpha+\beta\frac{\gamma'+\delta'z}{\alpha'+\beta'z}
		\right)^c\\
		&\quad+
		\beta^{p^{s_i}}(\alpha'+\beta'z)^{ai}
		\left(
		\alpha+\beta\frac{\gamma'+\delta'z}{\alpha'+\beta'z}
		\right)^{c-p^{s_i}}\\
		&=
		(\alpha'+\beta'z)^{-p^{s_i}}
		(\alpha''+\beta''z)^{c-p^{s_i}}
		\left(
		(\beta')^{p^{s_i}}(\alpha''+\beta''z)^{p^{s_i}}
		+\beta^{p^{s_i}}
		\right)\\
		&=
		(\beta'')^{p^{s_i}}
		(\alpha''+\beta''z)^{c-p^{s_i}}.
	\end{aligned}
	\]
	For the last equality, we used
	\[
	\beta'(\alpha''+\beta''z)+\beta
	=
	\beta''(\alpha'+\beta'z),
	\]
	which follows from $\alpha'\delta'-\beta'\gamma'=1$.
	Hence $\theta_1\circ\theta_1'=\theta_1''$ on their common
	domain. By regularity, the group law holds everywhere.
	Since the identity matrix acts trivially, these formulas
	define an $\SL_2$-action.
	
\ref{GL2actionUmemura}:
We now assume that only one coefficient, say $\lambda_i$, is nonzero.
To extend the $\SL_2$-action to $\GL_2$, put $q=p^{s_i}$.
For
$g=\begin{bsmallmatrix}\alpha&\beta\\\gamma&\delta\end{bsmallmatrix}
\in\GL_2$,
set $D=\det(g)$ and modify the formulas for
$\theta_1,\theta_2$ by multiplying the entire second
fiber coordinate by $D^{-q}$.
For instance, the first formula becomes
\[ \hspace{-20mm}
([x_0:x_1;y_0:y_1],z)\longmapsto
\left(
\left[
x_0:
D^{-q}\left(
x_1(\alpha+\beta z)^c+
x_0y_0^{i}y_1^{b-i}\lambda_i
\beta^q(\alpha+\beta z)^{c-q}
\right);
y_0(\alpha+\beta z)^a:y_1
\right],
\frac{\gamma+\delta z}{\alpha+\beta z}
\right).
\]
The compatibility and regularity checks above apply with
\[
\beta(\gamma+\delta z)+D
=
\delta(\alpha+\beta z),
\qquad
\gamma(\alpha z+\beta)+D
=
\alpha(\gamma z+\delta).
\]
Likewise, the group-law calculation applies using
\[
\beta'(\alpha''+\beta''z)+\beta\det(g')
=
\beta''(\alpha'+\beta'z)
\]
and $\det(gg')=\det(g)\det(g')$.
Thus these formulas define a $\GL_2$-action extending
the $\SL_2$-action.
For diagonal matrices, they recover
$\tilde{\theta}_1,\tilde{\theta}_2$, since
\[
D^{-q}\alpha^c=\frac{\alpha^{c-q}}{\delta^q},
\qquad
D^{-q}\delta^c=\frac{\delta^{c-q}}{\alpha^q}.
\]
The action is compatible with the standard action
on $\F_a$ given by \eqref{GL2Faexpl}.
\end{proof}

\begin{lemma}\label{lem:characterization-Umemura}
	Let $a\geq1$, and let $\pi\colon X\to\F_a$ be an
	indecomposable $\P^1$-bundle.
	Then $\pi$ is isomorphic over $\F_a$ to an Umemura
	$\P^1$-bundle if and only if
	\[
	\Im\bigl(\Autz(X)\to\Aut(\F_a)\bigr)
	\supseteq
	\Im\bigl(\SL_2\to\GL_2/\mu_a\bigr).
	\]
	In particular, the neutral automorphism group of every
	Umemura $\P^1$-bundle acts transitively on the base
	$\P^1$ of the ruling $\F_a\to\P^1$.
\end{lemma}

\begin{proof}
	For an Umemura $\P^1$-bundle, the inclusion follows
	from Proposition~\ref{ActionSL2GL2Umemura}
	\ref{SL2actionUmemura}.
	
	Conversely, suppose that the inclusion holds.
	The induced action on the base of the ruling is
	transitive, so $\pi$ has no jumping fibers.
	Since $\pi$ is indecomposable, the results of
	Section~\ref{sec:moduli P1 bundles over Fa} imply that
	it has numerical invariants $(a,b,c)$ with $b\geq1$
	and $c\geq2$.
	By Proposition~\ref{prop: paramter space for non-dceomposable P1 bundles},
	it therefore determines a point of $\M_a^{b,c}$.
	Lemma~\ref{Lemm:ActionAutFaParSpace}\ref{ActionGL2}
	shows that this point is fixed by $\SL_2$.
	Since $\SL_2$ has no nontrivial characters, every
	$\SL_2$-stable line in a rational representation is
	fixed pointwise.
	Thus the parameter belongs to $\Omega_a^{b,c}$,
	and $\pi$ is an Umemura $\P^1$-bundle.
\end{proof} 

\begin{remark} \label{rk: GL2-action on Umemura}
	Let $X \to \F_a$ be an Umemura $\P^1$-bundle. Then $\Im(\Autz(X) \to \Aut(\F_a))$ contains $\GL_2/\mu_a$ if and only if all $P_i$ vanish except for a single index $i_0\in \llbracket 0, b \rrbracket$. Indeed, by Lemma~\ref{Lemm:ActionAutFaParSpace}\ref{ActionGL2}, a scalar matrix $\lambda I_2 \in \GL_2$ acts on $P$ by sending $P_i$ to $\lambda^{c-2-ai}P_i$.
	Since $a\geq1$, these weights are pairwise distinct.
		Thus a line in the invariant subspace is fixed by
		the scalar matrices if and only if it is a coordinate
		line.
\end{remark}

\begin{remark} \label{rk aut vert Umemura bundles} 
\begin{enumerate}
\item With the notation of Definition~\ref{def Umemura bundle}, the group of automorphisms over $\F_a$ of
$\U_{a,\omega}^{b,c}$ identifies with the vector group $\bigoplus_{i=1}^{b} \k[z_0,z_1]_{ai-c}$ whose action on $\U_{a,\omega}^{b,c}$ can be described on the first chart by the following regular action 
\[\begin{array}{ccc}
\F_b\times \A^1& \to & \F_b\times \A^1\\
([x_0:x_1; y_0:y_1],z)&\mapsto& \left([x_0:x_1+x_0 \sum_{i=1}^{b} y_0^i y_1^{b-i} q_i(z,1); y_0:y_1],z\right),\end{array}\]
where $(q_i)_{i=1,\ldots,b} \in \bigoplus_{i=1}^{b} \k[z_0,z_1]_{ai-c}$.
This follows from Lemma~\ref{AutZaF1}.
The action on $\pi^{-1}(s_{-a})$, which corresponds to $y_0=0$ on both charts, is then trivial.
Furthermore, the group is nontrivial if and only if $ab\ge c$. In that case, the action on $\pi^{-1}(s_{a})$, which corresponds to $y_1=0$ on both charts, is not trivial.

\smallskip

\item 
The automorphism group $\Autz(\U_{a,\omega}^{b,c})$ fits into a short exact sequence
\[
1 \longrightarrow \bigoplus_{i=1}^{b} \k[z_0,z_1]_{ai-c}
\longrightarrow \Autz(\U_{a,\omega}^{b,c})
\longrightarrow H \longrightarrow 1,
\]
where
\[
H=\Im\bigl(\Autz(\U_{a,\omega}^{b,c})\to\Aut(\F_a)\bigr) \subseteq  \Aut(\F_a) \simeq  \k[z_0,z_1]_a \rtimes \GL_2/\mu_a
\]
is a connected algebraic subgroup containing the standard copy of $\Im(\SL_2 \to \GL_2/\mu_a)$.

\end{enumerate}
\end{remark}

\subsection{Classification of the \texorpdfstring{$\P^1$}{P1}-bundles over Hirzebruch surfaces}\label{sec: classification over Fa}
Before turning to $\P^1$-bundles over $\P^2$, we summarize the different
possibilities for $\P^1$-bundles over Hirzebruch surfaces $\F_a$,
with $a\geq 0$, that will be relevant in the sequel. The main result of this section is Theorem~\ref{th:P1bundlesFa}.

\smallskip

Recall that two $\P^1$-bundles $\pi\colon X\to\F_a$ and
$\pi'\colon X'\to\F_a$ are \emph{square-isomorphic} if there exist
isomorphisms $\varphi\colon X\to X'$ and $\psi\in\Aut(\F_a)$ such that
\[
\pi'\circ\varphi=\psi\circ\pi.
\]
We first introduce a normal form for indecomposable $\P^1$-bundles whose
automorphism group acts transitively on the base of the ruling.

\begin{definition}
	Let $a\geq1$. An indecomposable $\P^1$-bundle $X\to\F_a$ is said to be
	\emph{monomial} if it is
	$	Z_a^{b,c,P}\to\F_a$ 
	for some $b\geq1$ and $c\geq2$, where
	\[
	P(y_0,y_1,z)
	=
	\sum_{i=0}^b y_0^iy_1^{b-i}P_i(z)z^{ai+1},
	\qquad
	P_i\in\k[z]_{\leq c-2-ai},
	\]
	and where, for each $i\in\{0,\ldots,b\}$, either $P_i=0$, or
	$c-2-ai$ is even and
	$P_i=\lambda_i z^{(c-2-ai)/2}$ for some $\lambda_i\in\k^*$.
\end{definition}

\begin{lemma}\label{lemmaMonom}
	Let $\pi\colon X\to\F_a$ be an indecomposable $\P^1$-bundle
	with $a\geq1$ and no jumping fibers. Write
	\[
	H=\Im\bigl(\Autz(X)\to\Aut(\F_a)\bigr)
	\subseteq\k[z_0,z_1]_a\rtimes(\GL_2/\mu_a),
	\qquad
	S=\Im\bigl(\SL_2\to\GL_2/\mu_a\bigr),
	\]
	and assume that the projection of $H$ to $\GL_2/\mu_a$
	contains $S$. Then the following are equivalent:
	\begin{enumerate}
		\item
		The $\P^1$-bundle $\pi\colon X\to\F_a$ is isomorphic,
		over $\F_a$, to a monomial $\P^1$-bundle.
		\item
		The diagonal torus $T\subseteq S$ is contained in $H$.
	\end{enumerate}
\end{lemma}

\begin{proof}
	If $\pi\colon X=Z_a^{b,c,P}\to\F_a$ is monomial, the action of $T$
	fixes each $P_i$, so $T$ is contained in $H$.
	Conversely, assume that $T\subseteq H$.
	
	As $\pi$ has no jumping fibers, we may apply
	Proposition~\ref{prop: paramter space for non-dceomposable P1 bundles}
	and write $X\simeq Z_a^{b,c,P}\to\F_a$ for some
	$b\geq1$ and $c\geq2$, where
	\[
	P(y_0,y_1,z)
	=
	\sum_{i=0}^b y_0^iy_1^{b-i}P_i(z)z^{ai+1},
	\qquad
	P_i\in\k[z]_{\leq n_i},
	\qquad
	n_i=c-2-ai.
	\]
	Since $P\neq0$, the index
	$i_0=\min\{i\mid P_i\neq0\}$ is well defined.
	
	By hypothesis, for each $\alpha\in S$, there exists
	$R\in\k[z_0,z_1]_a$ such that $(R,\alpha)\in H$.
	This element fixes the point corresponding to $P$ in
	\[
	\M_a^{b,c}
	=
	\P\left(
	\bigoplus_{i=0}^b y_0^iy_1^{b-i}\k[z]_{\leq n_i}
	\right).
	\]
	By Lemma~\ref{Lemm:ActionAutFaParSpace}, the action of $R$
	is induced by replacing $y_1$ with $y_1-y_0R(1,z)$
	and leaves the first nonzero component $P_{i_0}$ unchanged.
	The action of $\alpha$ on this component is the one described
	in Remark~\ref{rk: def of dual Vm}.
	Consequently, the line $\k P_{i_0}$ is $\SL_2$-stable.
	Since $\SL_2$ has no nontrivial characters, $P_{i_0}$ is
	$\SL_2$-invariant, and hence $T$-invariant.
	
	As $T\subseteq H$, it fixes the point corresponding to $P$.
	Since $P_{i_0}$ is $T$-invariant and nonzero, the scalar by which
	each element of $T$ acts on this representative is $1$.
	Thus every $P_i$ is $T$-invariant.
	By Remark~\ref{rk: def of dual Vm}, this means that either $P_i=0$,
	or $n_i$ is even and
	$P_i=\lambda_i z^{n_i/2}$ for some $\lambda_i\in\k^*$.
	Hence $X\to\F_a$ is monomial.
\end{proof}

For a monomial normal form, we now study the action of translations
on its components. This gives a normalization under which $H$ splits
over its image in $\GL_2/\mu_a$.

\begin{proposition}\label{propMonom2}
	Let $\pi\colon X=Z_a^{b,c,P}\to\F_a$ be a monomial
	$\P^1$-bundle given by
	\[
	P(y_0,y_1,z)
	=
	\sum_{i=0}^b y_0^iy_1^{b-i}P_i(z)z^{ai+1},
	\qquad
	P_i\in\k[z]_{\leq n_i},
	\qquad
	n_i=c-2-ai,
	\]
	where each nonzero $P_i$ has the form
	$P_i=\lambda_i z^{n_i/2}$.
	Let
	\[
	H=\Im\bigl(\Autz(X)\to\Aut(\F_a)\bigr),
	\qquad
	G=\mathrm{pr}_{\GL_2/\mu_a}(H),
	\qquad
	S=\Im\bigl(\SL_2\to\GL_2/\mu_a\bigr).
	\]
Assume that the induced morphism $H\to\Aut(\P^1)$ is surjective
and that either $p=2$ or $p$ is odd and $a$ is not a power of $p$.
	
	If $P_i=0$ for every $i<b$, then $H=\Aut(\F_a)$.
	Otherwise, set
	\[
	u=\min\bigl\{i+p^{\nu_p(b-i)}\mid 0\leq i<b,\ P_i\neq0\bigr\}
	\in\{1,\ldots,b\}.
	\]
	
	If $n_u<0$, then $H=\Aut(\F_a)$ or
	$H=\k[z_0,z_1]_a\rtimes S$.
	If $n_u\geq0$, then:
	\begin{enumerate}
		\item\label{LiftS1}
		For every $i<u$ with $P_i\neq0$, the monomial $P_i$ is
		$\SL_2$-invariant; equivalently,
		$n_i=2(p^{s_i}-1)$ for some $s_i\geq0$.
		
		\item\label{LiftS2}
		Let $i_1<i_2<b$ be such that $P_{i_1}$ and $P_{i_2}$ are
		nonzero and $\SL_2$-invariant, and set
		$v_\ell=\nu_p(b-i_\ell)$ for $\ell=1,2$.
		If $n_{i_2+p^{v_2}}\geq0$, then $v_1\leq v_2$.
		
		\item\label{LiftS3}
		There exists a unique $i_0<b$ such that
		$P_{i_0}\neq0$ and $i_0+p^{\nu_p(b-i_0)}=u$.
		Moreover,
		$		i_0=\min\{i\mid P_i\neq0\}$.
		Set $q=u-i_0=p^{\nu_p(b-i_0)}$.
		
		\item\label{LiftS4}
		If $a$ is even, an automorphism of $\F_a$ of the form
		\[
		[y_0:y_1;z_0:z_1]\longmapsto
		[y_0:y_1+\lambda y_0(z_0z_1)^{a/2};z_0:z_1]
		\]
		transforms $P$ into another monomial normal form with the same
		$u$ and with $P_u=0$.
		
		\item\label{LiftS5}
		If $P_u=0$, then $H=K\rtimes G$, where
		\[
		K=
		\left\{
		\sum_{t=0}^a\mu_t z_0^tz_1^{a-t}
		\ \middle|\
		\mu_t=0\text{ whenever }|2t-a|\leq\frac{n_u}{q}
		\right\}
		\subseteq\k[z_0,z_1]_a.
		\]
	\end{enumerate}
\end{proposition}

\begin{proof}
	Write $U=\k[z_0,z_1]_a$, so that
	$\Aut(\F_a)=U\rtimes(\GL_2/\mu_a)$.
	By assumption, $G$ surjects onto $\PGL_2$.
	Lemma~\ref{SL2OverPGL2} implies that $S\subseteq G$;
	since $(\GL_2/\mu_a)/S$ is a one-dimensional torus and $G$ is
	connected, $G$ is either $S$ or $\GL_2/\mu_a$.
	
	The group $H$ is the reduced identity component of the stabilizer
	of the point corresponding to $P$ in
	\[
	\M_a^{b,c}
	=
	\P\left(
	\bigoplus_{i=0}^b y_0^iy_1^{b-i}\k[z]_{\leq n_i}
	\right).
	\]
	Since $P$ is monomial, the diagonal torus $T\subseteq S$
	fixes every $P_i$, and hence $T\subseteq H$.
	
	For each $j$, identify $\k[z]_{\leq n_j}$ with
	\[
	V_j=
	\k[z,z^{-1}]
	\big/
	\bigl(z^{n_j+1}\k[z]+z^{-1}\k[z^{-1}]\bigr),
	\]
	where $V_j=0$ if $n_j<0$.
	We denote the class of a Laurent polynomial $f$ in $V_j$ by $[f]$.
	By Lemma~\ref{Lemm:ActionAutFaParSpace}, the action of
	$-R\in U$ is induced by replacing $y_1$ with $y_1+y_0R(1,z)$.
	Thus its $j$-th component is
	\begin{equation}\tag{$\heartsuit$}
		P_j^R=P_j+\sum_{i=0}^{j-1}\delta_{i,j}(R),
	\end{equation}
	where
	\[
	\delta_{i,j}(R)
	=
	\left[
	\binom{b-i}{j-i}R(1,z)^{j-i}P_i(z)z^{-a(j-i)}
	\right]\in V_j.
	\]
	If $P_i\neq0$, then $P_i=\lambda_i z^{n_i/2}$ and
	$n_i=n_j+a(j-i)$, so
	\begin{equation}\tag{$\spadesuit$}\label{deltaij}
		\delta_{i,j}(R)
		=
		\left[
		\lambda_i\binom{b-i}{j-i}
		R(1,z)^{j-i}z^{(n_j-a(j-i))/2}
		\right].
	\end{equation}
	
	If $P_i=0$ for every $i<b$, then $U$ fixes $P$.
	Since $G$ contains $S$, the line $\k P_b$ is $\SL_2$-stable,
	and hence fixed, as $\SL_2$ has no nontrivial characters.
	Homotheties also preserve this line, so the whole group
	$\Aut(\F_a)$ fixes the point corresponding to $P$.
	Thus $H=\Aut(\F_a)$.
	
	We now assume that $P_i\neq0$ for some $i<b$.
	For each such $i$, set $u_i=i+p^{\nu_p(b-i)}$.
	Lucas's theorem gives
	\[
	\binom{b-i}{j-i}=0\quad\text{in }\k
	\qquad\text{if }i<j<u_i.
	\]
	Consequently, $U$ leaves every component $P_i$ with $i<u$
	unchanged. Since every element of $S$ has a lift to $H$,
	each nonzero line $\k P_i$ with $i<u$ is $\SL_2$-stable,
	and therefore fixed.
	Corollary~\ref{cor:case n=0} proves~\ref{LiftS1}.
	
	If $n_u<0$, then $V_j=0$ for every $j\geq u$.
	Thus $U$ fixes $P$, and $H=U\rtimes G$, giving the two
	possibilities stated above.
	Henceforth assume that $n_u\geq0$.
	
	To prove~\ref{LiftS2}, suppose that $v_2<v_1$.
	Then
	\[
	\nu_p(i_2-i_1)
	=
	\nu_p\bigl((b-i_1)-(b-i_2)\bigr)=v_2.
	\]
	Since $P_{i_1}$ and $P_{i_2}$ are nonzero and invariant, we may
	write $n_{i_\ell}=2(p^{r_\ell}-1)$, for some $r_1,r_2\ge 0$.
	We have $r_1>r_2$ and
	\[
	a(i_2-i_1)=2(p^{r_1}-p^{r_2}).
	\]
	Writing $a=p^ta_1$ with $p\nmid a_1$, comparison of valuations gives
	$t+v_2=r_2$ if $p$ is odd, and $t+v_2=r_2+1$ if $p=2$.
	For $j_2=i_2+p^{v_2}$, it follows that
	\[
	n_{j_2}+2=
	\begin{cases}
		p^{r_2}(2-a_1),&p\text{ odd},\\
		2^{r_2+1}(1-a_1),&p=2.
	\end{cases}
	\]
	Since $n_{j_2}\geq0$, the second case is impossible.
	In the first case, necessarily $a_1=1$, contradicting the
	assumption that $a$ is not a power of $p$. 
	This proves~\ref{LiftS2}.
	
	For~\ref{LiftS3}, suppose that $i_0<i_1$ both attain the minimum $u$.
	Then $i_0,i_1<u$, so both corresponding components are invariant
	by~\ref{LiftS1}. Writing $v_\ell=\nu_p(b-i_\ell)$, the equality
	$i_0+p^{v_0}=i_1+p^{v_1}$ gives $v_0>v_1$, contradicting
	\ref{LiftS2}, since $n_{i_1+p^{v_1}}=n_u\geq0$.
	Thus the minimizer $i_0$ is unique.
	If $i_1<i_0$ and $P_{i_1}\neq0$, then~\ref{LiftS2} gives
	$\nu_p(b-i_1)\leq\nu_p(b-i_0)$, and hence $u_{i_1}<u$,
	a contradiction. This proves~\ref{LiftS3}.
	
	Suppose that $a$ is even and take
	$R=\mu(z_0z_1)^{a/2}$.
	This element commutes with $T$, so its action preserves
	the monomial normal form.
	By uniqueness in~\ref{LiftS3},
	\[
	P_u^R
	=
	P_u+
	\lambda_{i_0}\binom{b-i_0}{q}\mu^qz^{n_u/2}.
	\]
	The binomial coefficient is nonzero in $\k$ by Lucas's theorem.
	Since $\k$ is algebraically closed, we can choose $\mu$ so that
	$P_u^R=0$.
	The components below $u$ are unchanged, so $u$ is unchanged.
	This proves~\ref{LiftS4}.
	
	We now prove~\ref{LiftS5}, assuming $P_u=0$.
	If $P_i\neq0$, $j-i=p^e$, and
	$\binom{b-i}{p^e}\neq0$ in $\k$, then $\delta_{i,j}$ is a
	group homomorphism. For
	$R=\sum_{t=0}^a\mu_tz_0^tz_1^{a-t}$, formula~\eqref{deltaij} gives
	\[
	\delta_{i,j}(R)
	=
	\left[
	\lambda_i\binom{b-i}{p^e}
	\sum_{t=0}^a\mu_t^{p^e}
	z^{(n_j+(a-2t)p^e)/2}
	\right].
	\]
	The exponents are distinct, and the terms surviving in $V_j$
	are precisely those with $|2t-a|p^e\leq n_j$.
	Consequently,
	\begin{equation}\tag{$\diamondsuit$}\label{Kerdelta}
		\ker(\delta_{i,j})_{\mathrm{red}}
		=
		\left\{
		\sum_{t=0}^a\mu_tz_0^tz_1^{a-t}
		\ \middle|\
		\mu_t=0\text{ whenever }|2t-a|\leq\frac{n_j}{p^e}
		\right\}.
	\end{equation}
	
	For every $(R,g)\in H$, the element $g$ preserves the line
	corresponding to the first nonzero component $P_{i_0}$.
	Since $P_u=0$, comparison of the $u$-th components gives
	$\delta_{i_0,u}(R)=0$.
	Thus $R\in K=\ker(\delta_{i_0,u})_{\mathrm{red}}$.
	Since $H$ is reduced, the translation-coordinate morphism
	factors through the reduced kernel $K$.
	If $aq\leq n_u$, formula~\eqref{Kerdelta} gives $K=0$;
	then $H=G$, as required.
	We may therefore assume that $aq>n_u$.
	
	We prove by induction on $i$ that $P_i$ is $\SL_2$-invariant
	and that
	\begin{equation}\tag{$\clubsuit$}\label{reciK}
		\delta_{i,j}(R)=0
		\qquad\text{for every }j>i\text{ and }R\in K.
	\end{equation}
	There is nothing to prove if $P_i=0$.
	Otherwise $i\geq i_0$.
	For $i=i_0$, invariance follows from~\ref{LiftS1}.
	For $i>i_0$, the induction hypothesis implies that translations
	by elements of $K$ leave the $i$-th component unchanged.
	Since every element of $S$ has a lift $(R,g)\in H$ with $R\in K$,
	the line $\k P_i$ is $\SL_2$-stable, and hence $P_i$ is invariant.
	
	To prove~\eqref{reciK}, we may assume that $n_j\geq0$ and
	$\binom{b-i}{j-i}\neq0$ in $\k$.
	Set $v=\nu_p(b-i)$.
	Lucas's theorem implies that $j-i=mp^v$ for some $m\geq1$.
	Since $n_{i+p^v}\geq n_j\geq0$, assertion~\ref{LiftS2}
	gives $q\leq p^v$ when $i>i_0$; this also holds when $i=i_0$.
	Moreover, $j\geq i+p^v\geq u$.
	
	If $m=1$, formula~\eqref{Kerdelta}, together with
	$p^v\geq q$ and $n_j\leq n_u$, gives
	$K\subseteq\ker(\delta_{i,j})_{\mathrm{red}}$.
	If $m\geq2$, then
	\[
	j\geq i+p^v+p^v\geq u+q,
	\qquad
	n_j\leq n_u-aq<0,
	\]
	a contradiction.
	This completes the induction.
	
	Thus $K$ fixes $P$, so $K\subseteq H$.
	For every $g\in G$, choose a lift $(R,g)\in H$.
	Since $R\in K\subseteq H$, we obtain $(0,g)\in H$.
	Hence $H=K\rtimes G$, proving~\ref{LiftS5}.
\end{proof}
\begin{corollary}\label{cor:H-maps-onto-PGL2-even-a}
	Let $a\geq2$ be even.
	Let $\pi\colon X\to\F_a$ be a $\P^1$-bundle, and set
	\[
	H=\Im\bigl(\Autz(X)\to\Aut(\F_a)\bigr),
	\qquad
	U_a=\k[z_0,z_1]_a,
	\qquad
	S=\Im\bigl(\SL_2\to\GL_2/\mu_a\bigr).
	\]
	Let $q\colon\Aut(\F_a)\to\PGL_2$ be the homomorphism induced
	by the ruling $\F_a\to\P^1$.
	Then $q(H)=\PGL_2$ if and only if
	$uSu^{-1}\subseteq H$ for some $u\in U_a$.
\end{corollary}

\begin{proof}
	The reverse implication is immediate.
	Assume that $q(H)=\PGL_2$.
	Then $X\to\F_a$ has no jumping fiber.
	If it is decomposable, Lemma~\ref{lem:aut-group-Fabc}
	gives $H=\Aut(\F_a)$.
	
	Otherwise, we first apply Lemma~\ref{lem:lifttorus} and may assume,
	after conjugating $H$ by an element of $U_a$, that $H$ contains
	the diagonal torus of $S$. By Lemma~\ref{lemmaMonom}, this implies
	that $X\to\F_a$ is monomial.
	If $p=2$, Proposition~\ref{propMonom2} applies directly.
	If $p$ is odd, it applies because the even integer $a$
	is not a power of $p$.
	If only $P_b$ is nonzero, or if $n_u<0$, it gives $S\subseteq H$.
	Otherwise, assertion~\ref{LiftS4} allows us, after another conjugation
	by an element of $U_a$, to assume that $P_u=0$.
	Assertion~\ref{LiftS5} then gives $H=K\rtimes G$ with $S\subseteq G$.
	Undoing the conjugations proves the claim.
\end{proof}

Under the hypotheses of the following theorem,
the preceding results show that every indecomposable
$\P^1$-bundle over $\F_a$ either admits an equivariant
birational map to a decomposable $\P^1$-bundle with strictly
larger neutral automorphism group, or, for $a\geq1$,
is square-isomorphic to an Umemura $\P^1$-bundle.
We therefore obtain the following classification.

	\begin{theorem}\label{th:P1bundlesFa}
Let $a\geq0$, and let $\pi\colon X\to\F_a$ be a
$\P^1$-bundle.
Assume that $p\geq3$ or that $a$ is even.
Then exactly one of the following holds:
\begin{enumerate}
	\item\label{item i: Prop:P1bundlesFa}
	$X\to\F_a$ is decomposable, hence isomorphic to
	$\FF_a^{b,c}\to\F_a$ for some $b,c\in\Z$
	$($Definition~$\ref{def Fabc})$;
	
	\item\label{item ii: Prop:P1bundlesFa}
	$a\geq1$ and $X\to\F_a$ is square-isomorphic to an Umemura
	$\P^1$-bundle $\U_{a,\omega}^{b,c}\to\F_a$, for some
	$b,c,\omega$ as in Definition~$\ref{def Umemura bundle}$;
	
	\item\label{item iv: Prop:P1bundlesFa}
	$X\to\F_a$ is indecomposable, and either $a=0$, or $a\geq1$
	and $\Autz(X)$ fixes a point on the base of the ruling $\F_a\to\P^1$.
	In both cases, there exists a decomposable $\P^1$-bundle
	$\FF_a^{b,c}\to\F_a$ and a commutative diagram
	\[
	\xymatrix@R=3mm@C=2cm{
		X \ar@{-->}[rr]^{\psi} \ar[rd]_{\pi}
		&& \FF_a^{b,c} \ar[ld] \\
		& \F_a
	}
	\]
	where $\psi$ is a birational map satisfying
	$\psi\Autz(X)\psi^{-1}\subsetneq\Autz(\FF_a^{b,c})$.
\end{enumerate}
	\end{theorem}
	
	\begin{proof}
		If $X\to\F_a$ is decomposable, we obtain
		case~\ref{item i: Prop:P1bundlesFa}.
		We may therefore assume that it is indecomposable.
		If $a=0$, Corollary~\ref{cor: non-maximality of indecomposable}
		gives case~\ref{item iv: Prop:P1bundlesFa}.
		Henceforth assume that $a\geq1$. 
		
		Set $H=\Im(\Autz(X) \to \Aut(\F_a))$.
		Since $H$ is connected, its image in $\Aut(\P^1)=\PGL_2$
		either fixes a point or is the whole group.
		In the first case, Lemma~\ref{Lemm:fiberInvariant} gives
		case~\ref{item iv: Prop:P1bundlesFa}.
		In the second case, Lemma~\ref{lem:H-maps-onto-PGL2-odd-a}
		and Corollary~\ref{cor:H-maps-onto-PGL2-even-a}, according to
		the parity of $a$, show that $H$ contains a conjugate of
		$\Im(\SL_2\to\GL_2/\mu_a)$.
		After replacing $\pi$ by a square-isomorphic $\P^1$-bundle,
		Lemma~\ref{lem:characterization-Umemura} therefore gives
		case~\ref{item ii: Prop:P1bundlesFa}.
		
		The first case is disjoint from the other two by
		indecomposability. The second and third cases are disjoint
		because Umemura $\P^1$-bundles are defined for $a\geq1$ and their
		neutral automorphism groups act transitively on the base
		of the ruling, by Lemma~\ref{lem:characterization-Umemura}.
	\end{proof}
	
The following example shows that the case $(a,p)=(1,2)$
cannot be included in Theorem~\ref{th:P1bundlesFa}
without changing its conclusion: an indecomposable
$\P^1$-bundle over $\F_1$ can have an automorphism group
acting transitively on the base $\P^1$ without being
square-isomorphic to an Umemura $\P^1$-bundle.

\begin{lemma}\label{lem: counter example for p=2}
Assume that $p=2$, and consider
\[
X=Z_1^{2,8,P}\to\F_1,
\qquad
P(y_0,y_1,z)=y_1^2z^4+y_0^2z^5.
\]
Then $X$ is an indecomposable $\P^1$-bundle and
$H=\Im\bigl(\Autz(X)\to\Aut(\F_1)\bigr)$ is scheme-theoretically equal to
\[
\left\{
\left(
\lambda z_0+\mu z_1,
\begin{bsmallmatrix}a&b\\c&d\end{bsmallmatrix}
\right)
\in \k[z_0,z_1]_1\rtimes\SL_2
\ \middle|\
\lambda^2=cd,\ \mu^2=ab
\right\}.
\]
The group $H$ is isomorphic to $\PGL_2$. Its action on the base $\P^1$
corresponds to the Frobenius morphism $\PGL_2\to\PGL_2$, so $H$ acts
transitively on $\P^1$. However, $H$ does not contain a conjugate of the
standard $\SL_2\subset\Aut(\F_1)$.
\end{lemma}

\begin{proof}
	The group $H$ is the reduced identity component of the
	stabilizer group scheme $G$ of
\[
[Q]=[y_1^2z^3+y_0^2z^2]\in\M_1^{2,8}
=\P\left(
y_1^2\k[z]_{\leq6}\oplus
y_0y_1\k[z]_{\leq5}\oplus
y_0^2\k[z]_{\leq4}
\right)
\]
(see Proposition~\ref{prop: paramter space for non-dceomposable P1 bundles}).
In particular, $X$ is indecomposable.

To calculate $G\subseteq\Aut(\F_1)=\k[z_0,z_1]_1\rtimes\GL_2$, we first
consider $g=\begin{bsmallmatrix}a&b\\c&d\end{bsmallmatrix}\in\GL_2$.
By Lemma~\ref{Lemm:ActionAutFaParSpace}\ref{ActionGL2}, we have
\[
g(Q)=[y_1^2g(z^3)+y_0^2g(z^2)].
\]
Using the action of $\GL_2$ on $\k[z]_{\leq6}$ and $\k[z]_{\leq4}$
described in Remark~\ref{rk: def of dual Vm} and made explicit in
Remark~\ref{ExplicitActionSL2km}, and writing
$\delta=ad+bc=\det(g)$, we obtain
\[
g(Q)
=[y_1^2\delta^3z^3+
y_0^2(\delta cdz+\delta^2z^2+\delta abz^3)].
\]

We next calculate the action of $R\in\k[z_0,z_1]_1$.
Lemma~\ref{Lemm:ActionAutFaParSpace}\ref{ActionFap} yields
$R(Q)=\hat Q$, where $\hat Q$ corresponds to
$\hat X=Z_1^{2,8,\hat P}$, with
\[
\hat P(y_0,y_1,z)=P(y_0,y_1+y_0R(1,z),z).
\]
Writing $R=\lambda z_0+\mu z_1$, we obtain
\begin{align*}
\hat P(y_0,y_1,z)
 &=y_1^2z^4+y_0^2\bigl(R(1,z)^2z^4+z^5\bigr),\\
R(Q)
&=[y_1^2z^3+y_0^2(\lambda^2z+z^2+\mu^2z^3)].
\end{align*} 
The action of $R$ does not change the coefficients of $y_1^2z^3$ and
$y_0^2z^2$. Hence, if $(R,g)$ stabilizes $[Q]$, comparing with a common
projective scalar gives $\delta^3=\delta^2$. Since $\delta$ is
invertible, this is scheme-theoretically equivalent to $\delta=1$.
Comparing the remaining coefficients then gives
$\lambda^2=cd$ and $\mu^2=ab$.
Thus, scheme-theoretically,
\[
G=
\left\{
\left(
\lambda z_0+\mu z_1,
\begin{bsmallmatrix}a&b\\c&d\end{bsmallmatrix}
\right)
\in\k[z_0,z_1]_1\rtimes\SL_2
\ \middle|\ 
\lambda^2=cd \text{ and } \mu^2=ab.
\right\}.
\]

We now prove that $G\simeq\PGL_2$. Consider the homomorphism
\[
\psi\colon\PGL_2\longrightarrow\GL_3,\qquad
\begin{bsmallmatrix}\alpha&\beta\\ \gamma&\delta\end{bsmallmatrix}
\longmapsto
\frac{1}{\alpha\delta+\beta\gamma}
\begin{bsmallmatrix}
\alpha^2&\alpha\beta&\beta^2\\
0&\alpha\delta+\beta\gamma&0\\
\gamma^2&\gamma\delta&\delta^2
\end{bsmallmatrix},
\]
induced by the action of $\PGL_2$ on $\P(V_2)$. 
The homomorphism $\psi$ is a closed immersion: it is
injective on geometric points and its differential is injective.
Its image is
\[
\left\{
\begin{bsmallmatrix}
	a&\mu&b\\
	0&1&0\\
	c&\lambda&d
\end{bsmallmatrix}
\ \middle|\
\begin{bsmallmatrix}a&b\\c&d\end{bsmallmatrix}\in\SL_2,\
\lambda^2=cd,\ \mu^2=ab
\right\},
\]
which is naturally isomorphic to $G$. Thus $G\simeq\PGL_2$ is smooth and connected, and therefore $H=G$.

More precisely, under this isomorphism the projection
$G\to\SL_2$ is the very special isogeny
\[
\pi\colon\PGL_2\to\SL_2,\qquad
\begin{bsmallmatrix}\alpha&\beta\\ \gamma&\delta\end{bsmallmatrix}
\longmapsto
\frac{1}{\alpha\delta+\beta\gamma}
\begin{bsmallmatrix}
\alpha^2&\beta^2\\
\gamma^2&\delta^2
\end{bsmallmatrix},
\]
whose kernel is isomorphic to $\alpha_2^2$. If
$q\colon\SL_2\to\PGL_2$ denotes the central quotient, then
$q\circ\pi=F_{\PGL_2}$. Consequently, the action of
$H\simeq\PGL_2$ on the base $\P^1$ corresponds to the Frobenius
morphism $F_{\PGL_2}\colon\PGL_2\to\PGL_2$. In particular, it is
surjective, and its kernel fits into an exact sequence
\[
1\to\alpha_2^2\to\ker(F_{\PGL_2})\to\mu_2\to1.
\]
Thus $H$ acts transitively on $\P^1$.

It remains to prove that $H$ does not contain a conjugate of the standard
$\SL_2$. Such a conjugate is of the form
\[
u\SL_2u^{-1}
=\{(u-\alpha(u),\alpha)\mid\alpha\in\SL_2\}
\]
for some $u=u_0z_0+u_1z_1\in\k[z_0,z_1]_1$. For
$\alpha=\begin{bsmallmatrix}a&b\\c&d\end{bsmallmatrix}$, we have
\[
\alpha(u)
=u_0(dz_0+bz_1)+u_1(cz_0+az_1),
\]
and hence
\[
u-\alpha(u)
=\bigl(u_0(1+d)+cu_1\bigr)z_0
+\bigl(bu_0+u_1(1+a)\bigr)z_1.
\]
Thus, if $u\SL_2u^{-1}\subseteq H$, then
\[
cd=u_0^2(1+d^2)+c^2u_1^2,
\qquad
ab=b^2u_0^2+u_1^2(1+a^2)
\]
for every $\begin{bsmallmatrix}a&b\\c&d\end{bsmallmatrix}\in\SL_2$.
Taking $b=c=0$ and $d=a^{-1}$ with $a$ general gives
$u_0=u_1=0$. The two identities would then give $ab=cd=0$ for every
element of $\SL_2$, a contradiction. Hence $H$ contains no conjugate of
the standard $\SL_2$.
\end{proof}

\section{\texorpdfstring{$\P^1$}{P1}-bundles over the projective plane} \label{section: P1 bundles over P2} 
The results concerning decomposable $\P^1$-bundles and
the bundles $\L^{s,b}$ hold in arbitrary positive characteristic.
The Schwarzenberger results, Lemma~\ref{SubgroupsPGL3},
and Theorem~\ref{th:P1bundlesP2} require $p\geq3$.
	
We study decomposable $\P^1$-bundles, the $\P^1$-bundles $\L^{s,b}$,
and the generalized Schwarzenberger $\P^1$-bundles over $\P^2$,
describing their automorphism groups and their restrictions
to lines.
For $p\geq3$, we then prove Theorem~\ref{th:P1bundlesP2}:
every $\P^1$-bundle over $\P^2$ belongs to one of these families,
up to square isomorphism, or its neutral automorphism group
fixes a point on the base.

\subsection{Decomposable \texorpdfstring{$\P^1$}{P1}-bundles over \texorpdfstring{$\P^2$}{P2}}\label{sec: dec P1 bundles over P2}
In this section we recall some properties of the decomposable $\P^1$-bundles over $\P^2$ that we will use later.

\begin{definition} \label{def Pb}
Let $b\in \Z$. Define $\PP_b$ to be the quotient of $(\A^2\setminus \{0\})\times (\A^3\setminus \{0\})$ by the action of $(\G_m)^2$ given by
\[\begin{array}{ccc}
(\G_m)^2 \times (\A^2\setminus \{0\})\times (\A^3\setminus \{0\}) & \to & (\A^2\setminus \{0\})\times (\A^3\setminus \{0\})\\
((\mu,\rho), (y_0,y_1;z_0,z_1,z_2))&\mapsto& (\mu\rho^{-b} y_0,\mu y_1;\rho z_0,\rho z_1,\rho z_2)\end{array}\]
The class of $(y_0,y_1,z_0,z_1,z_2)$ will be written $[y_0:y_1;z_0:z_1:z_2]$. The projection 
\[\PP_b\to \P^2, \ \ [y_0:y_1;z_0:z_1:z_2]\mapsto [z_0:z_1:z_2]\]
identifies $\PP_b$ with 
\[\P(\OPP(b) \oplus \OPP)=\P(\OPP \oplus \OPP(-b))\]
as a $\P^1$-bundle over $\P^2$. As before, we get an isomorphism of $\P^1$-bundles $\PP_{b}\simeq \PP_{-b}$ by exchanging $y_0$ with $y_1$, and will then often assume that $b\ge0$ in the sequel.
\end{definition}

\begin{lemma}\label{lemm:surjectiveDecP2}
	For each $i\geq0$, denote by $\k[z_0,z_1,z_2]_i$ the
	vector space of homogeneous polynomials of degree $i$.
	For each $b\geq1$, the group $\Autz(\PP_b)$ consists of
	all transformations of the form
	\[
	[y_0:y_1;z_0:z_1:z_2]\longmapsto
	\left[
	\begin{bmatrix}1&0\\F&\lambda\end{bmatrix}
	\begin{bmatrix}y_0\\y_1\end{bmatrix};
	A\begin{bmatrix}z_0\\z_1\\z_2\end{bmatrix}
	\right],
	\]
	where $\lambda\in\k^*$, $F\in\k[z_0,z_1,z_2]_b$,
	and $A\in\GL_3(\k)$.
	Moreover, $\PP_b$ is toric, and the projection
	$\PP_b\to\P^2$ induces a surjective group homomorphism
	\[
	\rho\colon\Autz(\PP_b)\twoheadrightarrow
	\Aut(\P^2)=\PGL_3.
	\]
\end{lemma}

\begin{proof}
	The existence of $\rho$ follows from Blanchard's lemma
	(Proposition~\ref{prop:blanchard-lemma}).
	The natural $\GL_3$-action on $z_0,z_1,z_2$, leaving
	$y_0,y_1$ unchanged, descends to $\PP_b$ and proves
	the surjectivity of $\rho$.
	
	Every automorphism over $\P^2$ preserves the section
	$y_0=0$, since its restriction over each line in $\P^2$
	is the unique negative section of $\F_b$.
	Its action on the complement, which is the total space
	of $\OPP(b)$, is therefore
	\[
	y_1/y_0\longmapsto\lambda(y_1/y_0)+F,
	\qquad
	\lambda\in\k^*,\quad F\in H^0(\P^2,\OPP(b)).
	\]
	Conversely, these formulas define a connected group
	of automorphisms over $\P^2$.
	Together with the $\GL_3$-action, this gives the
	asserted description.
	Finally, the locus where all five coordinates are
	nonzero is a dense torus $(\G_m)^5/(\G_m)^2$
	whose action extends to $\PP_b$, so $\PP_b$ is toric.
\end{proof}

\subsection{The \texorpdfstring{$\P^1$}{P1}-bundles \texorpdfstring{$\L^{s,b} \to \P^2$}{P2}}\label{sec: Ls bundles}
We introduce a family of $\P^1$-bundles
$\L^{s,b}\to\P^2$ whose neutral automorphism groups
map onto the stabilizer of a line when $b<p^s$,
and onto $\PGL_3$ when $b=p^s$
(Proposition~\ref{prop: LsbAut}).

The case $b<p^s$ exhibits a phenomenon absent in
characteristic zero: by \cite[Proposition~4.3.4]{BFT},
the image of the neutral automorphism group of a
$\P^1$-bundle over $\P^2$ cannot then be the full
stabilizer of a line.

\begin{definition} \label{def: P1 bundles Lsb}
Let $s \ge 0$ and $1 \le b \le p^s$. 
Let $U_0=\{[y_0:y_1:y_2]\mid y_0\neq0\}$ and
$U_1=\{[y_0:y_1:y_2]\mid y_1\neq0\}$.
We denote by $\L^{s,b} \to \P^2$ the $\P^1$-bundle obtained by gluing 
$\P^1\times U_0$ and $\P^1\times U_1$ along $\P^1\times (U_0\cap U_1)$
via the map
\[
\begin{array}{rccc}
\nu_{01}\colon \P^1\times U_0 & \dasharrow & \P^1\times U_1 \\
\left(\begin{bmatrix} x_0 \\ x_1 \end{bmatrix},[y_0:y_1:y_2]\right) 
& \longmapsto &
\left(
\begin{bmatrix}
y_0^{2p^s-b} & 0\\
y_0^{p^s-b} y_1^{p^s-b} y_2^b & y_1^{2p^s-b}
\end{bmatrix}
\begin{bmatrix} x_0 \\ x_1 \end{bmatrix},
[y_0:y_1:y_2]
\right).
\end{array}
\]
We will also write $\L^{s}$ for $\L^{s,p^s}$.
\end{definition}

\begin{remark}\label{rk: theta02}
In the previous definition we implicitly used that a $\P^1$-bundle over 
$\P^2\setminus\{[0:0:1]\}$ extends uniquely, up to isomorphism, to a $\P^1$-bundle over $\P^2$; see Lemma~\ref{lem:descent-lemma}.

The transition functions $\nu_{02}$ and $\nu_{12}$
can be determined from $\nu_{01}$ using the matrix identity
 that holds in $\GL_2\!\left(\k[y_0^{\pm 1},y_1^{\pm 1},y_2^{\pm 1}]\right)$:
\[ \begin{bmatrix}
-\frac{{y_2}^{p^s}}{{y_1}^{p^s}} & \frac{{y_0}^{p^s}{y_2}^{p^s-b}}{y_1^{2p^s-b}}\\
0 & \frac{{y_2}^{p^s-b}}{{y_1}^{p^s-b}}
\end{bmatrix}
\begin{bmatrix}
1 & 0\\
\frac{{y_2}^b{y_1}^{p^s-b}}{y_0^{p^s}} & \frac{{y_1}^{2p^s-b}}{y_0^{2p^s-b}}
\end{bmatrix}=
\begin{bmatrix}
0 & \frac{y_2^{p^s-b}}{y_0^{p^s-b}}\\
\frac{y_2^{p^s}}{y_0^{p^s}} & \frac{y_1^{p^s}y_2^{p^s-b}}{y_0^{2p^s-b}}
\end{bmatrix}.\]
This corresponds to the cocycle relation
\[
\nu_{12}\circ\nu_{01}=\nu_{02}.
\]
Hence, $\nu_{12}$ may be defined by
\[\begin{aligned}
\nu_{12}\colon & \P^1\times U_1 \dashrightarrow \P^1\times U_2 \\
&\left(
\begin{bmatrix}x_0\\x_1\end{bmatrix},
\left[y_0:y_1:y_2\right]
\right)
\longmapsto
\left(
\begin{bmatrix}
-{y_1}^{p^s-b}{y_2}^{b} & {y_0}^{p^s}\\
0 & {y_1}^{p^s}
\end{bmatrix}
\begin{bmatrix}x_0\\x_1\end{bmatrix},
\left[y_0:y_1:y_2\right]
\right).
\end{aligned}
\]
This choice implies that $\nu_{02}=\nu_{12}\circ\nu_{01}$ is given by
\[
\begin{aligned}
\nu_{02}\colon & \P^1\times U_0 \dashrightarrow \P^1\times U_2 \\
&\left(
\begin{bmatrix}x_0\\x_1\end{bmatrix},
[y_0:y_1:y_2]
\right)
\longmapsto
\left( 
\begin{bmatrix}
0 &  y_0^{p^s}\\
y_0^{p^s-b} y_2^{b} & y_1^{p^s}
\end{bmatrix} 
\begin{bmatrix}x_0\\x_1\end{bmatrix},
\left[y_0:y_1:y_2\right]
\right).
\end{aligned}
\]\end{remark}

\begin{remark}\label{Rem:ExchangeLsb}
The automorphism 
\[
\begin{array}{rccc}
\tau_2\colon & \P^1\times U_2 &\longrightarrow & \P^1\times U_2 \\
&([x_0:x_1],
[y_0:y_1:y_2])&
\longmapsto
&([x_1:x_0],
[y_1:y_0:y_2])
\end{array}
\]
satisfies
\[\tau_2\circ  \nu_{02}=\nu_{12}\circ \tau_2'\]
where $\tau_2'$ is the isomorphism
\[
\begin{array}{rccc}
\tau_2'\colon & \P^1\times U_0 &\longrightarrow & \P^1\times U_1 \\
&([x_0:x_1],
[y_0:y_1:y_2])&
\longmapsto
&([-x_0:x_1],
[y_1:y_0:y_2])
\end{array}
\]
This follows from the matrix identity
\[\begin{bmatrix} 0 & 1 \\ 1 & 0\end{bmatrix}\begin{bmatrix}
0 &  y_0^{p^s}\\
y_0^{p^s-b} y_2^{b} & y_1^{p^s}
\end{bmatrix} = \begin{bmatrix}
-{y_0}^{p^s-b}{y_2}^{b} & {y_1}^{p^s}\\
0 & {y_0}^{p^s}
\end{bmatrix} \begin{bmatrix} -1 & 0 \\ 0 & 1\end{bmatrix}.\]
Hence, $\tau_2$ extends to a unique biregular automorphism of $\L^{s,b}$.
\end{remark}

We first describe the case $b=p^s$, which is obtained
from the incidence $\P^1$-bundle by Frobenius pullback on
the base.

\begin{lemma}\label{Lem:LsPGL3}
Let $s\ge 0$. The $\P^1$-bundle $\L^s=\L^{s,p^s}\to \P^2$ is isomorphic to
\[
Z_s:=\left\{(x,y)\in \P^2 \times \P^2 \;\middle|\; \sum_{i=0}^2 x_i y_i^{p^s}=0\right\},
\]
which is a $\P^1$-bundle over $\P^2$ via the projection onto the second factor.
In particular, $\L^0 \simeq \P(T_{\P^2})$ is the projectivization of the tangent bundle of $\P^2$.
 
The group $\Aut(\P^2)\simeq \PGL_3$ acts biregularly and transitively on $\L^s$ as follows:
\[
\left(\begin{bmatrix}x_0\\x_1\\ x_2\end{bmatrix},
\begin{bmatrix}y_0\\y_1\\ y_2\end{bmatrix}\right)
\mapsto
\left(
\begin{bmatrix}
a_{00}^{p^s} & a_{10}^{p^s} & a_{20}^{p^s}\\
a_{01}^{p^s} & a_{11}^{p^s} & a_{21}^{p^s}\\
a_{02}^{p^s} & a_{12}^{p^s} & a_{22}^{p^s}
\end{bmatrix}^{-1}
\begin{bmatrix}x_0\\x_1\\ x_2\end{bmatrix},
\begin{bmatrix}
a_{00} & a_{01} & a_{02}\\
a_{10} & a_{11} & a_{12}\\
a_{20} & a_{21} & a_{22}
\end{bmatrix}
\begin{bmatrix}y_0\\y_1\\ y_2\end{bmatrix}
\right).
\]
\end{lemma}

\begin{proof}
	Writing homogeneous coordinates as column vectors,
	let $F_s$ denote coordinatewise $p^s$-th powers,
	also applied entrywise to matrices.
	Then
	\[
	Z_s=\{(x,y)\in\P^2\times\P^2
	\mid {}^tx\,F_s(y)=0\},
	\]
	and $\PGL_3$ acts on $Z_s$ by
	\[
	(A,(x,y))\longmapsto
	\bigl(F_s({}^tA^{-1})x,Ay\bigr).
	\]
	This action is transitive: $\PGL_3$ acts transitively
	on the second factor, and the stabilizer of
	$[1:0:0]$ acts transitively on the fiber
	$\{[0:x_1:x_2]\}\simeq\P^1$, since the Frobenius
	homomorphism of $\GL_2$ is surjective.
	
We now consider the projection $\pi$ onto the second factor. We obtain isomorphisms
\[
\begin{array}{ccc}
\P^1\times U_0&\to& \pi^{-1}(U_0)\\
([x_0:x_1], [1:u:v]) & \mapsto & ([-x_0v^{p^s}-x_1u^{p^s}:x_1:x_0],[1:u:v])\\[0.3em]
\P^1\times U_1&\to& \pi^{-1}(U_1)\\
([x_0:x_1], [u:1:v]) & \mapsto & ([-x_1:-x_0v^{p^s}+x_1u^{p^s}:x_0],[u:1:v])
\end{array}
\]  
which yield the transition function
\[
\begin{array}{ccc}
\P^1\times U_0&\dasharrow &\P^1 \times U_1\\
([x_0:x_1], [1:u:v])  & \mapsto & ([x_0:x_0v^{p^s}+x_1u^{p^s}],[\tfrac{1}{u}:1:\tfrac{v}{u}])
\end{array}
\] 
which coincides with the transition function of $\L^s$ given in Definition~\ref{def: P1 bundles Lsb}.

The resulting isomorphism over $U_0\cup U_1$ extends
uniquely over $\P^2$.
For $s=0$, the equation defines the incidence variety
of points and lines in $\P^2$, hence
$Z_0\simeq\P(T_{\P^2})$.
\end{proof}

\begin{remark} \label{rk: Ps}
For each $s \geq 0$, the $\P^1$-bundle $\L^s \to \P^2$ identifies with the flag variety $\PGL_3/P_s$, together with the natural projection $\PGL_3/P_s \to \PGL_3/P$, where the parabolic subgroup schemes $P_s$ and $P$ are defined by
\[
P_s=
\left\{
\begin{bsmallmatrix}
	*&*&*\\
	0&*&*\\
	0&t&*
\end{bsmallmatrix}
\ \middle|\ t^{p^s}=0
\right\}
\subseteq\PGL_3,
\qquad
P=
\left\{
\begin{bsmallmatrix}
	*&*&*\\
	0&*&*\\
	0&*&*
\end{bsmallmatrix}
\right\}
\subseteq\PGL_3.
\]
Indeed, $P_s$ is the stabilizer of
$([0:0:1],[1:0:0])\in Z_s$, while $P$ is the
stabilizer of its image $[1:0:0]$ in the base.
(Note that $P_s$ is non-smooth when $s>0$.) 
This implies that $\Autz(\L^s) \simeq \PGL_3$ (see e.g.~\cite[Theorem~A]{Mac25}), a result that also follows from Proposition~\ref{prop: LsbAut} below.
\end{remark} 

\begin{lemma}\label{Lem:LsbBlowUp}
Let $s\ge 0$ and $1\le b\le p^s$. We consider the $\P^1$-bundle $\L^{s,b} \to \P^2$ of Definition~$\ref{def: P1 bundles Lsb}$. Then the blow-up of $p_0=[0:0:1]\in \P^2$ and of its fiber in $\L^{s,b}$ yields a commutative diagram
\[
\xymatrix@R=4mm@C=2cm{ Z_{1}^{b,c,P} \ar[r] \ar[d] & \L^{s,b} \ar[d] \\ \F_1 \ar[r]^{\eta} & \P^2 }
\]
where $Z_{1}^{b,c,P}\to \F_1$ is the Umemura $\P^1$-bundle
\[\U_{1,\omega}^{b,c}=Z_{1}^{b,c,P}\to \F_1\ \text{ with}\  P(y_0,y_1,z)= y_1^b z^{p^s} \ \text{and}\ c=2p^s,\] 
and $\eta([y_0:y_1; z_0:z_1])= [y_0 z_0 : y_0 z_1 : y_1]$.
\end{lemma}

\begin{proof}
As in Definition~\ref{def: P1 bundles Lsb}, we denote by $U_0,U_1 \subseteq \P^2$ the affine charts
\[
U_0=\{[x:y:z]\mid x\neq 0\},\qquad
U_1=\{[x:y:z]\mid y\neq 0\}.
\]
The preimages by $\eta$ of $[1:u:v]\in U_0$ and $[u:1:v]\in U_1$ correspond to $[1:v;1:u]$ and $[1:v;u:1]$, respectively.
According to Definition~\ref{Not:ZabcP}, the transition function $\nu_{c,P}\in \Aut(\F_b \times \A^1 \setminus \{0\})$ of the $\P^1$-bundle $Z_{1}^{b,c,P}\to \F_1$ is
\[
\nu_{c,P}\colon([x_0:x_1;y_0:y_1],z) \mapsto ([x_0:x_1z^{c}+x_0y_1^bz^{p^s};y_0z:y_1],{ \tfrac{1}{z}}).
\]
Moreover, the $\P^1$-bundle $Z_{1}^{b,c,P}\to \F_1$ sends $([x_0:x_1;y_0:y_1],z)\in \F_b\times \A^1$ onto $[y_0:y_1;1:z]\in \F_1$ on the first chart and $[y_0:y_1;z:1]\in \F_1$ on the second chart. Hence the points $([1:u:v],[x_0:x_1])\in U_0\times \P^1$ and $([u:1:v],[x_0:x_1])\in U_1\times \P^1$ both correspond to $([x_0:x_1;1:v],u)\in \F_b\times \A^1$. We then compute
\[
\begin{array}{rcl}
\nu_{c,P}([x_0:x_1;1:v],u)
&=&([x_0:x_1u^{c}+x_0v^bu^{p^s};u:v],\tfrac{1}{u})\\[2mm]
&=&([x_0:x_1u^{c-b}+x_0v^bu^{p^s-b};1:\tfrac{v}{u}],\tfrac{1}{u}).
\end{array}
\]
Since $c=2p^s$, we obtain the transition function
\[
\begin{array}{rcccccc}
\nu_{01} \colon & \P^1 \times U_0 &\dasharrow &\P^1 \times U_1 \\
&\left(\begin{bmatrix} x_0 \\ x_1 \end{bmatrix} ,[1:u:v] \right) 
&\longmapsto&
\left(
\begin{bmatrix} 1 & 0 \\ v^b u^{p^s-b} & u^{2p^s-b} \end{bmatrix}
\begin{bmatrix} x_0 \\ x_1 \end{bmatrix},
\left[\dfrac{1}{u}:1:\dfrac{v}{u}\right]
\right).
\end{array}
\]
This is precisely the transition function of the $\P^1$-bundle 
$\L^{s,b}\to\P^2$ of Definition~\ref{def: P1 bundles Lsb}.

It remains to extend this identification over the exceptional
curve. Put $q=p^s$. Over $\eta^{-1}(U_2)$, use the two
base charts $[u:1;1:z]$ and $[u:1;z:1]$.
The fiber coordinates in the trivialization over $U_2$
are respectively
\[
[x_1:x_0+z^q x_1]
\qquad\text{and}\qquad
[-x_0+z^q x_1:x_1].
\]
These formulas are compatible with $\nu_{c,P}$ and with
the transition functions of Remark~\ref{rk: theta02}.
Their matrices are invertible even along $u=0$.
Thus
\[
Z_1^{b,c,P}\simeq
\F_1\times_{\P^2}\L^{s,b}.
\]
Since blowing up commutes with flat base change,
the induced morphism to $\L^{s,b}$ is the blow-up
of the fiber over $p_0$.
\end{proof}

\begin{proposition}\label{prop: LsbAut} 
Let $s\ge 0$ and $1\le b\le p^s$. Consider the $\P^1$-bundle $\pi\colon \L^{s,b} \to \P^2$ of Definition~$\ref{def: P1 bundles Lsb}$. Then:

\begin{enumerate}
\item\label{AutLp2trivial}
The group $\Aut(\L^{s,b})_{\P^2}$ of $\P^1$-bundle automorphisms over $\P^2$ is trivial.
\item\label{item: iso in LsbAut}
The action of $\Autz(\L^{s,b})$ on the base yields
\[
\Autz(\L^{s,b}) \simeq 
\begin{cases}
\Aut(\P^2,L) & \text{if } b < p^s,\\
\Aut(\P^2) & \text{if } b = p^s,
\end{cases}
\]
where $\Aut(\P^2,L) \simeq \GL_2\ltimes \G_a^2$ is the parabolic subgroup preserving the line $L=\{y_2=0\}$. The lifted action of $\Aut(\P^2,L)$ on
$\pi^{-1}(U_2)=\P^1\times U_2$ is given by
\[
\begin{array}{ccc}
 \P^1\times U_2 &\longrightarrow & \P^1\times U_2 \\
\left(
\begin{bmatrix}x_0\\x_1\end{bmatrix},
[y_0:y_1:y_2]
\right)&
\longmapsto
&\left(
\begin{bmatrix}
\alpha^{p^s} & \beta^{p^s}\\
\gamma^{p^s} & \delta^{p^s}
\end{bmatrix}
\begin{bmatrix}x_0\\x_1\end{bmatrix},
[\alpha y_0+\beta y_1+\epsilon_1 y_2:\gamma y_0+\delta y_1+\epsilon_2 y_2:y_2]
\right).
\end{array}
\]
\item\label{LSblines}
The preimage of $L$ in $\L^{s,b}$ is $\F_{2p^s - b}$, while the preimage of any other line in $\P^2$ is $\F_b$.
\end{enumerate}
\end{proposition}

\begin{proof}
\ref{AutLp2trivial}:
Consider the birational morphism $\eta\colon Z_{1}^{b,c,P}\to \L^{s,b}$ of Lemma~\ref{Lem:LsbBlowUp}, with $c=2p^s$. Since $\eta$ is the blow-up of the fiber over $[0:0:1]\in \P^2$, which is invariant under $\Aut(\L^{s,b})_{\P^2}$, we get an embedding 
\[
\Aut(\L^{s,b})_{\P^2} \hookrightarrow \Aut(Z_{1}^{b,c,P})_{\F_1}.
\] 
Since $a=1$, $b\le p^s$, and $c=2p^s$, we have
$ai-c<0$ for all $i\in\{0,\ldots,b\}$.
Lemma~\ref{AutZaF1} therefore gives
$\Aut(Z_1^{b,c,P})_{\F_1}=\{1\}$.
The same calculation over arbitrary $\k$-algebras shows
that this automorphism group scheme is trivial.
As every automorphism of $\L^{s,b}$ over $\P^2$ lifts
to its pullback $Z_1^{b,c,P}$ over $\F_1$,
the scheme-theoretic kernel of $\rho$ is trivial.
This proves~\ref{AutLp2trivial}.

\medskip
\ref{item: iso in LsbAut} and \ref{LSblines}:
By Blanchard's lemma (Proposition~\ref{prop:blanchard-lemma}), the $\Autz(\L^{s,b})$-action on the base defines an injective homomorphism 
\[
\rho\colon \Autz(\L^{s,b}) \hookrightarrow \Aut(\P^2).
\] 
Let $H = \operatorname{Im}(\rho)$. We show $H = \Aut(\P^2,L)$ if $b<p^s$ and $H = \Aut(\P^2)$ if $b=p^s$.  

An element of $\Aut(\P^2,L)$ has the form
\[
[x:y:z]\mapsto [\alpha x+\beta y+\epsilon_1 z : \gamma x+\delta y+\epsilon_2 z : z], 
\quad \text{where}\ 
\begin{bmatrix}\alpha & \beta \\ \gamma & \delta\end{bmatrix}\in \GL_2\ \text{and}\ \ (\epsilon_1,\epsilon_2)\in \G_a^2,  
\]
with action on $\P^1\times U_2$ given by
\[
\left(\begin{bmatrix}x_0\\x_1\end{bmatrix},[y_0:y_1:y_2]\right)
\longmapsto
\left(
\begin{bmatrix}\alpha^{p^s} & \beta^{p^s}\\ \gamma^{p^s} & \delta^{p^s}\end{bmatrix} 
\begin{bmatrix}x_0\\x_1\end{bmatrix},
[\alpha y_0+\beta y_1+\epsilon_1 y_2:\gamma y_0+\delta y_1+\epsilon_2 y_2:y_2]
\right).
\]

For $\beta=\epsilon_1=0$ (so $U_0$ is invariant), conjugation via the transition function $\nu_{02}$ (Remark~\ref{rk: theta02}) yields
\[
\left(\begin{bmatrix}x_0\\x_1\end{bmatrix},[y_0:y_1:y_2]\right) 
\mapsto
\left(
\begin{bmatrix}\delta^{p^s}& -\epsilon_2^{p^s}(y_2/y_0)^{p^s-b} \\ 0 & \alpha^{p^s-b}\end{bmatrix} 
\begin{bmatrix}x_0\\x_1\end{bmatrix}, 
[\alpha y_0 : \gamma y_0+\delta y_1+\epsilon_2 y_2 : y_2]
\right).
\]
Indeed, this follows from the matrix identity 
\[\ \begin{array}{lll} &&\begin{bmatrix} \alpha^{p^s} & 0\\ \gamma^{p^s} & \delta^{p^s} \end{bmatrix} \begin{bmatrix} 0 & y_0^{p^s}\\ y_0^{p^s-b} y_2^{b} & y_1^{p^s} \end{bmatrix}=\begin{bmatrix} 0 & (\alpha y_0)^{p^s}\\ \delta^{p^s}y_0^{p^s-b} y_2^{b} & (\gamma y_0+\delta y_1)^{p^s} \end{bmatrix}\vspace{0.2cm}\\ &=& \begin{bmatrix} 0 & (\alpha y_0)^{p^s}\\ (\alpha y_0)^{p^s-b} y_2^{b} & (\gamma y_0+\delta y_1+\epsilon_2 y_2)^{p^s} \end{bmatrix} \begin{bmatrix} \delta^{p^s}& -\epsilon_2^{p^s} (\frac{y_2}{y_0})^{p^s-b}\\ 0 & \alpha^{p^s-b} \end{bmatrix}\frac{1}{\alpha^{p^s-b}}\end{array}.\]

The formulas above define the triangular action over
$U_0\cup U_2$. They extend across the missing point
$[0:1:0]$ by \cite[Lemma~2.3.1]{BFT}.
The exchange automorphism of Remark~\ref{Rem:ExchangeLsb}
also lifts.
Their induced transformations on the base generate
$\Aut(\P^2,L)$, so
$\Aut(\P^2,L)\subseteq H$.

Since $\Aut(\P^2,L)$ acts transitively on lines $L'\neq L$, the preimage of any such $L'$ in $\L^{s,b}$ is $\F_b$; this holds in particular for lines through $p_0=[0:0:1]$ by Lemma~\ref{Lem:LsbBlowUp}.  
Restricting the transition function of Definition~\ref{def: P1 bundles Lsb} to $y_2=0$, the preimage of $L$ is $\F_{2p^s-b}$.  
Hence \ref{LSblines} is proven. Moreover, if $b<p^s$, then $\F_{2p^s-b}\not\simeq \F_b$, forcing $H\subseteq  \Aut(\P^2,L)$, which establishes \ref{item: iso in LsbAut}. 
If $b=p^s$, then $H=\Aut(\P^2)$ by Lemma~\ref{Lem:LsPGL3}, completing the proof.
\end{proof}

\subsection{The Schwarzenberger \texorpdfstring{$\P^1$}{P1}-bundles over \texorpdfstring{$\P^2$}{P2}}\label{sec: Schwarzenbergger bundles}
We assume that $p \geq 3$ and now turn our attention to (generalized) Schwarzenberger $\P^1$-bundles over $\P^2$.

\begin{definition}[Schwarzenberger $\P^1$-bundles]\label{def:Schwarz bc}
Let $b \geq 1$ and $c \geq 3$ be such that
$c=b+2(l+1)$,
where $l=0$, or $l$ satisfies
\[
p^{s-1}\le l<p^s
\quad\text{and}\quad
b+l+1\equiv 0 \pmod{p^s}
\]
for some $s\ge 1$. Assume moreover that $l$ is even.

The \emph{Schwarzenberger $\P^1$-bundle}
\[
\SS^{b,c}\to \P^2
\]
is the $\P^1$-bundle defined as the quotient of the $\P^1$-bundle
$\T^{b,c}\to \F_0$
by the unique involution $\sigma\in\Aut(\T^{b,c})$ such that
\[
\xymatrix@R=3mm{
\T^{b,c}\ar[r]^{\sigma}\ar[d] & \T^{b,c}\ar[d]\\
\F_0\ar[r]^{\tau} & \F_0
}
\]
commutes, where $\tau([y_0:y_1],[z_0:z_1])=([z_0:z_1],[y_0:y_1])$; see Proposition~\ref{prop: def of the Sbc}.
\end{definition}

\begin{remark}\label{rk: kappa and Gamma}
Let 
$\kappa\colon \F_0 \to \P^2$
be the $(2\!:\!1)$-cover defined by
\[
\begin{array}{rccc}
\kappa\colon & \F_0 & \to & \P^2\\
& ([y_0:y_1],[z_0:z_1]) 
& \mapsto 
& [y_0 z_0 : y_0 z_1 + y_1 z_0 : y_1 z_1],
\end{array}
\]
whose ramification locus is the diagonal $\Delta \subseteq \F_0$, and whose branch locus is the smooth conic
\begin{equation*}
\Gamma=\{[X:Y:Z]\mid Y^2=4XZ\}\subseteq \P^2.
\end{equation*}
Let $m \ge 1$ be an integer. The $\P^1$-bundle defined by
\[
\SS_m=\P\bigl(\kappa_* \O_{\F_0}(-m-1,0)\bigr)\to \P^2
\]
and studied in \cite[Section~4.2]{BFT} coincides with the $\P^1$-bundle $\SS^{m,m+2} \to \P^2$. This justifies our use of the same name and notation as in Definition~\ref{def:Schwarz bc}.
\end{remark}

\begin{proposition}\label{prop:aut-orbits-Sbc}
Assume that $p \geq 3$.
Let $b,c,l$ be as in Definition~$\ref{def:Schwarz bc}$, and let
$\pi\colon \SS^{b,c}\to\P^2$
be the corresponding Schwarzenberger $\P^1$-bundle.
Let $\Gamma \subseteq \P^2$ be the smooth conic which is the
branch locus of the $(2\!:\!1)$-cover
$\kappa\colon\F_0\to\P^2$ $($see Remark~$\ref{rk: kappa and Gamma})$.
\begin{enumerate}
\item\label{FnLineSchwarz}
If $L\subseteq \P^2$ is a line, the $\P^1$-bundle $\pi^{-1}(L)\stackrel{\pi}{\to} L\simeq \P^1$ is isomorphic to
\[\begin{array}{lll}
\F_b\to \P^1 & \text{ if $L$ is a line tangent to $\Gamma$;}\\
\F_0\to \P^1 & \text{ if $L$ is a line not tangent to $\Gamma$ and $b+l+1$ is odd; or}\\
\F_{l+1}\to \P^1 & \text{ if $L$ is a line not tangent to $\Gamma$ and $b+l+1$ is even.}\end{array}
\]

\item\label{item:aut-Sbc} We have $\Aut(\SS^{b,c})_{\P^2}=\{\Id\}$. Moreover:
\begin{itemize}
\item If $c=3b$, then $b=p^s$ for some $s \geq 0$ $($note that the case $s=0$ corresponds to $l=0)$. 
In this case, the $\P^1$-bundle $\SS^{b,c}\to\P^2$ identifies with the flag variety $\PGL_3/P_s \simeq \L^s$ of Remark~$\ref{rk: Ps}$, and hence
\[
\Autz(\SS^{b,c}) \simeq \PGL_3.
\]
\item  
If $c \neq 3b$, then the action on the base induces an isomorphism
\[
\Autz(\SS^{b,c})
\xrightarrow{\ \sim\ }
\Aut(\P^2,\Gamma)
\simeq\PGL_2.
\]
\end{itemize} 
\item\label{item:open-orbit-Sbc}
Set $m=b+l+1$.
Then $\SS^{b,c}$ contains an open $\PGL_2$-orbit isomorphic to
$\PGL_2/H'$,
where
\[
H'
=
\left\langle
\left\{
\begin{bmatrix}
\alpha&0\\
0&\alpha^{-1}
\end{bmatrix}
\ \middle|\
\alpha^{2m}=1
\right\},
\begin{bmatrix}
0&1\\
-1&0
\end{bmatrix}
\right\rangle.
\]
Here the matrices are viewed in $\PGL_2$, and
$H'$ denotes the closed subgroup scheme generated by
the image of $\mu_{2m}$ and the displayed involution.
\end{enumerate}
\end{proposition}

\begin{proof}
The pull-back of $\pi\colon \SS^{b,c}\to\P^2$ by the $(2\!:\!1)$-cover
$\kappa\colon\F_0\to\P^2$ is precisely
$\hat \pi\colon \T^{b,c}\to\F_0$. 
Hence an automorphism of $\SS^{b,c}$ inducing the
identity on $\P^2$ lifts to an automorphism of $\T^{b,c}$ inducing the
identity on $\F_0$. By
Proposition~\ref{prop:properties of the Tbc}\ref{item vert part trivial},
the latter is trivial. Therefore
\[
\Aut(\SS^{b,c})_{\P^2}=\{\Id\}.
\]
The same argument over arbitrary $\k$-algebras, using
the calculation in Lemma~\ref{AutZaF1}, shows that
the scheme-theoretic kernel is also trivial.

Let
$q\colon\T^{b,c}\longrightarrow\SS^{b,c}=\T^{b,c}/\left\langle \sigma \right \rangle $
be the quotient by the involution $\sigma$ of
Proposition~\ref{prop:properties of the Tbc}. 
By Proposition~\ref{prop:properties of the Tbc}\ref{item vert part trivial},
we have $\Autz(\T^{b,c})\simeq\PGL_2$,
and its action on the base $\F_0=\P^1\times\P^1$ is the diagonal one.
Moreover, the action commutes with $\sigma$, so it descends through $q$.
Under
$\kappa\colon\F_0\longrightarrow\P^2$,
the diagonal action of $\PGL_2$ induces the usual action of
$\Aut(\P^2,\Gamma)\simeq\PGL_2$.
By Blanchard's lemma
(Proposition~\ref{prop:blanchard-lemma}), every element
of $\Autz(\SS^{b,c})$ induces an automorphism of $\P^2$.
We thus obtain inclusions
\[
\PGL_2 \simeq \Aut(\P^2,\Gamma)
\subseteq  \Autz(\SS^{b,c}) \subseteq  \Aut(\P^2) \simeq \PGL_3.
\]  
Taking a line $L\subseteq \P^2$ tangent to $\Gamma$, the preimage $\kappa^{-1}(L)\subseteq \P^1\times \P^1$ consists of two fibers $f_1,f_2$ of the two projections, exchanged by the involution associated to the double cover $\kappa$. In particular, $\pi^{-1}(L)$ is isomorphic to $\hat\pi^{-1}(f_1)\simeq \hat\pi^{-1}(f_2)\simeq \F_b$ since $\T^{b,c}$ has numerical invariants $(0,b,c)$.

We now take a line $L\subseteq \P^2$ not tangent to $\Gamma$ and show that $\pi^{-1}(L)$ is isomorphic to $\F_0$ or $\F_{l+1}$, depending on whether $b+l+1$ is odd or even. Using the action of $\Autz(\SS^{b,c})$, we can choose the line of equation $Y=0$. Using the notation of Proposition~\ref{prop: def of the Sbc}, the restriction of $\SS^{b,c}\to \P^2$ on $\P^2 \setminus \{ [0:1:0] \}$ is given by gluing $\P^1 \times U_0$ and $\P^1 \times U_2$ with
\[
\begin{array}{rcccccc}
\nu_{02} \colon & \P^1 \times U_0 & \dasharrow & \P^1 \times U_2 \\
& \left(\begin{bmatrix} x_0 \\ x_1 \end{bmatrix}, [1:u:v] \right) 
& \longmapsto &
\left(
\begin{bmatrix} \beta_{11}(u,v) & \beta_{12}(u,v)\\ \beta_{21}(u,v) & \beta_{22}(u,v) \end{bmatrix}
\begin{bmatrix} x_0 \\ x_1 \end{bmatrix},\;
\left[ \frac{1}{v} : \frac{u}{v} : 1 \right]
\right),
\end{array}
\]
where $\beta_{11}, \beta_{12}, \beta_{21}, \beta_{22} \in \k[u,v]$ are the polynomials such that
\[ \hspace{-15mm}
\begin{bmatrix} 
\beta_{11}(y+z,yz) & \beta_{12}(y+z,yz)\\ 
\beta_{21}(y+z,yz) & \beta_{22}(y+z,yz) 
\end{bmatrix}
=
\frac{1}{(y-z)^{l+1}}
\begin{bmatrix}
(y-z)^{l+1} (y^{b+l+1}+z^{b+l+1}) & (y-z)^{2(l+1)}(y^{b+l+1}-z^{b+l+1})\\
(yz)^{l+1}(y^{b+l+1}-z^{b+l+1}) & (y-z)^{l+1} (yz)^{l+1} (y^{b+l+1}+z^{b+l+1})
\end{bmatrix}.
\]
We then observe that $y+z$ divides $y^{b+l+1}+z^{b+l+1}$ if $b+l+1$ is odd, and divides $y^{b+l+1}-z^{b+l+1}$ if $b+l+1$ is even, since in each case the corresponding polynomial vanishes when $z=-y$. Put $n=b+l+1$. On $L=\{Y=0\}$, we have $u=0$, hence $z=-y$ and
$v=yz=-y^2$. If $n$ is odd, then
\[
y^n+(-y)^n=0,
\qquad
y^n-(-y)^n=2y^n,
\]
whereas, if $n$ is even,
\[
y^n-(-y)^n=0,
\qquad
y^n+(-y)^n=2y^n.
\]
It follows that, in $\PGL_2(\k(v))$,
\[
\begin{bmatrix}
\beta_{11}(0,v) & \beta_{12}(0,v)\\
\beta_{21}(0,v) & \beta_{22}(0,v)
\end{bmatrix}
=
\begin{cases}
\begin{bmatrix}
0 & (-4)^{l+1}\\
1 & 0
\end{bmatrix}
& \text{if $b+l+1$ is odd},\\[4mm]
\begin{bmatrix}
1 & 0\\
0 & v^{l+1}
\end{bmatrix}
& \text{if $b+l+1$ is even}.
\end{cases}
\]
Indeed, in the first case,
\[
\frac{\beta_{12}(0,v)}{\beta_{21}(0,v)}
=
\frac{(2y)^{2(l+1)}}{(-y^2)^{l+1}}
=
(-4)^{l+1},
\]
while in the second case
\[
\frac{\beta_{22}(0,v)}{\beta_{11}(0,v)}
=
(-y^2)^{l+1}
=
v^{l+1}.
\]
In the first case the transition function is constant, so
$\pi^{-1}(L)\simeq\P^1\times\P^1=\F_0$. In the second case the transition
function is $[x_0:x_1]\longmapsto[x_0:v^{l+1}x_1]$,
which is the standard transition function of $\F_{l+1}$. Hence
\[
\pi^{-1}(L)\simeq
\begin{cases}
\F_0 & \text{if $b+l+1$ is odd},\\
\F_{l+1} & \text{if $b+l+1$ is even}.
\end{cases}
\]
This proves~\ref{FnLineSchwarz}.

\smallskip

We now use~\ref{FnLineSchwarz} to prove~\ref{item:aut-Sbc}.
Suppose that $c\neq 3b$, or equivalently, $b\neq l+1$.
By~\ref{FnLineSchwarz}, the lines $L$ for which
$\pi^{-1}(L)\simeq\F_b$ are precisely the tangent lines
to $\Gamma$. Thus $\Autz(\SS^{b,c})$ preserves this
family of lines, and hence preserves $\Gamma$.
Together with the inclusion proved above, this gives
\[
\Autz(\SS^{b,c})
\simeq \Aut(\P^2,\Gamma)
\simeq \PGL_2.
\]

We now assume that $c=3b$, so $b=l+1$.
If $l=0$, then $b=1=p^0$.
Otherwise, let $s\geq1$ be as in
Definition~\ref{def:Schwarz bc}.
Since $p^{s-1}\leq l<p^s$ and
$2(l+1)=b+l+1\equiv0\pmod{p^s}$,
the assumption $p \geq 3$ gives $b=l+1=p^s$.
Thus, in both cases, $b=l+1=p^s$ for some $s\geq0$.
The transition function $\nu_{02}$ is then given by
\begin{align*}
\hspace{-15mm} 
\begin{bmatrix} 
\beta_{11}(y+z,yz) & \beta_{12}(y+z,yz)\\ 
\beta_{21}(y+z,yz) & \beta_{22}(y+z,yz) 
\end{bmatrix}
&=
\frac{1}{(y-z)^{p^s}}
\begin{bmatrix}
(y-z)^{p^s} (y^{2{p^s}}+z^{2{p^s}}) & (y-z)^{2{p^s}}(y^{2{p^s}}-z^{2{p^s}})\\
(yz)^{p^s}(y^{2{p^s}}-z^{2{p^s}}) & (y-z)^{p^s} (yz)^{p^s} (y^{2{p^s}}+z^{2{p^s}})
\end{bmatrix}\\
&=
\begin{bmatrix}
 (y^2+z^2)^{p^s} & (y-z)^{2p^s}(y+z)^{p^s}\\
(yz)^{p^s}(y+z)^{p^s} & (yz)^{p^s} (y^2+z^2)^{p^s}
\end{bmatrix}
\end{align*}
from which we deduce
\[
\begin{bmatrix} 
\beta_{11}(u,v) & \beta_{12}(u,v)\\ 
\beta_{21}(u,v) & \beta_{22}(u,v) 
\end{bmatrix}
=
\begin{bmatrix}
 (u^2-2v)^{p^s} & (u^2-4v)^{p^s}u^{p^s}\\
u^{p^s}v^{p^s} & v^{p^s} (u^2-2v)^{p^s}
\end{bmatrix}.
\]
In $\GL_2(\k[u,v,1/v])$, we have the identity
\[ \hspace{-15mm}
\begin{bmatrix} 
-1 & \left(\frac{u}{v}\right)^{p^s}
\\[2mm]
0 & 2
\end{bmatrix}
\begin{bmatrix}
 (u^2-2v)^{p^s} & (u^2-4v)^{p^s}u^{p^s}\\
u^{p^s}v^{p^s} & v^{p^s} (u^2-2v)^{p^s}
\end{bmatrix}
\begin{bmatrix} 
u^{p^s} & 2\\ 
-1 & 0 
\end{bmatrix}
= 4v^{p^s}
\begin{bmatrix} 
0 & 1\\ 
v^{p^s} & u^{p^s} 
\end{bmatrix}.
\]
The matrices on the left and right define changes of
trivialization on $U_2$ and $U_0$, respectively.
After removing the invertible scalar $4v^{p^s}$,
the resulting matrix is the transition matrix of $\L^s$
from Remark~\ref{rk: theta02}.
The resulting isomorphism extends across $[0:1:0]$
by \cite[Lemma~2.3.1]{BFT}.
Hence $\SS^{b,3b}\simeq\L^s\simeq\PGL_3/P_s$,
and $\Autz(\SS^{b,3b})\simeq\PGL_3$.
This proves~\ref{item:aut-Sbc}.

\smallskip

Assertion~\ref{item:open-orbit-Sbc} is exactly
Proposition~\ref{prop:properties of the Tbc}\ref{item: open orbit in the quotient}, since
$\SS^{b,c}=\T^{b,c}/\langle\sigma\rangle$.
\end{proof}

The following result completes our description of the geometry of
$\SS^{b,c}$ by describing its orbits under $\Autz(\SS^{b,c})$.

\begin{corollary}\label{cor: description of the orbits for Sbc}
Recall that we assume $p \geq 3$.
Suppose that $c\neq 3b$. The action of $\Autz(\SS^{b,c})\simeq\PGL_2$ on $\SS^{b,c}$ has four or five orbits:
one open orbit of dimension $3$, two orbits of dimension $2$, and one or two
closed orbits of dimension $1$. More precisely:
\begin{itemize}
\item
If $p\mid \binom{l+b+1}{l+1}$, then $\Autz(\SS^{b,c})$ has exactly two
invariant curves $\ell$ and $\ell'$. The preimage of the conic
$\Gamma\subset\P^2$ in $\SS^{b,c}$ is isomorphic to
$\F_{2(l+1)}$, and $\ell$ and $\ell'$ are sections of
self-intersection $2(l+1)$ and $-2(l+1)$, respectively.

\item
If $p\nmid \binom{l+b+1}{l+1}$, then $l+1=p^s$ for some $s\geq 0$, and
$\Autz(\SS^{b,c})$ has exactly one invariant curve $\ell$. The preimage
of the conic $\Gamma\subset\P^2$ in $\SS^{b,c}$ is isomorphic to
$\P^1\times\P^1$, endowed with the Frobenius-twisted diagonal
$\PGL_2$-action
\[
g\cdot(x,y)=\bigl(F^s(g)\cdot x,g\cdot y\bigr),
\]
where $F^s\colon\PGL_2\to\PGL_2$ denotes the $s$-th Frobenius
morphism. Under this identification, $\ell$ is the graph of the
$s$-th Frobenius morphism $F^s\colon\P^1\to\P^1$, and hence has
bidegree $(1,p^s)=(1,l+1)$.
\end{itemize}
\end{corollary}

\begin{proof}
The diagonal $\PGL_2$-action on
$\F_0=\P^1\times\P^1$ has exactly two orbits, namely the diagonal
$\Delta$ and its complement. Their images under $\kappa$ are respectively
$\Gamma$ and $\P^2\setminus\Gamma$.
By Proposition~\ref{prop:three case P1 bundles over F0}
\ref{P1P1sch}, the $\P^1$-bundle
$\T^{b,c}\to\F_0$ admits two $\PGL_2$-invariant
sections $S_1,S_2\simeq\P^1\times\P^1$ with
reduced intersection
$C=(S_1\cap S_2)_{\mathrm{red}} \simeq \P^1$.
Since $C$ is $\PGL_2$-invariant, its image in
$\F_0$ is the diagonal $\Delta$, and
$C\to\Delta$ is an isomorphism.
Moreover, the involution $\sigma$ exchanges
$S_1$ and $S_2$, since it exchanges the two
rulings of $\F_0$.
Let
\[
q\colon\T^{b,c}\longrightarrow
\SS^{b,c}=\T^{b,c}/\langle\sigma\rangle
\]
be the quotient morphism.
Since $\sigma$ exchanges $S_1$ and $S_2$, their common
image
\[
S=q(S_1)=q(S_2)
\]
is an invariant surface whose normalization is
$q|_{S_1}\colon S_1\to S$.
This morphism restricts to an isomorphism
$S_1\setminus C\simeq S\setminus\ell$,
where $\ell=q(C)$.
Moreover, $\pi\circ q|_{S_1}$ identifies with $\kappa$.
In particular, $S\setminus\ell$ is a single
two-dimensional $\PGL_2$-orbit.

Over $\P^2\setminus\Gamma$, the stabilizer of a point of the base is the
normalizer of a maximal torus of $\PGL_2$. Its action on the corresponding
$\P^1$-fiber has exactly two orbits: an open orbit and an orbit consisting
of the images of two distinct points. Hence
$\pi^{-1}(\P^2\setminus\Gamma)$ is the disjoint union of an open
$\PGL_2$-orbit and the two-dimensional orbit $S\setminus\ell$.

On the other hand, $\kappa\colon\P^1\times\P^1\to\P^2$ restricts to a
$\PGL_2$-equivariant isomorphism from the diagonal
$\Delta\subset\P^1\times\P^1$ onto the conic $\Gamma\subset\P^2$.
Therefore, it remains to determine the orbit decomposition of
\[
V:=\pi'^{-1}(\Delta)\simeq\pi^{-1}(\Gamma),
\]
where $\pi'\colon\T^{b,c}\to\P^1\times\P^1$ is the structure morphism.

Restricting the transition function of $\T^{b,c}$ from
Definition~\ref{def hat Schwrzenberger} to $V$, and using
\[
\sum_{i=0}^{b}\binom{l+i}{l}
=\binom{l+b+1}{l+1}
\qquad\text{and}\qquad
c=b+2(l+1),
\]
we obtain
\[
([x_0:x_1;1:z],z)
\longmapsto
\left(
\left[
x_0:
x_1z^c+x_0z^{b+l+1}\binom{l+b+1}{l+1};
1:z
\right],
\frac1z
\right)
\]
\[
=
\left(
\left[
x_0:
x_1z^{2(l+1)}
+x_0z^{l+1}\binom{l+b+1}{l+1};
\frac1z:1
\right],
\frac1z
\right).
\]
The curve $\ell$ corresponds to $x_0=0$ in both charts. We now distinguish
two cases, according to whether or not
$p$ divides $\binom{l+b+1}{l+1}$.

\begin{itemize}
\item Suppose first that $p$ divides $\binom{l+b+1}{l+1}$. In this case, the transition function becomes
\[
([x_0:x_1;1:z],z)
\longmapsto
\left([x_0:x_1z^{2(l+1)};z^{-1}:1],z^{-1}\right),
\]
which shows that $V\simeq \F_{2(l+1)}$, via the isomorphism given on the two standard charts by
\[
\begin{array}{rcl}
([x_0:x_1;1:z],z)
&\longmapsto&
[x_1:x_0;1:z],\\[2mm]
([x_0:x_1;z:1],z)
&\longmapsto&
[x_1:x_0;z:1].
\end{array}
\]
Under this isomorphism, the curve $\ell$ is sent onto the section of self-intersection $2(l+1)$. 
Since the $\PGL_2$-action on $V\simeq\F_{2(l+1)}$
is transitive on the base $\P^1$ and preserves
the section $\ell$ of positive self-intersection,
Proposition~\ref{prop:trichotomy for algebraic subgroups of AUt(Fa)}
implies that $V$ has exactly three $\PGL_2$-orbits:
$\ell$, the unique section $\ell'$ of
self-intersection $-2(l+1)$, and their complement.

\smallskip

\item  Suppose now that
$p\nmid t:=\binom{l+b+1}{l+1}$.
The following formulas define an isomorphism
$V\simeq\P^1\times\P^1$ over $\Delta$:
\[\begin{array}{rcl}
([x_0:x_1;1:z],z)&\mapsto& ([x_1:tx_0+x_1z^{l+1}],[1:z])\\
([x_0:x_1;z:1],z)&\mapsto &([-tx_0+x_1z^{l+1}:x_1],[z:1]).
\end{array}\]

Let us check that $l+1=p^s$ for some $s\geq 0$.
If $l=0$, there is nothing to prove, since $l+1=p^0$.
Assume therefore that $l>0$. By
Definition~\ref{def:Schwarz bc}, there exists $s\geq 1$ such that
\[
p^{s-1}\leq l<p^s
\qquad\text{and}\qquad
p^s\mid b+l+1.
\]
Write the base-$p$ expansions
\[
b+l+1=\sum_{i\geq 0}a_i p^i,
\qquad
l+1=\sum_{i\geq 0}c_i p^i.
\]
Since $p^s\mid b+l+1$, we have
$a_0=\cdots=a_{s-1}=0$. By Lucas' theorem,
\[
\binom{l+b+1}{l+1}
\equiv
\prod_{i\geq0}\binom{a_i}{c_i}
\pmod p.
\]
As this binomial coefficient is nonzero modulo $p$, each factor is
nonzero. Hence $c_0=\cdots=c_{s-1}=0$, and therefore
$p^s\mid l+1$. Since $l<p^s$, we obtain $l+1=p^s$.

Under the isomorphism $V\simeq\P^1\times\P^1$,
the second projection is the morphism $V\to\Delta\simeq\P^1$.
Consequently, the action of $\PGL_2$ on the second factor is the
standard one. Moreover, every automorphism of
$\P^1\times\P^1$ preserving the second projection acts on the first
factor by an element of $\PGL_2$. Hence the action is of the form
\[
g\cdot(x,y)=\bigl(\varphi(g)\cdot x,g\cdot y\bigr)
\]
for some homomorphism
$\varphi\colon\PGL_2\to\PGL_2$.
Since $\ell$ is $\PGL_2$-invariant and, under the above identification,
is the graph of
$F^s\colon\P^1\to\P^1$,
we have
\[
\varphi(g)\cdot F^s(y)
=
F^s(g\cdot y)
=
F^s(g)\cdot F^s(y)
\]
for every $g\in\PGL_2$ and $y\in\P^1$. It follows that
$\varphi(g)=F^s(g)$,
and therefore the action on $V$ is the Frobenius-twisted diagonal
action
\[
g\cdot(x,y)=\bigl(F^s(g)\cdot x,g\cdot y\bigr).
\]
This action has exactly two orbits: $\ell$ and $V \setminus \ell$. 
In particular, $\ell$ is the unique
$\PGL_2$-invariant curve in $V$.
\end{itemize} 
This concludes the proof of the result.
\end{proof}

\subsection{Classification of the \texorpdfstring{$\P^1$}{P1}-bundles over the projective plane}
\label{sec: classification over P2}
The main result of this section is Theorem~\ref{th:P1bundlesP2}.
Before proving it, we establish two preliminary results concerning the possible actions
of $\Autz(X)$ on the base.

\smallskip

We first recall a classical result concerning the classification of connected algebraic subgroups of $\PGL_3$.

\begin{lemma} \emph{(\cite[Lemma~4.3.3]{BFT})}\label{SubgroupsPGL3}
Assume that $p \geq 3$. If $G\subsetneq \Aut(\P^2)=\PGL_3$ is a proper connected algebraic subgroup that does not fix any point of $\P^2$, then one of the following holds. Either
\begin{enumerate}
\item \label{PGL2case}
$G=\Aut(\P^2,\Gamma)\simeq \PGL_2$, where $\Gamma\subseteq \P^2$ is a smooth conic; or
\item\label{SL2GL2case}
$G$ is a subgroup of $\Aut(\P^2,L)\simeq \GL_2 \ltimes \G_a^2 $, where $L\subseteq \P^2$ is a line, equal either to $\Aut(\P^2,L)$ or to $\SL_2\ltimes \G_a^2\subsetneq \Aut(\P^2,L)$.
\end{enumerate} 
Moreover, in both cases $G$ acts on $\P^2$ with two orbits: the smooth conic $\Gamma$ in Case~\ref{PGL2case} or the line $L$ in Case~\ref{SL2GL2case}, and its open complement.
\end{lemma}

The next result will also be an essential ingredient in the proof of Theorem~\ref{th:P1bundlesP2}.

\begin{lemma} \label{lem: contracting -1 curve yields a fixed point}
Let $\hat{\pi}\colon \hat{X} \to \F_1$ be a $\P^1$-bundle such that 
$\hat{H} = \Im(\Autz(\hat{X}) \to \Aut(\F_1))$ contains a conjugate of $\SL_2\subseteq \Aut(\F_1)$.  
Let $\pi\colon X \to \P^2$ be the $\P^1$-bundle obtained from $\hat{\pi}$ 
by contracting the $(-1)$-section of $\F_1$ to a point $p_0 \in \P^2$ 
$($see Lemma~$\ref{lem:descent-lemma})$.
Let $H=\Im(\Autz(X)\to\Aut(\P^2))$. Then we have the following distinct possibilities:
\begin{enumerate}
\item \label{item 0: contracting -1 curve yields a fixed point}the $\P^1$-bundle $X\to\P^2$ is  decomposable, and $H=\Aut(\P^2)$;
\item the $\P^1$-bundle $X\to\P^2$ is square-isomorphic to the $\P^1$-bundle
$\L^{s,b}\to\P^2$ of Definition~$\ref{def: P1 bundles Lsb}$, for some $s \geq 0$ and $1 \leq b \leq p^s$, and $H$ contains $\Aut(\P^2,L)$; or
\item \label{item ii: contracting -1 curve yields a fixed point}$H$ fixes the point $p_0$.
\end{enumerate}
\end{lemma}

\begin{proof}
Let $\tau\colon\F_1\to\P^2$ be the contraction.
By the equivariance in Lemma~\ref{lem:descent-lemma},
$H$ contains the image of the given $\SL_2$-subgroup,
which fixes $p_0$.
The corresponding connected group acting on $X$
also acts on $X\times_{\P^2}\F_1$, with horizontal
image containing the same $\SL_2$-subgroup.
We may therefore replace $\hat X$ by this pullback:
the hypothesis is preserved, and its descent is still $X$.
	
The assumption that $\hat{H}$ contains a conjugate of $\SL_2$ implies that the $\P^1$-bundle $\hat{\pi}\colon \hat{X} \to \F_1$ is either decomposable or square-isomorphic to an Umemura $\P^1$-bundle (see Definition~\ref{def Umemura bundle}). In the decomposable case, it follows from the descent lemma (see Lemma~\ref{lem:descent-lemma}) that $X \to \P^2$ is a decomposable $\P^1$-bundle. In this situation, we have $H = \Aut(\P^2)$ by Lemma~\ref{lemm:surjectiveDecP2} (and directly for $\P^1\times\P^2$), which yields Case~\ref{item 0: contracting -1 curve yields a fixed point}.

We can therefore assume that $\hat{X} = \U_{1,\omega}^{b,c} \to \F_1$ for some $(b,c,\omega)$ as in Definition~\ref{def Umemura bundle}. 
We fix homogeneous coordinates $[y_0:y_1;z_0:z_1]$ on $\F_1$ and $[x:y:z]$ on $\P^2$.  
The blow-up of the point $p_0 = [0:0:1]$ is given by
\[
[y_0:y_1; z_0:z_1] \mapsto [y_0 z_0 : y_0 z_1 : y_1].
\]  
The standard $\SL_2$-subgroup acts on $\F_1$ with three orbits: 
the $(-1)$-section $s_{-1}=\{y_0=0\}$, 
the $1$-section $s_1=\{y_1=0\}$, and the open complement of these two disjoint sections.  
We denote by $L = \{z=0\}$ the image of $s_1$ in $\P^2$.

The restriction of $X \to \P^2$ to a line through $p_0$ is then isomorphic to the restriction of $\hat{X} \to \F_1$ over a fiber of $\F_1 \to \P^1$, i.e., to $\F_b$. On the other hand, if $L' \neq L$ is a line not passing through $p_0$, one can use the transition function of $\U_{1,\omega}^{b,c}$ to determine the restriction of $X \to \P^2$ to $L'$. Since $\SL_2$ acts transitively on lines distinct from $L$ and not passing through $p_0$, we may take $L'=\{y=z\}$ in $\P^2$,
whose strict transform is given on the first chart
by $y_1=y_0z$.
 Denote by $\F_m$ the Hirzebruch surface obtained by restricting the $\P^1$-bundle $X \to \P^2$ to $L'$. If $m\neq b$, the lines whose inverse images are
 isomorphic to $\F_b$ form the pencil through $p_0$,
 possibly together with the isolated line $L$.
 Thus $H$ preserves this pencil and fixes $p_0$,
 giving Case~\ref{item ii: contracting -1 curve yields a fixed point}. Therefore, it remains to consider the case where $m=b$. We now compute the integer $m$.

The transition function of $\U_{1,\omega}^{b,c}$ on the overlap of the two charts restricts as
\[
([x_0:x_1;1:z], z) \mapsto ([z^b x_0: z^c x_1 + Q(z)x_0; 1:1], \frac{1}{z}), \quad 
Q(z) = \sum_{i\in I} \lambda_i z^{b + p^{s_i}},
\]
where $I=\{i\in I_{1,b,c}\mid\lambda_i\neq0\}$ is the support of $\omega$.
Set $i_+=\max I$. Then $s_i>s_{i_+}$ for every $i\in I\setminus\{i_+\}$. 
This corresponds to the matrix
\[
M:=\begin{bmatrix} z^b & 0 \\ Q(z) & z^c \end{bmatrix} \in \PGL_2(\k[z,1/z]).
\]
It follows from the definition of Umemura $\P^1$-bundles that we always have $c \leq b+2p^{s_{i_+}}$.
We distinguish several cases:
\begin{itemize}
\item
If $c \le b + p^{s_{i_+}}$, then $z^{-c} Q(z) \in \k[z]$. Hence we may replace $M$ with
\[
M \cdot 
\begin{bsmallmatrix}
1 & 0 \\
- z^{-c} Q(z) & 1
\end{bsmallmatrix}
=
\begin{bsmallmatrix}
z^b & 0 \\
0 & z^c
\end{bsmallmatrix},
\]
which corresponds to the transition function of $\F_{\lvert b - c\rvert}$.

    \item If $b + p^{s_{i_+}} < c \le b + 2p^{s_{i_+}}$, then since
\[
2p^{s_{i_+}} \le p^{1+s_{i_+}} \le p^{s_i}
\quad \text{for all } i\in I\setminus\{i_+\},
\]
we have $b - c + p^{s_i} \ge 0$.
Thus $\lambda_i z^{b-c+p^{s_i}} \in \k[z]$, and therefore we may replace $M$ with 
\[
M \cdot 
\begin{bsmallmatrix}
1 & 0 \\
-\sum_{i\in I\setminus\{i_+\}} \lambda_i z^{b-c+p^{s_i}} & 1
\end{bsmallmatrix}
=
\begin{bsmallmatrix}
z^b & 0 \\
\lambda_{i_+} z^{b+p^{s_{i_+}}} & z^c
\end{bsmallmatrix} 
=
\begin{bsmallmatrix}
1 & 0 \\
\lambda_{i_+} z^{p^{s_{i_+}}} & z^{c-b}
\end{bsmallmatrix} \in \PGL_2(\k[z,1/z]).
\]
After constant changes of trivialization, the remaining
matrix is
\[
N=\begin{bmatrix}1&0\\ z^q&z^d\end{bmatrix},
\qquad q=p^{s_{i_+}},\quad d=c-b,
\]
where $q<d\leq2q$.
The identity
\[
\begin{bmatrix}1&-z^{-q}\\0&1\end{bmatrix}
N
\begin{bmatrix}z^{d-q}&1\\-1&0\end{bmatrix}
=
\begin{bmatrix}z^{d-q}&0\\0&z^q\end{bmatrix}
\]
uses changes of trivialization regular on the target
and source charts, respectively.
Hence $m=2q-d=b-c+2p^{s_{i_+}}$.
\end{itemize}  

\smallskip

Recall also that (see \eqref{eq: relation betweenc and lj}):
\[\forall i\in I,\ c = i + 2p^{s_i}.\]
\begin{itemize}
\item
If $c \le b$, then $m=b-c<b$, which is  impossible.
\item If $b < c \le b + p^{s_{i_+}}$, then $m = c - b$, hence $c = 2b$ and $b \le p^{s_{i_+}}$. The equality $i + 2p^{s_i} = c$ gives $i + 2p^{s_i} = 2b \le 2p^{s_{i_+}} \le 2p^{s_i}$. 
We find $i=0$ for every $i\in I$, so $I=\{0\}$.

\item If $b + p^{s_{i_+}} < c \le b + 2p^{s_{i_+}}$, then $m = b - c + 2p^{s_{i_+}}$, hence $c = 2p^{s_{i_+}}$ and $b < p^{s_{i_+}}$. The equality $i + 2p^{s_i} = c$ gives $i + 2p^{s_i} = 2p^{s_{i_+}}\le 2p^{s_i}$. Thus $i=0$ and $s_i=s_{i_+}$ for every $i\in I$,
so again $I=\{0\}$.
\end{itemize}

Therefore, under the assumption $m=b$, we obtain $I=\{0\}$, hence $c=2p^s$, with $s=s_0$. We moreover have 
\[s \geq 0, \ c=2p^s,\ \text{and}\ 1 \leq b \leq p^s.\] 
After normalizing the nonzero coefficient $\lambda_0$,
we obtain
\[\hat{X}=Z_{1}^{b,c,P}\to \F_1 \ \text{ with }\ P(y_0,y_1,z)= y_1^b z^{p^s}.\]  
Lemma~\ref{Lem:LsbBlowUp} then proves that $X\to \P^2$ is the $\P^1$-bundle
$\L^{s,b} \to \P^2$ of Definition~\ref{def: P1 bundles Lsb}.
Then, Proposition~\ref{prop: LsbAut} shows that $H$ contains $\Aut(\P^2,L)$.
\end{proof}

We can now state and prove the main result of Section~\ref{section: P1 bundles over P2}.

\begin{theorem}\label{th:P1bundlesP2}
Assume that $p \geq 3$.
Let $\pi\colon X\to\P^2$ be a $\P^1$-bundle and let $H:=\Im(\Autz(X)\to\Aut(\P^2))$. Then one of the following holds:
\begin{enumerate}
\item \label{item i: Prop:P1bundlesP2} $H=\Aut(\P^2)$, and $\pi$ is either a decomposable $\P^1$-bundle or square-isomorphic to the $\P^1$-bundle $\L^{s} \to \P^2$ $($Definition~$\ref{def: P1 bundles Lsb})$, for some $s \geq 0$;
\item \label{item ii: Prop:P1bundlesP2} $H=\Aut(\P^2,\Gamma)\simeq\PGL_2$, where $\Gamma\subset\P^2$ is a smooth conic, and $\pi$ is square-isomorphic to a Schwarzenberger $\P^1$-bundle $\SS^{b,c} \to \P^2$, for some $b \geq 1$ and $c \geq 3$ as in Definition~$\ref{def:Schwarz bc}$ with $c \neq 3b$;
\item \label{item iii: Prop:P1bundlesP2} $H = \Aut(\P^2,L)\simeq \GL_2\ltimes\G_a^2$, where $L\subset\P^2$ is a line, and $\pi$ is square-isomorphic to the $\P^1$-bundle $\L^{s,b}\to\P^2$ $($Definition~$\ref{def: P1 bundles Lsb})$, for some $s\ge1$ and $1\le b< p^s$; or
\item \label{item iv: Prop:P1bundlesP2} there exists an $H$-fixed point $p_0\in\P^2$, and the blow-up $\eta\colon \hat X\to X$
of the fiber $\pi^{-1}(p_0)$ yields a commutative diagram of $\Autz(X)$-equivariant maps
\[
\xymatrix@R=4mm@C=2cm{
\hat X \ar[r]^{\eta} \ar[d]_{\hat\pi} & X \ar[d]^\pi \\
\F_1 \ar[r]_{\tau} & \P^2
}
\]
where $\hat{X}\to \F_1$ is a $\P^1$-bundle and $\tau$ is the blow-up of $p_0$.
\end{enumerate}
\end{theorem}

\begin{proof}
Let $H\subseteq \PGL_3$ be the image of the natural homomorphism $\Autz(X)\to \Aut(\P^2)$.
If $H$ fixes a point $p_0\in \P^2$, we blow up simultaneously $p_0$ in $\P^2$ and $f=\pi^{-1}(p_0)$ in $X$. Computing explicitly in a trivializing local chart containing $p_0$, we see that $\hat{X}\to\F_1$ is again a $\P^1$-bundle. Moreover, since $p_0$ and $f$ are $\Autz(X)$-stable in $\P^2$ and $X$, respectively, the group $\Autz(X)$ acts on $\hat{X}$. This yields Case~\ref{item iv: Prop:P1bundlesP2}.

Suppose now that $H$ fixes no point.
If $H=\PGL_3$, it contains
$\SL_2\ltimes\G_a^2\subseteq\Aut(\P^2,L)$.
If $H\neq\PGL_3$, Lemma~\ref{SubgroupsPGL3} applies.
Thus there are two possibilities:
\begin{itemize}
\item
$H$ contains $\SL_2\ltimes \G_a^2\subseteq  \Aut(\P^2,L)=\GL_2\ltimes \G_a^2$.  In this case, let $p_0\in \P^2\setminus L$, and
set
\[
G_0:=\bigl((G_{p_0})_{\mathrm{red}}\bigr)^0,
\qquad
H_0:=\bigl((H_{p_0})_{\mathrm{red}}\bigr)^0,
\qquad G:=\Autz(X).
\]
Then $G_0$ and $H_0$ are connected algebraic groups,
the induced homomorphism $G_0\to H_0$ is surjective,
and $H_0$ contains $\SL_2$.
Blowing up $p_0$ and its fiber yields a $\P^1$-bundle $\hat{X}\to\F_1$ with a $G_0$-action.  
Lemma~\ref{lem: contracting -1 curve yields a fixed point} implies that either $\pi\colon X \to \P^2$ is decomposable, in which case $H = \Aut(\P^2)$ and we obtain Case~\ref{item i: Prop:P1bundlesP2}, or $\pi$ is square-isomorphic to the $\P^1$-bundle $\L^{s,b} \to \P^2$ for some $s \ge 0$ and $1 \le b \le p^s$. If $b<p^s$, then $s\ge 1$ and Proposition~\ref{prop: LsbAut} yields $H=\Aut(\P^2,L)$, so we get Case~\ref{item iii: Prop:P1bundlesP2}.
If $b=p^s$, $\L^{s,b}=\L^{s}$ and Proposition~\ref{prop: LsbAut} yields $H=\Aut(\P^2)$, so we get Case~\ref{item i: Prop:P1bundlesP2}.

\item Otherwise, $H$ is conjugate to $\Aut(\P^2,\Gamma)\simeq\PGL_2$, where $\Gamma\subset\P^2$ is a smooth conic. 
As $p \geq 3$, we change coordinates and assume that $\Gamma$ is given by
\[
\{[x:y:z]\in \P^2 \mid 4xz = y^2\}.
\]
We consider the double covering
\[
\begin{array}{rccc}
\kappa \colon & \F_0 & \to & \P^2 \\
& ([y_0:y_1],[z_0:z_1]) & \mapsto & [y_0 z_0 : y_0 z_1 + y_1 z_0 : y_1 z_1],
\end{array}
\]
branched along $\Gamma$. We consider the pullback $\P^1$-bundle
\[
\hat{\pi}\colon \hat{X} := X \times_{\P^2} \F_0 \to \F_0.
\]
The group $\Autz(X)$ acts biregularly on $\hat{X}$, and its induced action on $\F_0$ is the diagonal action of $\PGL_2$.
By Lemma~\ref{lemma:horizontal_part_Tbc}, either $\hat{X} \to \F_0$ is decomposable, or $\hat{X} \simeq \T^{b,c} \to \F_0$ for some $b,c \in \N$ satisfying Definition~\ref{def hat Schwrzenberger}.

We first assume that $\hat{X} \to \F_0$ is decomposable. Arguing as in the proof of \cite[Lemma~4.3.1]{BFT}, which only requires $p \neq 2$, we deduce that $X \to \P^2$ is also decomposable. Hence $H=\Aut(\P^2)\simeq \PGL_3$, contradicting our assumption that $H\simeq \PGL_2$.

It remains to consider the case where $\hat{X}\to \F_0$ is isomorphic over $\F_0$ to $\T^{b,c}\to\F_0$ for some $b,c\in \N$ satisfying Definition~\ref{def hat Schwrzenberger}. Let $\sigma\in\Aut(\hat{X})$ be the biregular involution of $\hat{X}=X\times_{\P^2}\F_0$ acting trivially on $X$ and exchanging the two factors of $\F_0$. Then $X=\hat{X}/\left\langle \sigma \right \rangle$.
Since $\hat{X}\simeq \T^{b,c}$, the involution $\sigma$ coincides with that of Proposition~\ref{prop:properties of the Tbc}\ref{item unique involution Tbc}. By Lemma~\ref{TbcInvolutionExplicit}, the smoothness of $X=\T^{b,c}/\left \langle \sigma \right \rangle$ implies that $l$ is even. Hence $X\to\P^2$ is isomorphic to the Schwarzenberger $\P^1$-bundle $\SS^{b,c}\to\P^2$ of Definition~\ref{def:Schwarz bc}. Since $H\simeq\PGL_2$,
Proposition~\ref{prop:aut-orbits-Sbc}\ref{item:aut-Sbc}
implies $c\neq3b$.
This gives Case~\ref{item ii: Prop:P1bundlesP2}.
\qedhere 
\end{itemize}
\end{proof}

\section{Maximality and non-maximality results}\label{sec:max and non-max results} 
The results of Sections~\ref{sec: Lsb non max}, \ref{sec: maximality in dec case}, and~\ref{sec: max for Ume bundles}
hold in arbitrary positive characteristic, while those of Sections~\ref{sec: relative max of Sbc}, \ref{proof: classification rel max}  and
\ref{proof: equivariant links} require $p\geq3$.

We study the $\P^1$-bundles introduced in the previous sections from the point of view of relative maximality, stiffness, and superstiffness, as defined in the introduction. More precisely, for each family, we determine whether its connected automorphism group is relatively maximal and whether the corresponding $\P^1$-bundle is stiff or superstiff. When relative maximality or stiffness fails, we exhibit explicit equivariant birational maps witnessing this failure. 
We treat successively the different families appearing in our classification.

\subsection{The case of \texorpdfstring{$\L^{s,b} \to \P^2$}{Lsb}}\label{sec: Lsb non max}
We study the relative maximality of the automorphism groups of the
$\P^1$-bundles $\L^{s,b}\to\P^2$. The behavior depends on whether
$b=p^s$ or $b<p^s$: in the first case $\L^{s,b}=\L^s$ is a flag
variety and is superstiff, whereas in the second case its automorphism
group is not relatively maximal.

\begin{lemma}
Let $s\geq 0$, and let $\L^s\to\P^2$ be the $\P^1$-bundle of
Definition~$\ref{def: P1 bundles Lsb}$. Then $\Autz(\L^s)$ is relatively maximal and
$\L^s$ is superstiff.
\end{lemma}

\begin{proof}
	By Lemma~\ref{Lem:LsPGL3}, $\Autz(\L^s)$ acts
	transitively on $\L^s$.
	Any equivariant birational map from $\L^s$ to another
	$\P^1$-bundle is therefore everywhere defined, since
	its indeterminacy locus is invariant.
	It is a birational morphism, and its isomorphism locus
	is nonempty and invariant, hence equals $\L^s$.
	Thus every such map is an isomorphism, proving both assertions.
\end{proof}

\begin{lemma}\label{lem:Lsb-non-maximal}
Let $s\geq 1$ and $1\leq b<p^s$. 
Let $\L^{s,b} \to \P^2$ be the $\P^1$-bundle of
Definition~$\ref{def: P1 bundles Lsb}$.
Then there exists a birational map over $\P^2$
$\varphi\colon \L^{s,b}\dashrightarrow \P^1\times\P^2$
such that
\[
\varphi\Autz(\L^{s,b})\varphi^{-1}
\subsetneq
\Aut(\P^1\times\P^2).
\]
In particular, $\Autz(\L^{s,b})$ is not relatively maximal.
\end{lemma}

\begin{proof}
	The trivialization
	$\pi^{-1}(U_2)\simeq\P^1\times U_2$
	defines a birational map
	$\varphi\colon\L^{s,b}\dashrightarrow\P^1\times\P^2$
	over $\P^2$.
	By Proposition~\ref{prop: LsbAut}\ref{item: iso in LsbAut},
	the action in this trivialization is constant on the
	$\P^1$-factor and extends to $\P^1\times\P^2$.
	Thus $\varphi$ conjugates $\Autz(\L^{s,b})$ into
	$\Aut(\P^1\times\P^2)$.
	The inclusion is strict because its image on the base
	preserves $L$, whereas $\Aut(\P^1\times\P^2)$ induces
	the full group $\PGL_3$.
\end{proof}

\subsection{The case of decomposable \texorpdfstring{$\P^1$}{P1}-bundles over \texorpdfstring{$\P^2$}{P2} and \texorpdfstring{$\F_a$}{Fa}} \label{sec: maximality in dec case}
We recall the results concerning the relative maximality and
stiffness of decomposable $\P^1$-bundles over $\P^2$ and $\F_a$.
The first lemma describes the possible equivariant elementary links
between $\P^1$-bundles, and the subsequent results give the complete
picture in the decomposable case.

\begin{lemma}[{\cite[Lemma~5.3.1]{BFT}}]\label{Lemm:DecElLinks}
Let $G$ be a connected algebraic group acting on a $\P^1$-bundle
$\pi\colon X\to S$, and let $H\subset\Autz(S)$ be the image of $G$. 
Assume either that no curve of $S$ is invariant under $H$,
	or that $(H,S)$ is one of the following pairs:
	\[
	(K,\F_a),\quad a\geq1,\qquad
	(\Autz(\P^1\times\P^1,\Delta),\P^1\times\P^1),
	\qquad \text{or} \qquad
	(\Aut(\P^2,\Gamma),\P^2),
	\]
where $K \subset \Aut(\F_a)$ acts on $\F_a$ with two orbits, $\Delta\subset\P^1\times\P^1$ is the diagonal, and $\Gamma\subseteq\P^2$ is a smooth conic.

Let $\pi'\colon X'\to S$ be a $\P^1$-bundle and
$\varphi\colon X\dasharrow X'$ a $G$-equivariant birational map of
$\P^1$-bundles which is not an
isomorphism. Then there exist $\P^1$-bundles
$\pi_i\colon X_i\to S$, $i=0,\ldots,n$, with $\pi_0=\pi$ and
$\pi_n=\pi'$, such that
\[
\varphi=\varphi_n\circ\cdots\circ\varphi_1,
\]
where each $\varphi_i\colon X_{i-1}\dasharrow X_i$ is the blow-up of an
irreducible $G$-invariant curve $\ell_i\subseteq X_{i-1}$, followed by the contraction
of the strict transform of
$\pi_{i-1}^{-1}(\pi_{i-1}(\ell_i))$, and
\[
\pi_{i-1}|_{\ell_i}\colon
\ell_i\xrightarrow{\sim}\pi_{i-1}(\ell_i)
\]
identifies $\ell_i$ with either $s_{-a}\subset\F_a$ for some $a\geq1$,
$\Delta\subset\P^1\times\P^1$, or $\Gamma\subset\P^2$.

Moreover, if $n$ is minimal, the sequence
$\varphi_1,\ldots,\varphi_n$ is unique up to isomorphisms of
$\P^1$-bundles.
\end{lemma}

\begin{lemma}[{\cite[Lemma~5.7.1]{BFT}}]\label{Lem:DecP2NoCurve}
Let $b\geq 0$, and let $\PP_b\to \P^2$ be a decomposable $\P^1$-bundle.
\begin{enumerate}
\item
$\PP_b$ does not contain any $\Autz(\PP_b)$-invariant curve.
\item \label{AutPPbSuperstiff}
 $\Autz(\PP_b)$ is relatively maximal and $\PP_b$ is superstiff.
\end{enumerate}
\end{lemma}

\begin{proposition}[{\cite[Lemmas~5.2.2,~5.4.1 and~5.4.2, Corollary~5.4.3]{BFT}}]\label{prop:MaxDecBundlesFa}
\ \\ Let $a,b\geq 0$ and $c\in\Z$, and assume that $c\leq 0$ if $b=0$.
Then the following hold.
\begin{enumerate}
\item
If $a=0$, then $\Autz(\FF_a^{b,c})$ is relatively maximal and $\FF_a^{b,c}$ is superstiff.

\item
If $a\geq 2$ and $b=c=0$, then $\Autz(\FF_a^{b,c})$ is relatively maximal and $\FF_a^{b,c}=\FF_a^{0,0}\simeq\F_a\times\P^1$ is superstiff.
\item\label{FFabcOnlyMax}
Assume that $a\geq2$, $b\geq0$, and
$-a<c<ab$.
Then $\Autz(\FF_a^{b,c})$ is relatively maximal, but $\FF_a^{b,c}$ is not stiff.
More precisely, let $r\geq0$ be the smallest integer such that
$c-ra\leq0$. Then $r\leq b$, and there is an infinite sequence
\[
\FF_a^{b-r,c-ra}
\dasharrow\cdots\dasharrow
\FF_a^{b,c}
\dasharrow
\FF_a^{b+1,c+a}
\dasharrow\cdots\dasharrow
\FF_a^{b+n,c+an}
\dasharrow\cdots .
\]
Each elementary link 
is given by 
\[
\begin{array}{rccc}
\varphi_s\colon&
\FF_a^{b+s,c+as}&\dasharrow&
\FF_a^{b+s+1,c+a(s+1)}
\\
&
[x_0:x_1;y_0:y_1;z_0:z_1]
&\mapsto&
[x_0:x_1y_0;y_0:y_1;z_0:z_1]
\end{array}\quad \text{for $s \geq -r$.}
\]
It is obtained by blowing up the curve
$l_{00}=\{x_0=y_0=0\}$
and contracting the strict transform of
$\pi^{-1}(s_{-a})$ onto the curve
$l_{10}=\{x_1=y_0=0\}$.
Moreover,
\[
\varphi_s
\Autz(\FF_a^{b+s,c+as})
\varphi_s^{-1}
=
\Autz(\FF_a^{b+s+1,c+a(s+1)})
\]
for every $s\geq-r$. Up to isomorphisms of $\P^1$-bundles,
compositions of these links and their inverses give all
$\Autz(\FF_a^{b,c})$-equivariant square birational maps from
$\FF_a^{b,c}$ to another $\P^1$-bundle.

\item
In all remaining cases, $\Autz(\FF_a^{b,c})$ is not relatively maximal.
More precisely:
\begin{enumerate}
\item
If $a=1$, the birational morphism
$\eta\colon\F_1\to\P^2$ contracting $s_{-1}$ induces the square
birational map
\[
\begin{array}{rccc}
\varphi\colon&
\FF_1^{b,c}&\dasharrow&\PP_{b-c}\\
&
[x_0:x_1;y_0:y_1;z_0:z_1]
&\mapsto&
[x_0y_0^c:x_1;y_0z_0:y_1:y_0z_1],
\end{array}
\]
and
\[
\varphi\Autz(\FF_1^{b,c})\varphi^{-1}
\subsetneq
\Autz(\PP_{b-c}).
\]

\item
If $a\geq2$, $b \geq 0$, and $c\leq-a$, the elementary link $\varphi_0$ defined above satisfies
\[
\varphi_0\Autz(\FF_a^{b,c})\varphi_{0}^{-1}
\subsetneq
\Autz(\FF_a^{b+1,c+a}).
\]

\item
If $a\geq2$, $b>0$, and $c\geq ab$, there is a sequence of inverse
elementary links
\[
\FF_a^{b,c}
\dasharrow
\FF_a^{b-1,c-a}
\dasharrow\cdots\dasharrow
\FF_a^{1,c-(b-1)a}
\dasharrow
\FF_a^{0,c-ba}.
\]
For $i=1,\ldots,b$, the map
\[
\psi_i\colon
\FF_a^{b-i+1,c-(i-1)a}
\dasharrow
\FF_a^{b-i,c-ia}
\]
is the inverse of the elementary link
\[
\begin{array}{rccc}
&
\FF_a^{b-i,c-ia}& \dashrightarrow &\FF_a^{b-i+1,c-(i-1)a}\\
&
[x_0:x_1;y_0:y_1;z_0:z_1]
&\mapsto&
[x_0:x_1y_0;y_0:y_1;z_0:z_1].
\end{array}
\]
that is, it is the inverse of the birational transformation obtained
by blowing up $l_{00}$ and contracting the strict transform of
$\pi^{-1}(s_{-a})$ onto $l_{10}$.
The first $b-1$ links conjugate the corresponding connected
automorphism groups, whereas the last one yields
\[
\psi_b\cdots\psi_1\,
\Autz(\FF_a^{b,c})\,
(\psi_b\cdots\psi_1)^{-1}
\subsetneq
\Autz(\FF_a^{0,c-ab}).
\]
\end{enumerate}
\end{enumerate}
\end{proposition}

\subsection{The case of Schwarzenberger \texorpdfstring{$\P^1$}{P1}-bundles over \texorpdfstring{$\P^2$}{P2}}\label{sec: relative max of Sbc}
Throughout this subsection, we assume that $p\geq3$.
We study the relative maximality and stiffness of the Schwarzenberger $\P^1$-bundles $\SS^{b,c}\to\P^2$. By Proposition~\ref{prop:aut-orbits-Sbc}\ref{item:aut-Sbc}, we have $\Autz(\SS^{b,c})\simeq \PGL_2$ if $c\neq 3b$, and $\Autz(\SS^{b,c})\simeq \PGL_3$ if $c=3b$.

We first determine their possible
$\PGL_2$-equivariant birational models among decomposable
$\P^1$-bundles and the $\P^1$-bundles $\L^r \to \P^2$, and then combine this with
the classification of $\P^1$-bundles over $\P^2$ (Theorem~\ref{th:P1bundlesP2}) to obtain the
relative maximality and stiffness results (Proposition~\ref{prop: relative max of Sbc}).

\smallskip 

In the following, we use the convention $s=0$ when $l=0$.
Note that if $c=3b$, then $b=p^s$, $m=2p^s$, and
$\PGL_2\subsetneq\Autz(\SS^{b,c})$. In this case, the birational map
$\SS^{b,c}\dasharrow X$ given by
Lemma~\ref{lem: possible birational models Sbc}\ref{item: birational Sbc Lt}
below can be chosen to be an isomorphism. Moreover, all integers $r$
appearing in the statement are at least equal to the integer $s$
associated with $b,c$ (see Definition~\ref{def:Schwarz bc}), but this
will not be used in the sequel.

Recall that we denote by
$\psi\colon\PGL_2\hookrightarrow\PGL_3$
the standard embedding induced by the action on
$\P(V_2)$, whose image preserves $\Gamma$.

\begin{lemma}\label{lem: possible birational models Sbc}
	Assume that $p \geq 3$.
Let $b,c,l$ be as in Definition~$\ref{def:Schwarz bc}$, and let
$\pi\colon \SS^{b,c}\to\P^2$
be the corresponding Schwarzenberger $\P^1$-bundle. Set
$m=\frac{b+c}{2}=b+l+1$.
Let $X$ be either a decomposable $\P^1$-bundle $\PP_d\to\P^2$ $($Definition~$\ref{def Pb})$, with
$d\geq 0$, or one of the $\P^1$-bundles $\L^r\to\P^2$ 
$($Definition~$\ref{def: P1 bundles Lsb})$, with $r\geq 0$.
Then there exists a $\PGL_2$-action on $X$ and a
$\PGL_2$-equivariant birational map
$\SS^{b,c}\dashrightarrow X$
if and only if one of the following holds:
\begin{enumerate}
\item
$X=\PP_0\simeq\P^1\times\P^2$ and $m=p^r$
for some $r\geq 0$. In this case
 the $\PGL_2$-action on $X$ is given by 
\[ F^r \times \psi\colon  \PGL_2 \hookrightarrow \PGL_2 \times \PGL_3 \simeq \Aut(\P^1\times\P^2),\]
where $F$ is the Frobenius morphism.
\item\label{item: birational Sbc Lt}
$X=\L^r \simeq \PGL_3/P_r$ and $m=2p^r$ for some $r\geq 0$.
In this case the $\PGL_2$-action on $X$ is given by 
\[\psi\colon \PGL_2 \hookrightarrow \PGL_3 \simeq \Autz(\L^r).\]
\end{enumerate}
In both cases, the equivariant birational map can be
chosen over the identity of $\P^2$.
\end{lemma}

\begin{proof}
By Proposition~\ref{prop:aut-orbits-Sbc}, the variety $\SS^{b,c}$
contains an open $\PGL_2$-orbit
$U\simeq\PGL_2/H_m$,
where
\[
H_m=
\left\langle
\left\{
\begin{bmatrix}
\alpha&0\\
0&\alpha^{-1}
\end{bmatrix}
\;\middle|\;
\alpha^{2m}=1
\right\},
\begin{bmatrix}
0&1\\
-1&0
\end{bmatrix}
\right\rangle
\]
as group schemes.
If
$\SS^{b,c}\dashrightarrow X$
is a $\PGL_2$-equivariant birational map, then $X$ contains an open
$\PGL_2$-orbit isomorphic to $\PGL_2/H_m$. Conversely, if $X$
contains such an open orbit, an identification of the two open
orbits induces a $\PGL_2$-equivariant birational map
$\SS^{b,c}\dashrightarrow X$. It is therefore enough to determine
which of the varieties $\PP_d$ and $\L^r$ admit an open
$\PGL_2$-orbit isomorphic to $\PGL_2/H_m$.

Let $X\to\P^2$ be a $\P^1$-bundle with an open
$\PGL_2$-orbit.
By Blanchard's lemma
(Proposition~\ref{prop:blanchard-lemma}),
the action induces a nontrivial homomorphism
$\PGL_2\to\PGL_3$.
By the classification of these homomorphisms recalled below,
after changing coordinates on the base we may assume
that $\P^2=\P(V_2)$ and that this homomorphism is
$\psi\circ F^j$ for some $j\geq0$.

Let $G=\PGL_2$, let $y=[0:1:0]\in\P^2\setminus\Gamma$, and let
\[
T=
\left\{
\begin{bsmallmatrix}
\alpha&0\\
0&\alpha^{-1}
\end{bsmallmatrix}
\;\middle|\;
\alpha\in \k^*
\right\}
\simeq\G_m.
\]
Then $T \subseteq  N_G(T)\subseteq  G_y$, where $G_y=\{g\in G\mid g(y)=y\}$.
We first assume that
\[
X=\PP_d=\P_{\P^2}\bigl(\cO_{\P^2}\oplus\cO_{\P^2}(-d)\bigr)
\]
with $d\geq1$ and derive a contradiction. By Lemma~\ref{lemm:surjectiveDecP2}, $\PP_d$ has an
$\Autz(\PP_d)$-invariant section $S=\{y_0=0\}$, hence $S$ is
$G$-invariant. Let
$w=\begin{bsmallmatrix}0&1\\-1&0\end{bsmallmatrix}\in N_G(T)$.
Then $wtw^{-1}=t^{-1}$ for each $t\in T$.
Since $N_G(T)\subseteq  G_y$ and $S$ is $G$-invariant, $N_G(T)$
fixes the point $z=S\cap(\PP_d)_y$, where $(\PP_d)_y$ is the fiber of $y$ for the $\P^1$-bundle $\PP_d\to \P^2$. Thus its action on
$(\PP_d)_y\simeq\P^1$ induces a homomorphism of algebraic groups
\[
\varphi\colon N_G(T)\to\Aut(\P^1)_z\simeq\G_a\rtimes\G_m.
\]
Let $\chi\colon T\to\G_m$ be the character obtained by composing
$\varphi|_T$ with the projection $\G_a\rtimes\G_m\to\G_m$.
Since $wtw^{-1}=t^{-1}$, we have
$\chi(t)=\chi(t^{-1})=\chi(t)^{-1}$, hence $\chi^2=1$.
Since the character group of $T\simeq\G_m$ is torsion-free,
$\chi$ is trivial. Therefore $\varphi(T)\subseteq \G_a$, and since
every homomorphism $\G_m\to\G_a$ is trivial, $\varphi|_T$ is
trivial. Thus $T$ acts trivially on $(\PP_d)_y$.

Consequently, every $G$-orbit intersecting $(\PP_d)_y$ has dimension at
most $\dim G-\dim T=2$. As $X$ has an open $G$-orbit, its image in
$\P^2$ intersects $\P^2\setminus\Gamma$. Since $G$ acts transitively
on $\P^2\setminus\Gamma$, this open orbit intersects $(\PP_d)_y$, a
contradiction. 

We now consider the case $X=\PP_d$ with $d=0$. 
We have therefore
\[
X=\PP_0\simeq\P^1\times\P^2.
\]
Since
$\Autz(\P^1\times\P^2)\simeq\PGL_2\times\PGL_3$,
the $G$-action is given by a pair of homomorphisms
\[
\alpha\colon G\to\PGL_2,
\qquad
\rho\colon G\to\PGL_3.
\]
Since the action is $G$-equivariantly birational to the faithful
action on $\SS^{b,c}$, the homomorphism
\[
(\alpha,\rho)\colon G\to\PGL_2\times\PGL_3
\]
is faithful.
Up to conjugation, every nontrivial homomorphism
$\PGL_2\to\PGL_2$ is of the form $F^u$, and every nontrivial
homomorphism $\PGL_2\to\PGL_3$ is of the form
$\psi\circ F^v$, for some $u,v\geq0$, where
$\psi\colon\PGL_2\hookrightarrow\PGL_3$ is induced by linear $\SL_2$-action on $V_2$ (see Definition~\ref{def:Vn}). Indeed, a nontrivial homomorphism
$\PGL_2\to\PGL_2$ is an isogeny, hence it is a power
of the Frobenius morphism followed by an automorphism of $\PGL_2$ (see \cite[Chapter~7]{Ste68});
the latter is inner. 
Similarly, a nontrivial homomorphism
$\PGL_2\to\PGL_3$ has finite kernel and, after factoring out a power
of Frobenius, yields a faithful homomorphism
$\PGL_2\hookrightarrow\PGL_3$. The corresponding $3$-dimensional
projective representation lifts to a representation of $\SL_2$ and,
since $p \geq 3$, is the irreducible representation of highest weight $2$,
namely $V_2$. This gives $\psi\circ F^v$ up to conjugation.

Moreover, neither $\alpha$ nor $\rho$ can be trivial, since otherwise
the $G$-action on $\P^1\times\P^2$ cannot have a dense orbit.
They cannot both be nontrivial with $u,v>0$, since in this case the
first Frobenius kernel $G_1$ would be contained in the kernel of
$(\alpha,\rho)$.
Consequently, up to conjugation, there are two possibilities:
\[
(\alpha,\rho)
=
(F^t,\psi)
\quad\text{or}\quad
(\Id,\psi\circ F^t)
\]
for some $t\geq0$.

Assume first that
$(\alpha,\rho)=(\Id,\psi\circ F^t)$ with $t>0$.
Choose $i\in\k$ with $i^2=-1$, and set
\[
x=[1:i],\qquad y=[0:1:0],\qquad z=(x,y).
\]
Writing $B_x\subseteq G$ for the stabilizer of $x$
under the standard action on $\P^1$, we have
\[
G_z=B_x\cap(F^t)^{-1}(N_G(T))
\]
scheme-theoretically.
Since the differential of $F^t$ is zero,
\[
\operatorname{Lie}(G_z)=\operatorname{Lie}(B_x),
\qquad
\dim\operatorname{Lie}(G_z)=2.
\]
On the other hand,
$\dim\operatorname{Lie}(H_m)\leq1$.
Thus $G_z$ cannot be conjugate to $H_m$,
which excludes this case.

Assume now that
$(\alpha,\rho)=(F^t,\psi)$ for some $t\geq0$.
For the same point $z$, a direct computation yields
$G_z=H_{p^t}$. Since $G_z$ is finite, $G\cdot z$ has dimension
$3=\dim X$ and is therefore the open $G$-orbit of $X$.
This case is possible if and only if $m=p^t$.

We now consider the case
\[
X=\L^r\to \P^2
\]
for some $r\geq0$. 
The $G$-action is faithful scheme-theoretically, since it is
	equivariantly birational to the faithful action on $\SS^{b,c}$.
	As $\Autz(\L^r)\simeq\PGL_3$, the classification above implies
	that the corresponding homomorphism
	$G\to\Autz(\L^r)$ is conjugate to $\psi$:
	a positive Frobenius power would have a nontrivial
	scheme-theoretic kernel.
There is a Schwarzenberger $\P^1$-bundle $\SS^{p^r,3p^r} \to \P^2$, with $l=p^r-1$ and $s=r$ if $r\ge 1$ (see Definition~\ref{def:Schwarz bc}). Moreover, we have $\SS^{p^r,3p^r}\simeq\L^r$, as $\P^1$-bundles, by
Proposition~\ref{prop:aut-orbits-Sbc}\ref{item:aut-Sbc}.
We may thus apply  Proposition~\ref{prop:aut-orbits-Sbc}\ref{item:open-orbit-Sbc},  and obtain that the $\PGL_2$-action on $\L^r$ induced by
$\psi\colon\PGL_2\hookrightarrow\PGL_3$ has an open orbit
isomorphic to $\PGL_2/H_{2p^r}$.
Therefore this case is possible if and only if $m=2p^r$.
Since $p\geq3$ and $p^s$ divides $m$, with the convention
$s=0$ when $l=0$, we must have $r\geq s$.
This proves~\ref{item: birational Sbc Lt} and completes the proof.
\end{proof}

\begin{proposition}\label{prop: relative max of Sbc}
	Assume that $p \geq 3$.
Let $b,c,l$ be as in Definition~$\ref{def:Schwarz bc}$, and let
$\pi\colon \SS^{b,c}\to\P^2$
be the corresponding Schwarzenberger $\P^1$-bundle. Set
$m=\frac{b+c}{2}=b+l+1$,
and assume that $c\neq 3b$. Then:
\begin{enumerate}
\item $\Autz(\SS^{b,c})$ is relatively maximal if and only if
$m\neq p^r,2p^r$ for every $r\geq 0$.
\item $\SS^{b,c}$ is stiff if and only if $p\nmid m$.
\item $\SS^{b,c}$ is not superstiff.
\end{enumerate}
\end{proposition}

\begin{proof}
If $m=p^r$ or $2p^r$ for some $r\geq 0$, then
Lemma~\ref{lem: possible birational models Sbc} shows that
$\Autz(\SS^{b,c})\simeq\PGL_2$ is not relatively maximal, and hence is neither stiff nor superstiff.

We now assume that $m\neq p^r,2p^r$ for every $r\geq 0$, and suppose that there exists an
$\Autz(\SS^{b,c})$-equivariant square birational map of $\P^1$-bundles
$\varphi\colon\SS^{b,c}\dasharrow X$
over a birational map
$\eta\colon\P^2\dasharrow Y$,
where $Y$ is a smooth projective rational surface.
Set
$G=\Autz(\SS^{b,c})\simeq\PGL_2$.
The indeterminacy locus of $\eta$ is a finite $G$-invariant subset
of $\P^2$. Since the $G$-action on $\P^2$ has exactly two orbits,
namely the smooth conic $\Gamma$ and its complement, it has no
finite orbit. Hence $\eta$ is a birational morphism. Since
$\P^2$ admits no nontrivial birational morphism onto a smooth
projective surface, $\eta$ is an isomorphism. After identifying
$Y$ with $\P^2$ via $\eta$, we may therefore assume that $\eta$ is
the identity.

Hence $X\to\P^2$ is a $\P^1$-bundle endowed with a faithful
$G$-action whose action on the base has two orbits: the smooth conic $\Gamma$ and its open complement. According to
Theorem~\ref{th:P1bundlesP2}, the following possibilities remain:
\begin{itemize}
\item $X\to\P^2$ is a decomposable $\P^1$-bundle.
\item $X\to\P^2$ is square-isomorphic to $\L^r\to\P^2$, for some
$r\geq0$.
Since $\Autz(\L^r)\simeq\PGL_3$ and the $G$-action
is faithful, the corresponding homomorphism
$G\to\PGL_3$ is conjugate to $\psi$: a positive
Frobenius power would have a nontrivial
scheme-theoretic kernel.
\item $X\to\P^2$ is square-isomorphic to a Schwarzenberger
$\P^1$-bundle $\SS^{b',c'}\to\P^2$, for some $b'\geq1$ and
$c'\geq3$ as in Definition~\ref{def:Schwarz bc}, with $c'\not=3b'$.
\end{itemize}
Indeed, the other two cases of Theorem~\ref{th:P1bundlesP2} are
impossible, since the $G$-action on $\P^2$ preserves
neither a line nor a point.

The first two cases are excluded by
Lemma~\ref{lem: possible birational models Sbc}.
Thus $X\to\P^2$ is square-isomorphic to a Schwarzenberger
$\P^1$-bundle $\SS^{b',c'}\to\P^2$, with $c'\not=3b'$.

Set
$m'=b'+l'+1=\frac{b'+c'}{2}$.
By Proposition~\ref{prop:aut-orbits-Sbc}\ref{item:open-orbit-Sbc},
the open $G$-orbits of $\SS^{b,c}$ and $\SS^{b',c'}$ are
respectively isomorphic to $G/H_m$ and $G/H_{m'}$. Since $\varphi$
is $G$-equivariant, these open orbits are isomorphic, and therefore
$m=m'$, since the finite group scheme $H_m$ has length $2m$,
or equivalently
$b+c=b'+c'$.
Moreover,
$\varphi G\varphi^{-1}\subseteq \Autz(X)\simeq\PGL_2$.
This inclusion is an equality, since both groups are
connected of dimension $3$.
Hence $\Autz(\SS^{b,c})$ is relatively maximal.

It remains to determine the stiffness. Assume first that $p$ does not divide $m$, which implies $l=0$. 
As $p$ does not divide $m'=m$, we have $l'=0$. Hence,
$b'=m-1=b$, $c'=m+1=c$.
Thus $(b',c')=(b,c)$, and consequently $\SS^{b,c}\to\P^2$ is stiff.
It is, however, not superstiff. Indeed, by
Corollary~\ref{cor: description of the orbits for Sbc},
$\SS^{b,c}$ contains a $\PGL_2$-invariant curve $\ell$
such that $\pi|_\ell\colon\ell\to\Gamma$ is an isomorphism.
Blowing up $\ell$ and contracting afterwards the strict transform of the inverse image of
the smooth conic $\Gamma\subset\P^2$ yields a nontrivial
$\PGL_2$-equivariant elementary link to another $\P^1$-bundle over
$\P^2$. By the stiffness proved above, this $\P^1$-bundle is
square-isomorphic to $\SS^{b,c}\to\P^2$. Hence $\SS^{b,c}$ is not
superstiff.

Assume now that $p$ divides $m$. We exhibit two distinct Schwarzenberger
$\P^1$-bundles having the same value of $m$. The choice
$(l_0,b_0,c_0)=(0,m-1,m+1)$ gives one such $\P^1$-bundle. 
Since the cases $m=p^r$ and $m=2p^r$ have already been
excluded, write $m=up^t$ with $u\geq3$, $p\nmid u$,
and $t\geq1$, and set
\[
(l_1,b_1,c_1)=\bigl(p^t-1,(u-1)p^t,(u+1)p^t\bigr).
\]
Since $p$ is odd, $l_1$ is even, and
$p^{t-1}\leq l_1<p^t$ and $p^t\mid b_1+l_1+1=m$.
Thus $(b_1,c_1,l_1)$ also satisfies
Definition~\ref{def:Schwarz bc}. Moreover, $c_1=3b_1$ if and only
if $u=2$, which is excluded by the assumption $u\geq3$. 
The $\P^1$-bundles $\SS^{b_0,c_0}$ and $\SS^{b_1,c_1}$ are not
square-isomorphic: any square isomorphism would preserve
their invariant conics and hence identify their restrictions
to tangent lines, whereas these restrictions are
$\F_{b_0}$ and $\F_{b_1}$, with
$b_0=m-1\neq m-p^t=b_1$.

All three $\P^1$-bundles $\SS^{b,c}$, $\SS^{b_0,c_0}$, and
$\SS^{b_1,c_1}$ have open orbit $G/H_m$, with projection
to the open orbit of the base given by
\[
G/H_m\longrightarrow G/N_G(T).
\]
These identifications yield $G$-equivariant birational
maps over the identity of $\P^2$.
Since $\SS^{b,c}$ cannot be square-isomorphic to both
constructed $\P^1$-bundles, it is not stiff.
\end{proof}

\begin{remark} 
In characteristic zero, relative maximality implies stiffness for
Schwar\-zenberger $\P^1$-bundles by \cite[Corollary~5.6.3]{BFT}.
In positive characteristic, this is no longer true: by
Proposition~\ref{prop: relative max of Sbc}, there exist relatively
maximal Schwarzenberger $\P^1$-bundles which are not stiff.
\end{remark}

\subsection{The case of Umemura \texorpdfstring{$\P^1$}{P1}-bundles over \texorpdfstring{$\F_a$}{Fa}} \label{sec: max for Ume bundles}
We first consider Umemura $\P^1$-bundles over $\F_a$ $($Definition~$\ref{def Umemura bundle})$ for
arbitrary $a\geq1$ and study the equivariant elementary
links starting from them. Such links are obtained by
blowing up an invariant curve and contracting the strict
transform of the surface lying above its image in $\F_a$ (see Lemma~\ref{Lemm:DecElLinks}). 
Our first step is therefore to determine the invariant
curves that can serve as centers. (Since the induced action
on the base $\P^1$ of the ruling is transitive, an Umemura
bundle has no fixed points.)

\begin{definition}
	For each $\P^1$-bundle $Z_a^{b,c,P}\to\F_a$ and
	$i,j\in\{0,1\}$, let $l_{ij}\subseteq Z_a^{b,c,P}$
	be the closure of the curve defined by $x_i=y_j=0$
	in the first chart $\F_b\times\A^1$.
	This gives four curves $l_{00},l_{01},l_{10},l_{11}$.
\end{definition}

\begin{lemma}\label{Lem:InvariantUme}
	Let $\pi\colon X=\U_{a,\omega}^{b,c}=Z_a^{b,c,P}\to\F_a$
	be an Umemura $\P^1$-bundle, with
	$	P(y_0,y_1,z)=\sum_{i=0}^b y_0^iy_1^{b-i}P_i(z)z^{ai+1}$. 
	We describe the curves in $X$ invariant under $\Autz(X)$.
\begin{enumerate}
\item\label{Invl0Um}
The curve $l_{00}$, given by $x_0=y_0=0$ on both charts, is invariant.

\item\label{Invl1Um}
The curve $l_{10}$ is invariant if and only if $P_0=0$,
equivalently if and only if $y_0$ divides $P$.

\item\label{OnlyInvUm}
If $P_0=0$, the curves $l_{00}$ and $l_{10}$ are the only invariant
curves contained in $\pi^{-1}(s_{-a})$. If $P_0\not=0$, then $l_{00}$ is the unique invariant curve contained in $\pi^{-1}(s_{-a})$.
\item\label{OnlyInvTwo}
If the image of $\Autz(X)$ in $\Aut(\F_a)$ has exactly two orbits, every invariant curve belongs to $\{l_{00},l_{10}\}$.
\end{enumerate}
\end{lemma}

\begin{proof}
Since $a\geq1$, the negative section $s_{-a}\subset\F_a$ is invariant
by $\Autz(\F_a)$. Hence the surface
$\pi^{-1}(s_{-a})$ given by $\{y_0=0\}$ on both charts is invariant under
$\Autz(X)$. Moreover, the fibers of the
$\F_b$-bundle $X \to \P^1$ are permuted by
$\Autz(X)$. Since $b\geq1$, the surface $S_{-b}$
given by $\{x_0=0\}$ on both charts, which is the union of the negative sections of these
fibers, is invariant. Therefore
$l_{00}=\pi^{-1}(s_{-a})\cap S_{-b}$
is invariant, proving~\ref{Invl0Um}.

Recall that $X=\U_{a,\omega}^{b,c}=Z_a^{b,c,P}$ has transition function \[
\nu_{c,P}\colon([x_0:x_1;y_0:y_1],z) \mapsto \left([x_0:x_1z^{c}+x_0P(y_0,y_1,z);y_0z^a:y_1],\frac{1}{z}\right).\]

\ref{Invl1Um}: If $P_0=0$, then the restriction to $y_0=0$ yields $([x_0:x_1],z)\mapsto ([x_0:x_1z^c],\frac{1}{z})$, so 
$\pi^{-1}(s_{-a})$ is isomorphic to $\F_c$, and $l_{10}$
corresponds to its negative section $s_{-c}$. Since $c>0$, this is the
unique curve of negative self-intersection on $\F_c$, and is therefore
invariant. Hence $l_{10}$ is invariant.

If $P_0\not=0$, then $P(0,y_1,z)=y_1^b\lambda_0z^{c-1-l_0}$ with $\lambda_0\not=0$ (see Definition~\ref{def Umemura bundle})
and Proposition~\ref{ActionSL2GL2Umemura} implies that 
$\begin{bsmallmatrix}
1&1\\
0&1
\end{bsmallmatrix}\in\SL_2$
sends
$([1:0;0:1],0)\in l_{10}$
to
$([1:\lambda_0;0:1],0)\notin l_{10}$.
Thus $l_{10}$ is not invariant in this case, which proves
\ref{Invl1Um}.

\ref{OnlyInvUm}: 
Let $\ell\subseteq\pi^{-1}(s_{-a})$ be an irreducible
$\Autz(X)$-invariant curve. In particular, $\ell$ is
$\SL_2$-invariant. Since $\SL_2$ acts transitively on $s_{-a}$,
we have $\pi(\ell)=s_{-a}$.
Choose a point
$q=([u_0:u_1;0:1],0)\in\ell$.
The fiber of $\ell\to s_{-a}$ over $z=0$ is finite.
The diagonal torus of $\SL_2$ preserves this finite set and,
being connected, fixes each of its points. By
Proposition~\ref{ActionSL2GL2Umemura}, the element
$\begin{bsmallmatrix}
	\alpha&0\\
	0&\alpha^{-1}
\end{bsmallmatrix}$
sends $q$ to $([u_0:u_1\alpha^c;0:1],0)$.
Since $c>0$, this implies that $u_0=0$ or $u_1=0$.
For every $\gamma\in\k$, the element
$\begin{bsmallmatrix}
	1&0\\
	\gamma&1
\end{bsmallmatrix}$
sends $q$ to
$([u_0:u_1;0:1],\gamma)$. 
Consequently, $\ell$ contains a dense open subset of $l_{00}$
or $l_{10}$, respectively. Since $\ell$ is irreducible,
we obtain $\ell=l_{00}$ or $\ell=l_{10}$.
Together with~\ref{Invl0Um} and~\ref{Invl1Um},
this proves~\ref{OnlyInvUm}.

\ref{OnlyInvTwo}: 
	If the image of $\Autz(X)$ in $\Aut(\F_a)$ has exactly two
	orbits, then these are $s_{-a}$ and its open complement.
	If $\ell$ is an $\Autz(X)$-invariant curve, then $\pi(\ell)=s_{-a}$.
\end{proof}

The following proposition describes the elementary links between Umemura $\P^1$-bundles.

\begin{proposition}\label{prop: description of the equ links Umemura}
	Let $\pi\colon X=\U_{a,\omega}^{b,c}=Z_a^{b,c,P}\to\F_a$
	be an Umemura $\P^1$-bundle, with
	$	P(y_0,y_1,z)
	=
	\sum_{i=0}^b
	y_0^iy_1^{b-i}P_i(z)z^{ai+1}$.

	\begin{enumerate}
		\item\label{Blowl00egal}
		The $\P^1$-bundle $X'=Z_a^{b+1,c+a,y_0z^aP}\to \F_a$ is again an Umemura $\P^1$-bundle. Blowing up the invariant curve $l_{00}\subseteq X$ and contracting the strict transform of
		$\pi^{-1}(s_{-a})\subseteq X$ onto $l_{10}\subseteq X'$ yields a birational map
		$\varphi\colon X\dashrightarrow X'$ satisfying
		$\varphi \Autz(X) \varphi^{-1}=\Autz(X')$.
		
		\item\label{Blowl10egal}
		If $b\ge 2$ and $P_0=0$, then $X'=Z_a^{b-1,c-a,P/(y_0z^a)}\to \F_a$ is again an Umemura $\P^1$-bundle. Blowing up the invariant curve $l_{10}\subseteq X$ and contracting the strict transform of
		$\pi^{-1}(s_{-a})$ onto $l_{00}\subseteq X'$ yields a birational map
		$\varphi\colon X\dashrightarrow X'$ satisfying
		$\varphi \Autz(X) \varphi^{-1}=\Autz(X')$.
		
		\item\label{Blowl10smaller}
		If $b=1$ and $P_0=0$, blowing up the invariant curve $l_{10}\subseteq X$ and contracting the strict transform of
		$\pi^{-1}(s_{-a})$ onto a curve yields a birational map
		$\varphi\colon X\dashrightarrow X'$ satisfying
		$\varphi \Autz(X) \varphi^{-1}\subsetneq\Autz(X')$, where $X'\to \F_a$ is a decomposable $\P^1$-bundle.
	\end{enumerate}
\end{proposition}

\begin{proof}
	For each $\beta\in \N$, we denote by $\varphi_\beta$ the birational map
	\[
	\begin{aligned}
		\varphi_\beta\colon \F_\beta\times \A^1
		&\dashrightarrow \F_{\beta+1}\times \A^1,\\
		([x_0:x_1;y_0:y_1],z)
		&\mapsto ([x_0:x_1y_0;y_0:y_1],z),
	\end{aligned}
	\]
	which blows up $\{x_0=y_0=0\}$ and contracts $\{y_0=0\}$ onto $\{x_1=y_0=0\}$.
	
	For $\beta\in \N$, $\gamma\in \Z$ and
	\[
	Q=\sum_{i=0}^\beta y_0^iy_1^{\beta-i}Q_i(z)z^{ai+1},
	\]
	where $Q_i\in \k[z,1/z]$ for each $i\in \{0,\ldots,\beta\}$, we denote by
	$\nu_{\beta,\gamma,Q}\in \Aut(\F_{\beta}\times (\A^1\setminus\{0\}))$
	the automorphism
	\[
	\nu_{\beta,\gamma,Q}\colon
	([x_0:x_1;y_0:y_1],z)
	\mapsto
	\left(
	[x_0:x_1z^{\gamma}+x_0Q(y_0,y_1,z);y_0z^a:y_1],
	\frac{1}{z}
	\right).
	\]
	A direct calculation gives
	\[
	\nu_{\beta+1,\gamma+a,y_0z^aQ}\circ \varphi_\beta
	=
	\varphi_\beta\circ \nu_{\beta,\gamma,Q}.
	\]
	
	\ref{Blowl00egal}
	Choosing $\beta=b$, $\gamma=c$ and $Q=P$, gluing $\varphi_b$ on both charts gives the birational map
	$\varphi\colon X\dashrightarrow X'$ of \ref{Blowl00egal}, which is the blow-up of $l_{00}\subseteq X$ followed by the contraction of the strict transform of
	$\pi^{-1}(s_{-a})\subseteq X$ onto $l_{10}\subseteq X'$.
	As $c+a\ge c\ge 2$, $b+1\ge b\ge 1$ and $P'=y_0z^aP$ satisfy the conditions defining an Umemura $\P^1$-bundle, we find that $X'\to \F_a$ is an Umemura $\P^1$-bundle.
	Lemma~\ref{Lem:InvariantUme}\ref{Invl0Um} shows that $l_{00}\subset X$ is $\Autz(X)$-invariant, so
	$\varphi \Autz(X) \varphi^{-1}\subseteq \Autz(X')$.
	Lemma~\ref{Lem:InvariantUme}\ref{Invl1Um} shows that $l_{10}\subseteq X'$ is $\Autz(X')$-invariant, so
	$\varphi^{-1} \Autz(X') \varphi\subseteq \Autz(X)$.
	This gives
	$\varphi \Autz(X) \varphi^{-1}=\Autz(X')$.
	
	\ref{Blowl10egal}--\ref{Blowl10smaller}
	We now assume that $P_0=0$ and apply the construction above with
	$\beta=b-1$, $\gamma=c-a$ and $Q=P/(y_0z^a)$, so that
	$b=\beta+1$, $c=\gamma+a$ and $y_0z^aQ=P$.
	Hence, gluing $\varphi_{b-1}^{-1}$ on both charts gives a birational map
	$\varphi\colon X\dashrightarrow X'$, where
	$X'=Z_a^{b',c',P'}$, with
	$(b',c',P')=(b-1,c-a,P/(y_0z^a))$.
	Moreover, Lemma~\ref{Lem:InvariantUme}\ref{Invl1Um} shows that $l_{10}\subseteq X$ is $\Autz(X)$-invariant, so
	$\varphi \Autz(X) \varphi^{-1}\subseteq \Autz(X')$.
	In particular, $\SL_2$ acts on $X'$.
	
	If $b\ge 2$, then $b'=b-1\ge 1$ and $c'=c-a\ge 2$, since there exists $i\ge 1$ such that $P_i\neq 0$ (as $P_0=0$), and thus
	$0\le c-2-ai\le c-2-a$.
	Moreover, for each $i\in\{1,\ldots,b\}$, we have
	$P_i=P'_{i-1}\in \k[z]_{\le c'-2-a(i-1)}$, with
	$c'-2-a(i-1)=c-2-ai$.
	We thus obtain an Umemura $\P^1$-bundle $X'\to \F_a$.
	Applying \ref{Blowl00egal} to $\varphi^{-1}$ gives
	$\varphi \Autz(X) \varphi^{-1}=\Autz(X')$.
	This proves \ref{Blowl10egal}.
	
	If $b=1$, then $b'=b-1=0$, so $X'$ is not an Umemura $\P^1$-bundle.
	The morphism $X'\to \P^1$ is an $\F_0$-bundle, and thus
	$X'\to \F_a$ is isomorphic to the decomposable $\P^1$-bundle
	$\FF_a^{0,c}$ by \cite[Proposition~3.3.1]{BFT}.
	We then have
	$\varphi \Autz(X) \varphi^{-1}\subsetneq\Autz(X')$,
	since $\Autz(\FF_a^{0,c})$ contains a torus of dimension $3$
	(Lemma~\ref{lem:aut-group-Fabc}).
	This proves \ref{Blowl10smaller}.
\end{proof}

Let
$H=\Im\bigl(\Autz(\U_{a,\omega}^{b,c})\to\Autz(\F_a)\bigr)$.
Since $H$ contains the standard subgroup
$S=\Im(\SL_2\to\GL_2/\mu_a)$, there are two possibilities
in every positive characteristic:
\begin{enumerate}
	\item
	up to conjugation, $H$ is equal to $S$ or $\GL_2/\mu_a$,
	and acts on $\F_a$ with three orbits: the sections
	$s_a$, $s_{-a}$, and their open complement; or
	\item
	$H$ acts on $\F_a$ with two orbits: the section $s_{-a}$
	and its open complement.
\end{enumerate}
For $a\geq2$, this follows from
Proposition~\ref{prop:trichotomy for algebraic subgroups of AUt(Fa)}.
For $a=1$, the proof of Lemma~\ref{lem:UmeF1twoorbits}
shows that either $U_1\subseteq H$ or $H\subseteq\GL_2$,
where $U_1=\k[z_0,z_1]_1$.
In the latter case, $H=\SL_2$ or $H=\GL_2$.

We first treat the three-orbit case, which admits an equivariant
birational reduction to the two-orbit case.

\begin{proposition}\label{prop:case H=GL2/mu_a}
	Let $\pi\colon X=\U_{a,\omega}^{b,c}\to\F_a$ be an Umemura
	$\P^1$-bundle, and set
	\[
	H:=\Im\bigl(\Autz(X)\to\Aut(\F_a)\bigr), \quad
	G:=\mathrm{pr}_{\GL_2/\mu_a}(H).
	\]
	Assume that $H$ preserves the section $s_a$.
	Writing $X=Z_a^{b,c,P}$, there exists an infinite sequence of birational maps
	\[
	X=X_0\stackrel{\varphi_1}{\dashrightarrow} X_1
	\stackrel{\varphi_2}{\dashrightarrow} X_2
	\stackrel{\varphi_3}{\dashrightarrow} \cdots
	\]
	where, for each $i\ge 1$, $X_i=Z_a^{b+i,c,y_1^iP}$,
	$\varphi_i$ is the blow-up of $l_{01}\subseteq X_{i-1}$
	followed by the contraction of the strict transform of
	the preimage of $s_a$ onto $l_{11}\subseteq X_i$, and
	$\psi_i=\varphi_i\cdots\varphi_1\colon X\dashrightarrow X_i$
	satisfies
	$\psi_i\Autz(X)\psi_i^{-1}\subseteq\Autz(X_i)$.
	
	Moreover, there exists an integer $r\ge 1$ such that
	$\psi_r\Autz(X)\psi_r^{-1}\subsetneq\Autz(X_r)$
	and $\Autz(X_r)$ acts on $\F_a$ with exactly two orbits.
	If $p=2$, $G=\GL_2/\mu_a$, or $a$ is not a power of $p$,
	we may moreover assume that $\Autz(X_r)$ acts on $\F_a$
	as $\k[z_0,z_1]_a\rtimes G$.
\end{proposition}

\begin{proof}
	Write
	\[
	P=\sum_{i\in I}
	\lambda_i y_0^iy_1^{b-i}z^{ai+p^{s_i}},
	\qquad
	ai+2p^{s_i}=c,
	\qquad
	\lambda_i\in\k^*,
	\]
	where $I=\{i\mid P_i\neq0\}$, and set $i_0=\min I$.
	Choose $N$ such that $p^N>\max(b,c)$, and put
	\[
	B:=i_0+p^N,
	\qquad
	r:=B-b\geq1.
	\]
	For $0\leq j\leq r$, set
	\[
	Q_j:=y_1^jP,
	\qquad
	\pi_j\colon X_j:=Z_a^{b+j,c,Q_j}\to\F_a.
	\]
	The nonzero coefficients $P_i$ are unchanged (we only add
	some zero polynomials), so each $X_j$ is again an Umemura
	$\P^1$-bundle.
	For $0\leq j<r$, the birational map
	\[
	\begin{aligned}
		\F_{b+j}\times\A^1
		&\dashrightarrow \F_{b+j+1}\times\A^1,\\
		([x_0:x_1;y_0:y_1],z)
		&\mapsto ([x_0:x_1y_1;y_0:y_1],z)
	\end{aligned}
	\]
	on each chart is the blow-up of $l_{01}$, followed by the
	contraction of the strict transform of $\{y_1=0\}$,
	corresponding to the preimage of $s_a$.
	Since this map is compatible with the transition functions,
	it gives an elementary link
	$\varphi_{j+1}\colon X_j\dashrightarrow X_{j+1}$ over $\F_a$.
	
	These links are equivariant for the original group $\Autz(X)$,
	since $x_0=0$ is the union of the negative sections and
	$\{y_1=0\}=s_a$ is invariant under $H$.
	
	We now write
	\[
	H_r:=\Im\bigl(\Autz(X_r)\to\Aut(\F_a)\bigr).
	\]
	
	If $G=\GL_2/\mu_a$, then $I=\{i_0\}$ and
	\[
	P=\lambda_{i_0}y_0^{i_0}y_1^{b-i_0}
	z^{ai_0+p^{s_{i_0}}},
	\]
	so
	\[
	Q_r=\lambda_{i_0}y_0^{i_0}y_1^{p^N}
	z^{ai_0+p^{s_{i_0}}}.
	\]
	As $R\in\k[z_0,z_1]_a$ acts by replacing $y_1$ with
	$y_1+y_0R$, it sends $Q_r$ to
	\[
	Q_r+y_0^B R^{p^N}z^{ai_0+p^{s_{i_0}}},
	\]
	which is equivalent to $Q_r$, since $c-2-ap^N<0$.
	This proves that $H_r=\Aut(\F_a)$ in this case.
	
	We now assume that
	$G=S=\Im\bigl(\SL_2\to\GL_2/\mu_a\bigr)$.
	
	If $p=2$, or if $p$ is odd and $a$ is not a power of $p$,
	Proposition~\ref{propMonom2} implies that
	$\k[z_0,z_1]_a\subseteq H_r$.
	Indeed, in this case
	$u=i_0+p^{\nu_p(B-i_0)}=B$,
	$q=p^N$ and $n_u=n_{i_0}-aq<0$, so
	$K=\k[z_0,z_1]_a$.
	This proves that
	$H_r=\k[z_0,z_1]_a\rtimes G$ in this case.
	
	We may thus assume that $p$ is odd and that $a=p^t$
	for some $t\geq0$.
	We show that $H_r$ contains the translations
	$R=\mu z_0^a$, for all $\mu\in\k$.
	For $i\in I$, set
	\[
	q_i:=p^{\nu_p(B-i)}.
	\]
	We first observe that $aq_i\geq p^{s_i}$.
	For $i=i_0$, this follows from
	$q_{i_0}=p^N>c>2p^{s_{i_0}}$.
	For $i>i_0$, we have $s_i<s_{i_0}$ and
	\[
	a(i-i_0)=2(p^{s_{i_0}}-p^{s_i}).
	\]
	We obtain
	$\nu_p(i-i_0)=s_i-t$.
	Since $p^N>i-i_0$, this gives
	\[
	q_i=p^{\nu_p(B-i)}
	=p^{\nu_p(p^N-(i-i_0))}
	=p^{s_i-t},
	\]
	and thus $aq_i=p^{s_i}$.
	
	By Lemma~\ref{Lemm:ActionAutFaParSpace}\ref{ActionFap},
	the translation $R=\mu z_0^a$ acts by substituting
	$y_1-\mu y_0$ for $y_1$ and reducing to normal form.
	Writing $B-i=q_im_i$, with $p\nmid m_i$, we have
	\[
	(y_1-\mu y_0)^{B-i}
	=
	(y_1^{q_i}-\mu^{q_i}y_0^{q_i})^{m_i}.
	\]
	Thus every additional term has $y_0$-degree at least $q_i$.
	Hence each additional term arising from the $i$-th component
	of $Q_r$ has the form
	\[
	\alpha\,y_0^{i+m}y_1^{B-i-m}z^{ai+p^{s_i}},
	\qquad \alpha\in\k,\quad m\geq q_i.
	\]
	Its $z$-degree is strictly smaller than $a(i+m)+1$,
	since $am\geq p^{s_i}$.
	
	Hence, when writing the resulting polynomial in the normal form
	of Proposition~\ref{prop: paramter space for non-dceomposable P1 bundles},
	these additional terms do not contribute to the coefficient
	$P_{i+m}(z)$ of
	$y_0^{i+m}y_1^{B-(i+m)}z^{a(i+m)+1}$.
	
	Thus the connected subgroup
	$\{\mu z_0^a\mid \mu\in\k\}$
	fixes the parameter class of $X_r$ and is contained in $H_r$.
	This implies that $H_r$ acts on $\F_a$ with exactly two orbits,
	so
	$\psi_r\Autz(X)\psi_r^{-1}\subsetneq\Autz(X_r)$.
\end{proof}
	
According to Proposition~\ref{prop:case H=GL2/mu_a}, to study the relative maximality of Umemura $\P^1$-bundles,
it remains to consider those whose neutral automorphism groups
act on $\F_a$ with exactly two orbits.

\begin{proposition}\label{prop:birP1bUme}
	Let $\pi\colon X=\U_{a,\omega}^{b,c}=Z_a^{b,c,P}\to\F_a$
	be an Umemura $\P^1$-bundle, with
	$	P(y_0,y_1,z)=\sum_{i=0}^b y_0^iy_1^{b-i}P_i(z)z^{ai+1}$.
	Set $i_0=\min\{i\mid P_i\neq0\}$, and assume that the action of
	$\Autz(X)$ on $\F_a$ has exactly two orbits. Then the following hold:
	\begin{enumerate}
		\item\label{SequenceLinks}
		There is an infinite sequence of elementary links
		\[
		Y_s\dashrightarrow Y_{s+1}\dashrightarrow Y_{s+2}\dashrightarrow\cdots,
		\]
		where $s=\max(-i_0,-b+1)\leq0$ and
		$Y_n=Z_a^{b+n,c+an,(y_0z^a)^nP}$ for each $n\geq s$ and $Y_0=X$. Each link conjugates the neutral components of
		the automorphism groups.
		
		\item\label{i0bNotm}
		If $i_0=b$, there is a birational map of $\P^1$-bundles
		$\varphi\colon X\dashrightarrow X'$ satisfying
		\[
		\varphi\Autz(X)\varphi^{-1}\subsetneq\Autz(X'),
		\]
		where $X'\to\F_a$ is a decomposable $\P^1$-bundle.
		In particular, $\Autz(X)$ is not relatively maximal.
		
		\item\label{i0PbOnly}
		If $i_0<b$, then every $\Autz(X)$-equivariant birational map of $\p^1$-bundles is, up to isomorphisms of $\p^1$-bundles,
		a composition of the elementary links in~\ref{SequenceLinks}
		and their inverses.
	\end{enumerate}
\end{proposition}

\begin{proof}
	Assertion~\ref{SequenceLinks} follows from
	Proposition~\ref{prop: description of the equ links Umemura}.
	The links from left to right always exist, are given by
	\ref{prop: description of the equ links Umemura}\ref{Blowl00egal},
	and start by blowing up $l_{00}$.
	The links from right to left are given by
	\ref{prop: description of the equ links Umemura}\ref{Blowl10egal}
	and start by blowing up $l_{10}$.
	The choice of $s$ ensures that $Y_n\dashrightarrow Y_{n-1}$
	exists for each $n>s$, since $b+n\geq2$ and $n+i_0\geq1$,
	so $y_0$ divides $(y_0z^a)^nP$.
	The assertion about the automorphism groups follows from the same
	proposition.
	
	\ref{i0bNotm}:
	Assume that $i_0=b$. Then
	$Y_s=Y_{-b+1}=Z_a^{1,c+as,(y_0z^a)^{-b+1}P}$. 
	Proposition~\ref{prop: description of the equ links Umemura}
	\ref{Blowl10smaller} gives a birational map of $\P^1$-bundles
	$\psi\colon Y_s\dashrightarrow X'$ satisfying
	\[
	\psi\Autz(Y_s)\psi^{-1}\subsetneq\Autz(X'),
	\]
	where $X'\to\F_a$ is a decomposable $\P^1$-bundle.
	Composing $\psi$ with the links from $X$ to $Y_s$ gives
	\ref{i0bNotm}.
	
	\ref{i0PbOnly}:
	Assume that $i_0<b$, so $s=-i_0$.
	By Lemma~\ref{Lemm:DecElLinks}, up to isomorphisms of
	$\P^1$-bundles over $\F_a$, every equivariant birational map
	of $\P^1$-bundles over $\F_a$ factors into elementary links
	above $s_{-a}$.
	Since the links in~\ref{SequenceLinks} conjugate the neutral
	components of the automorphism groups, the action of
	$\Autz(Y_n)$ on $\F_a$ has exactly two orbits for every $n\geq s$.
	By Lemma~\ref{Lem:InvariantUme} and
	Proposition~\ref{prop: description of the equ links Umemura},
	the only such link from $Y_s=Y_{-i_0}$ is
	$Y_s\dashrightarrow Y_{s+1}$, whereas from $Y_n$, for $n>s$,
	they are $Y_n\dashrightarrow Y_{n+1}$ and
	$Y_n\dashrightarrow Y_{n-1}$.
	This proves~\ref{i0PbOnly}.
\end{proof}

We now consider Umemura $\P^1$-bundles over $\F_1$. The following lemma describes when the action on $\F_1$ has two orbits.

\begin{lemma}\label{lem:UmeF1twoorbits}
	Let $\pi\colon X=\U_{1,\omega}^{b,c}=Z_1^{b,c,P}\to\F_1$
	be an Umemura $\P^1$-bundle. We can write
	\[
	P=\sum_{i\in I}
	\lambda_i y_0^iy_1^{b-i}z^{i+p^{s_i}},
	\qquad
	i+2p^{s_i}=c,
	\qquad
	\lambda_i\in\k^*,
	\qquad s_i\geq0,
	\]
	where $I\subseteq\{0,\ldots,b\}$ is nonempty, and set
	$i_0=\min I$, $r=p^{s_{i_0}}$, and $m=b-i_0$. 
	Let $H=\Im\bigl(\Autz(X)\to\Aut(\F_1)\bigr)$ and
	$U=\k[z_0,z_1]_1$. The following conditions are equivalent:
	\begin{enumerate}
		\item\label{UmeF1twoorbits1}
		The group $\Autz(X)$ $($equivalently, $H)$ acts on $\F_1$ with exactly two orbits;
		\item\label{UmeF1twoorbits2}
		The group $H$ is not contained in $\GL_2$;
		\item\label{UmeF1twoorbits3}
		$U\subseteq H$;
		\item\label{UmeF1twoorbits4}
		$(U\cap H)(\k)\neq\{0\}$;
		\item\label{UmeF1twoorbits5}
		$r\mid m$.
	\end{enumerate}
\end{lemma}

\begin{proof}
	By Definition~\ref{def Umemura bundle}, we have
	\[
	P=\sum_{i\in I_{1,b,c}}
	\lambda_i y_0^{i}y_1^{b-i}z^{c-1-l_i},
	\qquad
	c-2-i=2l_i=2(p^{s_i}-1).
	\]
	Thus $c-i=2p^{s_i}$ and $c-1-l_i=i+p^{s_i}$.
	Retaining the nonzero coefficients yields the asserted form.
	
	Since
	$\SL_2\subseteq H\subseteq\Aut(\F_1)=U\rtimes\GL_2$
	and $\SL_2$ acts transitively on $U(\k)\setminus\{0\}$,
	either $(U\cap H)(\k)=\{0\}$ or $U\subseteq H$.
	In the first case, conjugating the standard $\SL_2$
	by elements of $H(\k)$ and projecting to $\GL_2$
	shows that $\SL_2(\k)$ is normal in $H(\k)$.
	Since $U^{\SL_2}=\{0\}$, every element of $H(\k)$
	has zero translation component.
	As $H$ is reduced, this gives $H\subseteq\GL_2$.
	The standard $\SL_2$ and $\GL_2$ each have exactly three orbits on $\F_1$,
	whereas $U\rtimes\SL_2$ has exactly two. This gives \ref{UmeF1twoorbits1}$\Leftrightarrow$\ref{UmeF1twoorbits2}$\Leftrightarrow$\ref{UmeF1twoorbits3}$\Leftrightarrow$ \ref{UmeF1twoorbits4} $\Leftrightarrow-z_0\in H$. It remains to show that $-z_0\in H\Leftrightarrow r\mid m$.
	
	By Lemma~\ref{Lemm:ActionAutFaParSpace}\ref{ActionFap},
	the action of $-z_0$ on $P$ is obtained by replacing $y_1$ with $y_1+y_0$
	and reducing to the normal form~\eqref{eq:def of P}.
	
	Assume first that $r\nmid m$, and write $q=p^{\nu_p(m)}<r$.
	Since $\binom{m}{q}\neq0$ in $\k$, the expansion of
	$\lambda_{i_0}y_0^{i_0}(y_1+y_0)^mz^{i_0+r}$
	produces the nonzero term
	\[
	\binom{m}{q}\lambda_{i_0}
	y_0^{i_0+q}y_1^{m-q}z^{i_0+r}.
	\]
	This term survives in normal form because
	\[
	i_0+q+1\leq i_0+r\leq c-1=i_0+2r-1.
	\]
	Moreover, for each $i\in I\setminus\{i_0\}$, we have
	\[
	i-i_0=2(r-p^{s_i})>0,
	\]
	so $s_i<s_{i_0}$ and hence $i-i_0\geq r$.
	Thus the term above cannot be canceled by a contribution from another monomial.
	It is absent from $P$, while the coefficient of
	$y_0^{i_0}y_1^mz^{i_0+r}$ remains unchanged.
	Therefore the transformed polynomial is not a scalar multiple of $P$,
	so $-z_0\notin H$.
	
	Assume now that $r\mid m$.
	For each $i\in I$, set $r_i=p^{s_i}$.
	Then $r_i\mid r$ and
	\[
	b-i=m-(i-i_0)=m-2(r-r_i),
	\]
	so $r_i\mid b-i$.
	Since $r_i$ is a power of $p$, every additional term arising from
	\[
	\lambda_i y_0^i(y_1+y_0)^{b-i}z^{i+r_i}
	\]
	has $y_0$-degree at least $i+r_i$ and $z$-degree $i+r_i$.
	All these terms vanish in normal form, since their $z$-degree
	is at most their $y_0$-degree.
	Thus $-z_0$ fixes $P$ so $-z_0\in H$.
\end{proof}

For each Umemura $\p^1$-bundle over $\F_1$, the descent lemma yields a
$\p^1$-bundle $\V\to\p^2$. Applying it to any bundle in the sequence of
Proposition~\ref{prop:birP1bUme}\ref{SequenceLinks} yields the same
$\p^1$-bundle $\V\to\p^2$. The next proposition characterizes when the corresponding
square birational map is a morphism.

\begin{proposition}\label{prop:UmeF1Contraction}
Let $\hat\pi\colon X=\U_{1,\omega}^{b,c}=Z_1^{b,c,P}\to\F_1$
be an Umemura $\P^1$-bundle, with
$P(y_0,y_1,z)
=
\sum_{i=0}^b
y_0^iy_1^{b-i}P_i(z)z^{i+1}$, and
assume that the action of $\Autz(X)$ on $\F_1$ has two orbits.

Let $\eta\colon\F_1\to\P^2$,
$[y_0:y_1;z_0:z_1]\mapsto[y_0z_0:y_1:y_0z_1]$,
be the blow-up of $q=[0:1:0]$, with exceptional section
$s_{-1}=\{y_0=0\}$. Then the following conditions are equivalent:
\begin{enumerate}
\item\label{item:UmeF1Contraction}
There exist a $\P^1$-bundle $\pi\colon \V=\V_{\omega}^{b,c} \to\P^2$ and a birational
morphism $\psi\colon X\to \V$ such that
\[
\xymatrix@R=3mm@C=2cm{
X \ar[r]^{\psi}\ar[d]_{\hat\pi}
& \V\ar[d]^{\pi}\\
\F_1\ar[r]_{\eta}
&\P^2
}
\]
commutes and $\psi$ is the blow-up of the fiber
$\pi^{-1}(q)\simeq\P^1$.

\item\label{item:UmeF1F0}
The $\P^1$-bundle
$\hat\pi^{-1}(s_{-1})\to s_{-1}$
is isomorphic to the trivial $\P^1$-bundle $\F_0\to\P^1$.

\item\label{item:UmeF1index0}
$P_0\neq 0$, or equivalently, $y_0$ does not divide $P$.
\end{enumerate}
\end{proposition}

\begin{proof}

$\ref{item:UmeF1Contraction}\Rightarrow\ref{item:UmeF1F0}:$
As $\pi$ is locally trivial, there exists an open subset
$U_q\subseteq\P^2$ containing $q$ such that
$\pi^{-1}(U_q)\simeq U_q\times\P^1$. Since $\psi$ is the blow-up of
the fiber over $q$, its restriction corresponds to
$\kappa\times\Id_{\P^1}\colon\hat U_q\times\P^1\to U_q\times\P^1$,
where $\kappa\colon\hat U_q\to U_q$ is the blow-up of $q$.
This proves that $\hat\pi^{-1}(s_{-1})\simeq\P^1\times\P^1$.

$\ref{item:UmeF1F0}\Rightarrow\ref{item:UmeF1index0}:$ 
We assume that \ref{item:UmeF1index0} does not hold and prove that \ref{item:UmeF1F0} does not hold either.
The surface $\hat\pi^{-1}(s_{-1})$ is obtained by restricting
$\nu_{c,P}$ to $y_0=0$. Since $P_0=0$, we have $P(0,y_1,z)=0$, and the gluing function becomes
\[
([x_0:x_1],z)\longmapsto([x_0:x_1z^c],1/z).
\]
Hence, $\hat\pi^{-1}(s_{-1})\simeq\F_c$, which is not isomorphic to $\F_0$
since $c\geq2$.

$\ref{item:UmeF1index0}\Rightarrow\ref{item:UmeF1Contraction}:$ 
Assume that $P_0\neq0$. By Proposition~\ref{prop:birP1bUme}, we have an infinite sequence of elementary links
\[
Y_0\dashrightarrow Y_1\dashrightarrow Y_2\dashrightarrow\cdots,
\]
where $Y_n=Z_1^{b+n,c+n,(y_0z)^nP}$ for each $n\geq0$, with $Y_0=X$.
Each link conjugates the neutral components of the automorphism
groups, and every $\Autz(X)$-equivariant birational map
of $\P^1$-bundles over $\F_1$ is, up to isomorphisms
of $\P^1$-bundles over $\F_1$, a composition of these
elementary links and their inverses.

The existence of a $\P^1$-bundle $\pi\colon\V\to\P^2$ and a birational map $\psi\colon X\dashrightarrow\V$ such that $\pi\psi=\eta\hat\pi$ follows from the descent lemma (Lemma~\ref{lem:descent-lemma}), as does the uniqueness of $\pi$ and $\psi$ up to isomorphism. Moreover, the descent lemma implies that $\psi$ is $\Autz(X)$-equivariant. We now prove that $\psi$ is the blow-up of the fiber $\pi^{-1}(q)$, and hence is a birational morphism.

Let $Y\to\V$ be the blow-up of the fiber $\pi^{-1}(q)$. The composition $Y\dashrightarrow X\stackrel{\hat\pi}{\longrightarrow}\F_1$ defines a $\P^1$-bundle $Y\to\F_1$, and $X\dashrightarrow Y$ is an $\Autz(X)$-equivariant birational map of $\P^1$-bundles. Hence, $Y\to\F_1$ is square isomorphic to $Y_i\to\F_1$ for some $i\geq0$. Applying the implication
$\ref{item:UmeF1Contraction}\Rightarrow\ref{item:UmeF1index0}$
to $Y\to\V$, we obtain that $y_0$ does not divide $(y_0z)^iP$. This implies that $i=0$, since $y_0$ does not divide $P$. Thus, $Y\to\F_1$ is square isomorphic to $X\to\F_1$, and the uniqueness in the descent lemma implies that $\psi$ is, up to isomorphism, the blow-up of $\pi^{-1}(q)$. This proves~\ref{item:UmeF1Contraction}.
\end{proof}

\begin{lemma}\label{Lem:UmeF1P2-positive-char}
	Let $\hat\pi\colon X=\U_{1,\omega}^{b,c}=Z_1^{b,c,P}\to\F_1$
	be an Umemura $\P^1$-bundle, with
	\[
	P(y_0,y_1,z)
	=
	\sum_{i=0}^b y_0^iy_1^{b-i}P_i(z)z^{i+1},
	\]
	and assume that the action of $\Autz(X)$ on $\F_1$ has two orbits
	and that $P_0\neq0$.
	Let $\pi\colon\V\to\P^2$ be the associated $\P^1$-bundle,
	obtained by contracting $\hat\pi^{-1}(s_{-1})$ onto $\pi^{-1}(q)$
	via $\psi\colon X\to\V$, as in
	Proposition~$\ref{prop:UmeF1Contraction}$.
	
	Then $c=2r$, where $r=p^s$ for some $s\geq0$.
	Setting $d=c-b\in\Z$, the following assertions hold.
	\begin{enumerate}
		\item\label{item:UmeP2-lines-through-q}
		If $L\subset\P^2$ is a line through $q$, then
		$\pi^{-1}(L)\to L\simeq\P^1$
		is isomorphic to $\F_b\to\P^1$.
		
		\item\label{item:UmeP2-lines-not-through-q}
		If $L\subset\P^2$ is a line not passing through $q$, then
		\[
		\pi^{-1}(L)\simeq
		\begin{cases}
			\F_{\lvert d\rvert},&P_b=0,\\
			\F_0,&P_b\neq0.
		\end{cases}
		\]
		The second case occurs only when $p=2$, $b=r\geq2$, and
		\[
		P=\lambda_0y_1^rz^r+\lambda_b y_0^rz^{3r/2},
		\qquad \lambda_0,\lambda_b\in\k^*.
		\]
		
		\item\label{item:UmeP2-aut}
		We have $\psi\Autz(X)\psi^{-1}\subseteq\Autz(\V)$.
		Moreover:
		\begin{enumerate}
			\item If $b\neq r$ or $P_b\neq0$, then
			\[
			\psi\Autz(X)\psi^{-1}=\Autz(\V).
			\]
			\item If $b=r$ and $P_b=0$, then, after rescaling $P$,
			we have $P=y_1^rz^r$ and
			\[
			\V\simeq\P((F^s)^*T_{\P^2})\simeq\L^s
			\qquad (\text{see Remark~}\ref{rk: Ps}).
			\]
			In this case, the action on the base induces an isomorphism
			\[
			\Autz(\V)\xrightarrow{\sim}\Aut(\P^2)\simeq\PGL_3,\qquad
			\text{and} \qquad 
			\psi\Autz(X)\psi^{-1}\subsetneq\Autz(\V).
			\]
		\end{enumerate}
	\end{enumerate}
\end{lemma}  
\begin{proof}
	Recall that $X$ is obtained by
	gluing two copies of $\F_b\times\A^1$ along
	$\F_b\times(\A^1\setminus\{0\})$ through the automorphism
	$\nu_{c,P}\in\Aut(\F_b\times(\A^1\setminus\{0\}))$ defined by
	\[
	\nu_{c,P}\colon
	([x_0:x_1;y_0:y_1],z)
	\longmapsto
	\left(
	[x_0:x_1z^c+x_0P(y_0,y_1,z);y_0z:y_1],
	\frac1z
	\right).
	\]
	The $\P^1$-bundle $\hat\pi\colon X\to\F_1$ sends
	$([x_0:x_1;y_0:y_1],z)\in\F_b\times\A^1$ to
	$[y_0:y_1;1:z]\in\F_1$ on the first chart and to
	$[y_0:y_1;z:1]\in\F_1$ on the second chart
	(see Definition~\ref{Not:ZabcP}).
	
	We now describe the descended $\P^1$-bundle
	$\pi\colon\V\to\P^2$ using the contraction
	\[
	\eta\colon\F_1\to\P^2,\qquad
	[y_0:y_1;z_0:z_1]\longmapsto[y_0z_0:y_1:y_0z_1].
	\]
	We consider the two open subsets of $\P^2$ given by
	\[
	U_0=\{X\neq0\},
	\qquad
	U_1=\{Z\neq0\}.
	\]
	A point $[1:u:v]\in U_0$ corresponds to
	$[1:u;1:v]\in\F_1$, and the points of its fiber in $X$
	are represented by $([x_0:x_1;1:u],v)$ in the first chart.
	Similarly, a point $[v:u:1]\in U_1$ corresponds to
	$[1:u;v:1]\in\F_1$, and the points of its fiber in $X$
	are represented by $([x_0:x_1;1:u],v)$ in the second chart.
	
	Hence $\pi$ is trivial over $U_0$ and $U_1$.
	To compute the transition function, we calculate, for $v\neq0$,
	\[
	\begin{aligned}
		\nu_{c,P}([x_0:x_1;1:u],v)
		&=([x_0:x_1v^c+x_0P(1,u,v);v:u],1/v)\\
		&=([x_0:x_1v^{c-b}+x_0v^{-b}P(1,u,v);1:u/v],1/v).
	\end{aligned}
	\]
	Thus the gluing function is
	\[
	\bigl([x_0:x_1],[1:u:v]\bigr)
	\longmapsto
	\left(
	[x_0:x_1v^{c-b}+x_0v^{-b}P(1,u,v)],
	[v^{-1}:uv^{-1}:1]
	\right).
	\]
	We write
	\[
	P=\sum_{i\in I}
	\lambda_i y_0^iy_1^{b-i}z^{i+p^{s_i}},
	\qquad
	i+2p^{s_i}=c,
	\qquad
	\lambda_i\in\k^*,
	\qquad
	s_i\geq0,
	\]
	as in Lemma~\ref{lem:UmeF1twoorbits}.
	Since $P_0\not=0$, we have $i_0=0\in I$ and
	\[
	P(0,y_1,z)=\lambda_0y_1^bz^{p^{s_0}}.
	\]
	Set $r_i=p^{s_i}$ for each $i\in I$, and write
	$s=s_0$ and $r=r_0$, so that $c=2r=2p^s$.
	Since the action has exactly two orbits, we have
	$r\mid b$ by Lemma~\ref{lem:UmeF1twoorbits}.
	In particular, $b\geq r$.
	
	If $L\subset\P^2$ is a line through $q$, its strict transform
	in $\F_1$ is a fiber $f$ of the ruling $\F_1\to\P^1$.
	Since $X$ has numerical invariants $(1,b,c)$, we have
	$\hat\pi^{-1}(f)\simeq\pi^{-1}(L)\simeq\F_b$.
	This proves~\ref{item:UmeP2-lines-through-q}.
	
	We now consider lines not passing through $q$.
	Since the action of $\Autz(X)$ on $\F_1$ has exactly two orbits,
	its image in $\Aut(\F_1)$ contains the translation subgroup
	$U=\k[z_0,z_1]_1$ (Lemma~\ref{lem:UmeF1twoorbits}).
	Together with the standard $\SL_2$, this subgroup acts transitively
	on the set of lines in $\P^2$ not passing through $q$.
	It is therefore enough to consider $L_0=\{Y=0\}$.
	
	On $L_0$, we have $u=0$.
	If $P_b=0$, then $P(1,0,v)=0$, and the transition function
	restricts to
	\[
	([x_0:x_1],v)\longmapsto
	([x_0:x_1v^d],1/v).
	\]
	Thus $\pi^{-1}(L_0)\simeq\F_{\lvert d\rvert}$.
	
	Suppose now that $P_b\neq0$, which is equivalent to $b\in I$.
	Since $b>0$, we have
	\[
	b=2(r-r_b),\qquad 0<r_b<r,
	\]
	and hence $b<2r$.
	Together with $r\mid b$, this gives $b=r$ and $r_b=r/2$.
	Since $r$ and $r_b$ are powers of $p$, necessarily $p=2$
	and $r\geq2$.
	Moreover, every $i\in I\setminus\{0\}$ satisfies
	\[
	0<i=2(r-r_i)\leq b=r,
	\]
	so $r/2\leq r_i<r$.
	Thus $r_i=r/2$ and $i=r$, giving
	\[
	P=\lambda_0y_1^rz^r+\lambda_b y_0^rz^{3r/2}.
	\]
	Writing $t=r/2$, the restriction to $L_0$ is represented by
	\[
	M(v)=
	\begin{bmatrix}
		1&0\\
		\lambda_b v^t&v^{2t}
	\end{bmatrix}.
	\]
	The identity
	\[
	\begin{bmatrix}
		1&-\lambda_b^{-1}v^{-t}\\
		0&1
	\end{bmatrix}
	M(v)
	\begin{bmatrix}
		v^t&1\\
		-\lambda_b&0
	\end{bmatrix}
	=
	\begin{bmatrix}
		v^t&0\\
		0&\lambda_b v^t
	\end{bmatrix}
	\]
	uses changes of trivialization regular on the target
	and source charts, respectively.
	Consequently, $\pi^{-1}(L_0)\simeq\F_0$.
	This proves~\ref{item:UmeP2-lines-not-through-q}.
	
	The inclusion
	\[
	\psi\Autz(X)\psi^{-1}\subseteq\Autz(\V)
	\]
	follows from the descent lemma (Lemma~\ref{lem:descent-lemma}).
	Set
	\[
	H=\Im\bigl(\Autz(X)\to\Aut(\F_1)\bigr),
	\qquad
	G=\Im\bigl(\Autz(\V)\to\Aut(\P^2)\bigr).
	\]
	Via the natural isomorphism
	$\Aut(\F_1)\simeq\Aut(\P^2,q)$, we have $H\subseteq G$.
	
	Suppose first that $b\neq r$ or $P_b\neq0$.
	If $P_b=0$, then $b\neq r$ implies
	$\lvert d\rvert=\lvert2r-b\rvert\neq b$.
	If $P_b\neq0$, then $b=r>0$, whereas the inverse image
	of every line not passing through $q$ is $\F_0$.
	In either case, the lines whose inverse images are isomorphic
	to $\F_b$ are precisely the lines through $q$.
	Hence $G$ preserves this pencil and fixes $q$.
	Every element of $\Autz(\V)$ therefore preserves
	the fiber $\pi^{-1}(q)$ and lifts to its blow-up $X$.
	Thus
	\[
	\psi\Autz(X)\psi^{-1}=\Autz(\V).
	\]
	
	Suppose now that $b=r$ and $P_b=0$.
	For every $i\in I\setminus\{0\}$, we would have
	\[
	0<i=2(r-r_i)<r,
	\]
	and hence $r/2<r_i<r$.
	No power of $p$ lies strictly between $r/2$ and $r$.
	Thus $I=\{0\}$ and, after rescaling $P$,
	\[
	P=y_1^rz^r.
	\]
	The fiber transition matrix is then obtained by applying the
	$s$-th Frobenius morphism to the entries of the transition matrix
	of $\P(T_{\P^2})$. Hence
	\[
	\V\simeq\P((F^s)^*T_{\P^2})\simeq\L^s.
	\]
	By Proposition~\ref{prop: LsbAut}, the action on the base induces
	an isomorphism
	\[
	\Autz(\V)\xrightarrow{\sim}\Aut(\P^2)\simeq\PGL_3.
	\]
	Since $\psi\Autz(X)\psi^{-1}$ fixes $q$ on the base, we obtain
	\[
	\psi\Autz(X)\psi^{-1}\subsetneq\Autz(\V).
	\]
	This proves~\ref{item:UmeP2-aut}.
\end{proof}

We follow the same strategy as in \cite[\S~5.5]{BFT}, using equivariant elementary links (Proposition~\ref{prop: description of the equ links Umemura}) and the contraction of the negative section of $\F_1$ (Lemma~\ref{Lem:UmeF1P2-positive-char}) to determine the relative maximality of Umemura $\P^1$-bundles in positive characteristic.

\begin{proposition}\label{prop:relative-max-Ume-F1}
	Let $\pi\colon X=\U_{1,\omega}^{b,c}=Z_1^{b,c,P}\to\F_1$
	be an Umemura $\P^1$-bundle. Write
	\[
	P=\sum_{i\in I}
	\lambda_i y_0^iy_1^{b-i}z^{i+p^{s_i}},
	\qquad
	i+2p^{s_i}=c,
	\qquad
	\lambda_i\in\k^*,
	\]
	and set
	$i_0=\min I$, $r=p^{s_{i_0}}$, and $m=b-i_0$. 
	Then $\Autz(X)$ is relatively maximal if and only if
	$r\mid m$ and either
	\[
	m\geq2r,
	\]
	or
	\[
	p=2,\qquad m=r,\qquad P_b\neq0.
	\]
	In all cases, $X$ is not stiff.
\end{proposition}

\begin{proof}
	If $m=0$, then $\Autz(X)$ is not relatively maximal by
	Proposition~\ref{prop:birP1bUme}\ref{i0bNotm}.
	Since $m<2r$ in this case, we are done.
	
	If $m\geq1$ and $r\nmid m$, then the action of $\Autz(X)$ on $\F_1$
	has at least three orbits and its image in $\Aut(\F_1)$ is contained in the standard $\GL_2$ by Lemma~\ref{lem:UmeF1twoorbits}.
	Thus $\Autz(X)$ preserves the positive section $s_1$.
	In this case, $\Autz(X)$ is not relatively maximal by
	Proposition~\ref{prop:case H=GL2/mu_a}.
	
	We now assume that $m\geq1$ and $r\mid m$.
	By Lemma~\ref{lem:UmeF1twoorbits}, the action of $\Autz(X)$ on $\F_1$
	has exactly two orbits.
	Moreover, $X$ is not stiff by Proposition~\ref{prop:birP1bUme}:
	there is an infinite sequence of links
	\[
	Y_{-i_0}\dashrightarrow Y_{-i_0+1}\dashrightarrow\cdots,
	\]
	where
	\[
	Y_n=Z_1^{b+n,c+n,(y_0z)^nP}
	\qquad\text{for each }n\geq-i_0,
	\qquad
	Y_0=X.
	\]
	Each link conjugates the neutral components of the automorphism groups.
	Replacing $X$ with $Y_{-i_0}$ therefore preserves relative maximality
	and leaves $m$ and $r$ unchanged.
	It also preserves the nonvanishing of the top component $P_b$.
	We may consequently assume that $i_0=0$, so that $m=b$,
	$P_0\neq0$, and $c=2p^{s_0}=2r$.
	
	Proposition~\ref{prop:UmeF1Contraction} yields a birational morphism
	$\psi\colon X\to\V$, where $\pi\colon\V\to\P^2$ is a $\P^1$-bundle.
	If $m=r$ and $P_b=0$, Lemma~\ref{Lem:UmeF1P2-positive-char} gives
	\[
	\psi\Autz(X)\psi^{-1}\subsetneq\Autz(\V).
	\]
	Hence $\Autz(X)$ is not relatively maximal.
	
	It remains to consider $m\geq2r$, or $m=r$ and $P_b\neq0$.
	In the latter case, $p=2$ by
	Lemma~\ref{Lem:UmeF1P2-positive-char}.
	In both cases, the same lemma gives
	\[
	\psi\Autz(X)\psi^{-1}=\Autz(\V).
	\]
	Write $G=\Autz(X)$, identifying it with $\Autz(\V)$ through $\psi$.
	Its action on $\P^2$ has exactly two orbits,
	namely $\{q\}$ and $\P^2\setminus\{q\}$.
	Thus $q$ is its unique fixed point, and no curve of $\P^2$ is invariant.
	
	Consequently, the only nontrivial equivariant
	birational modification of the base is the
	blow-up of $q$, which gives $\F_1$.
	On $\F_1$, the negative section $s_{-1}$
	is the only invariant curve, and there
	are no invariant points.
	
	The elementary-link factorization (Lemma~\ref{Lemm:DecElLinks})
	therefore reduces every equivariant square birational map
	to a birational map of $\p^1$-bundles over $\F_1$,
	up to composition with $\psi$ or $\psi^{-1}$.
	By Proposition~\ref{prop:birP1bUme}, such maps are compositions of
	links in the infinite sequence described above and their inverses,
	up to isomorphisms.
	These links, as well as $\psi$, conjugate the neutral components
	of the automorphism groups.
	Hence $\Autz(X)$ is relatively maximal.
\end{proof}

\begin{proposition}\label{prop:relative-max-Ume-Fa}
	Assume that $a\geq2$.
	Let $\pi\colon X=\U_{a,\omega}^{b,c}=Z_a^{b,c,P}\to\F_a$
	be an Umemura $\P^1$-bundle. Write
	\[
	P=\sum_{i\in I}
	\lambda_i y_0^iy_1^{b-i}z^{ai+p^{s_i}},
	\qquad
	ai+2p^{s_i}=c,
	\qquad
	\lambda_i\in\k^*,
	\]
	and set
	\[
	i_0=\min I,\qquad m=b-i_0,\qquad r=p^{s_{i_0}},
	\qquad H=\Im\bigl(\Autz(X)\to\Aut(\F_a)\bigr).
	\]
	Then $\Autz(X)$ is relatively maximal if and only if
	$i_0<b$ and $H$ acts on $\F_a$ with exactly two orbits.
	This is equivalent to $m>0$ and
	$a p^{\nu_p(m)}\geq r$.
	
	Assume that $\Autz(X)$ is relatively maximal. Then $X$ is not stiff.
	There is an infinite sequence of elementary links
	\[
	Y_{-i_0}\dashrightarrow Y_{-i_0+1}\dashrightarrow Y_{-i_0+2}\dashrightarrow\cdots,
	\]
	where $Y_n=Z_a^{b+n,c+an,(y_0z^a)^nP}$ for each $n\geq-i_0$
	and $Y_0=X$. Each link conjugates the neutral components of
	the automorphism groups, and every
	$\Autz(X)$-equivariant square birational map
	to another $\P^1$-bundle is, up to square isomorphisms,
	a composition of the above links and their inverses.
\end{proposition}

\begin{proof}
	If $i_0=b$, or equivalently if $m=0$, then
	Proposition~\ref{prop:birP1bUme}\ref{i0bNotm}
	shows that $\Autz(X)$ is not relatively maximal.
	We may thus assume that $i_0<b$.
	
	We first prove the numerical orbit criterion:
	$H$ acts on $\F_a$ with exactly two orbits if and only if
	$ap^{\nu_p(m)}\geq r$.
	Set $q=p^{\nu_p(m)}$ and $U=\k[z_0,z_1]_a$.
	For each $i\in I$ with $i<b$, put
	$q_i=p^{\nu_p(b-i)}$.
	
	Suppose that $aq\geq r$.
	For $i\in I$ with $i>i_0$, the equality
	$a(i-i_0)=2(r-p^{s_i})$ gives $s_{i_0}>s_i$.
	Taking $p$-adic valuations yields
	\[
	\nu_p(a)+\nu_p(i-i_0)=\nu_p(2)+s_i,
	\]
	so
	\[
	a p^{\nu_p(i-i_0)}
	=
	\frac{a}{p^{\nu_p(a)}}p^{\nu_p(2)+s_i}
	\geq p^{s_i}.
	\]
	Since $b-i=m-(i-i_0)$, it follows that, for $i<b$,
	\[
	aq_i\geq
	\min\bigl(aq,a p^{\nu_p(i-i_0)}\bigr)
	\geq p^{s_i}.
	\]
	The same inequality holds for $i=i_0$ by assumption.
	By Lemma~\ref{Lemm:ActionAutFaParSpace}\ref{ActionFap},
	the translation $R=\mu z_0^a$, with $\mu\in\k$,
	acts by replacing $y_1$ with $y_1-\mu y_0$.
	For $i<b$, every additional term arising from
	\[
	\lambda_i y_0^i(y_1-\mu y_0)^{b-i}z^{ai+p^{s_i}}
	\]
	has $y_0$-degree at least $i+q_i$
	and $z$-degree $ai+p^{s_i}\leq a(i+q_i)$,
	hence vanishes in normal form \eqref{eq:def of P}.
	The component with $i=b$, if present, is unchanged.
	Thus $\{\mu z_0^a\mid\mu\in\k\}\subseteq H$.
	Together with the standard $\SL_2$, this subgroup
	acts transitively on $\F_a\setminus s_{-a}$.
	Since the standard $\SL_2$ acts transitively on $s_{-a}$,
	$H$ acts with exactly two orbits.
	
	Conversely, suppose that $aq<r$, and let $R\in H\cap U$.
	By Lemma~\ref{Lemm:ActionAutFaParSpace}\ref{ActionFap},
	the action of $R$ replaces $P$ with
	$P(y_0,y_1-y_0R(1,z),z)$.
	The coefficient of
	$y_0^{i_0+q}y_1^{m-q}z^{ai_0+r}$
	in the transformed $i_0$-component is
	\[
	\lambda_{i_0}\binom{m}{q}(-R(1,0))^q.
	\]
	This monomial survives in normal form because
	\[
	a(i_0+q)<ai_0+r<c=ai_0+2r.
	\]
	No other component contributes to it, since
	\[
	ai+p^{s_i}=c-p^{s_i}>c-r=ai_0+r
	\qquad\text{for }i>i_0.
	\]
	As the first nonzero component is unchanged,
	preservation of the parameter class forces $R(1,0)=0$,
	using $\binom{m}{q}\neq0$ in $\k$.
	Since $H\cap U$ is stable under the standard $\SL_2$,
	applying this conclusion to all conjugates of $R$
	shows that $R$ vanishes everywhere on $\P^1$.
	Hence $(H\cap U)(\k)=\{0\}$.
	
	Write an element of $H(\k)$ as $(R,\alpha)$.
	Since $\alpha$ normalizes the standard subgroup $S$,
	conjugating $S\subseteq H$ by $(R,\alpha)$ and using
	$(H\cap U)(\k)=\{0\}$ shows that $R$ is $S$-invariant.
	As $U^S=\{0\}$, we obtain $R=0$.
	Since $H$ is reduced, it follows that
	$H\subseteq\GL_2/\mu_a$.
	Thus $H$ acts with three orbits.
	This proves the numerical orbit criterion.
	
	By Propositions~\ref{prop:case H=GL2/mu_a}
	and~\ref{prop:trichotomy for algebraic subgroups of AUt(Fa)},
	the three-orbit case is not relatively maximal.
	We may therefore assume that $H$ acts with exactly two orbits.
	
	Since $\F_a$ is minimal and $H$ has no fixed point,
	every $H$-equivariant birational map from $\F_a$ to
	a smooth projective surface is an isomorphism.
	Moreover, $s_{-a}$ is the unique $H$-invariant curve,
	and $H$ acts transitively on it.
	Thus Lemma~\ref{Lemm:DecElLinks} applies:
	up to square isomorphisms, every equivariant square
	birational map factors into elementary links above $s_{-a}$.
	As $i_0<b$, Proposition~\ref{prop:birP1bUme}\ref{i0PbOnly}
	shows that these links are precisely those in
	Proposition~\ref{prop:birP1bUme}\ref{SequenceLinks}
	and their inverses.
	All these links conjugate the neutral components of
	the automorphism groups.
	This proves relative maximality and the asserted description
	of equivariant square birational maps.
	The existence of these nontrivial elementary links also shows
	that $X$ is not stiff.
\end{proof}

\begin{corollary}\label{Cor:MaxUmeBundlesFa}
	Let $X=\U_{a,\omega}^{b,c}=Z_a^{b,c,P}\to\F_a$
	be an Umemura $\P^1$-bundle, where
	\[
	P=\sum_{i\in I}
	\lambda_i y_0^iy_1^{b-i}z^{ai+p^{s_i}},
	\qquad
	ai+2p^{s_i}=c,
	\qquad
	\lambda_i\in\k^*.
	\]
	Set $i_0=\min I$, $m=b-i_0$ and $r=p^{s_{i_0}}$.
	Then $\Autz(X)$ is relatively maximal if and only if
	$m>0$ and one of the following holds:
	\begin{enumerate}
		\item $a=1$, $r\mid m$, and either $m\geq2r$ or
		$p=2,\ m=r,\ P_b\neq0$;
		\item $a\geq2$ and $a p^{\nu_p(m)}\geq r$.
	\end{enumerate}
	In all these cases, $X$ is not stiff.
\end{corollary}

\begin{proof}
	For $a=1$, the assertions follow from
	Proposition~\ref{prop:relative-max-Ume-F1}.
	For $a\geq2$, they follow from
	Proposition~\ref{prop:relative-max-Ume-Fa}.
\end{proof}

\subsection{Proof of Theorem~\ref{th: classification rel max}} \label{proof: classification rel max}
The relative maximality assertions
in cases~\eqref{item:relmax-Fabc}--\eqref{item:relmax-Schwarzenberger}
follow from Proposition~\ref{prop:MaxDecBundlesFa},
Lemma~\ref{Lem:DecP2NoCurve},
Corollary~\ref{Cor:MaxUmeBundlesFa}, and
Proposition~\ref{prop: relative max of Sbc}, respectively.
The transitivity of $\Autz(\L^s)$
(Lemma~\ref{Lem:LsPGL3}) gives relative maximality in case~\eqref{item:relmax-Ls}.
For case~\eqref{item:relmax-V},
Lemma~\ref{Lem:UmeF1P2-positive-char} gives
$\psi\Autz(\U_{1,\omega}^{b,2r})\psi^{-1}
=\Autz(\V_{\omega}^{b,2r}),$
so the relative maximality follows from
Proposition~\ref{prop:relative-max-Ume-F1}.

For an arbitrary $\P^1$-bundle, Lemma~\ref{lem:descent-lemma}
first gives an equivariant square birational model over
$\P^2$ or a minimal Hirzebruch surface.
Theorems~\ref{th:P1bundlesFa} and~\ref{th:P1bundlesP2}
then reduce us to the families studied above; in the
fixed-point case over $\P^2$, we first blow up the fixed
point and its fiber.
The non-maximal cases admit the equivariant enlargements
described in Proposition~\ref{prop:MaxDecBundlesFa},
Proposition~\ref{prop:case H=GL2/mu_a},
Proposition~\ref{prop: description of the equ links Umemura},
Lemma~\ref{lem: possible birational models Sbc}, and
Lemma~\ref{lem:Lsb-non-maximal}.
These constructions terminate in one of the listed cases,
proving the existence assertion.

\smallskip

Suppose now that $\Autz(X)$ is relatively maximal.
By Lemma~\ref{lem:descent-lemma}, descent to a minimal
base gives a relatively maximal $\P^1$-bundle $X_0\to S_0$
and conjugates the full neutral automorphism groups.
The preceding classification and maximality results
give cases~\eqref{item:relmax-Fabc}
and~\eqref{item:relmax-Umemura} when $S_0$ is a
Hirzebruch surface, and cases~\eqref{item:relmax-Pb},
\eqref{item:relmax-Schwarzenberger}
and~\eqref{item:relmax-Ls} when $S_0=\P^2$
and the base action has no fixed point.

If the action on $\P^2$ fixes a point $q$, blowing up
$q$ and its fiber gives a relatively maximal $\P^1$-bundle
over $\F_1$. By Theorem~\ref{th:P1bundlesFa} and the
preceding maximality results, this $\P^1$-bundle must be
Umemura. 
Since this $\P^1$-bundle is the pullback of $X_0\to\P^2$
along $\eta\colon\F_1\to\P^2$, its restriction to
the exceptional section $E=\eta^{-1}(q)$ is
$(X_0)_q\times E\longrightarrow E$, 
hence is trivial.
Proposition~\ref{prop:UmeF1Contraction} gives
$i_0=0$ and $c=2r$. Proposition~\ref{prop:relative-max-Ume-F1}
then gives $r\mid b$ and $b\geq2r$, yielding
case~\eqref{item:relmax-V}.

Finally, we must recover the original base $S$.
The descent morphism $S\to S_0$ is equivariant for the
full neutral automorphism group of $X_0$, since descent
conjugates the neutral automorphism groups.
It factors into equivariant point blow-ups, whose centers
are fixed points because the group is connected.
All the listed actions on $S_0$ have no fixed point,
except in case~\eqref{item:relmax-V}, where $q$ is the
unique fixed point. Thus either $S\simeq S_0$, or the
first blow-up is at $q$ and gives $\F_1$.
The lifted action on $\F_1$ has no fixed point,
so no further blow-up is possible.

Hence the only additional possibility is $S\simeq\F_1$;
Theorem~\ref{th:P1bundlesFa} and
Proposition~\ref{prop:relative-max-Ume-F1} then give
case~\eqref{item:relmax-Umemura}.
This proves that the list is exhaustive. \qed

\subsection{Proof of Theorem~\ref{th: equivariant links}} \label{proof: equivariant links}
The stiffness and superstiffness assertions in
cases \eqref{item:relmax-Fabc}--\eqref{item:relmax-Schwarzenberger}
follow from Proposition~\ref{prop:MaxDecBundlesFa},
Lemma~\ref{Lem:DecP2NoCurve},
Corollary~\ref{Cor:MaxUmeBundlesFa}, and
Proposition~\ref{prop: relative max of Sbc}, respectively.
The transitivity of $\Autz(\L^s)$
(Lemma~\ref{Lem:LsPGL3}) implies superstiffness.
In case~\eqref{item:relmax-V}, the inverse contraction gives
an equivariant square birational map to a $\P^1$-bundle over $\F_1$,
so these $\P^1$-bundles are not stiff.

The description and exhaustiveness of the decomposable
and Umemura links follow from the cited propositions,
together with Propositions~\ref{prop:relative-max-Ume-F1}
and~\ref{prop:relative-max-Ume-Fa}.

For Schwarzenberger $\P^1$-bundles, the proof of
Proposition~\ref{prop: relative max of Sbc} shows that the
base map is an isomorphism and that every target is again
a Schwarzenberger $\P^1$-bundle.
Lemma~\ref{Lemm:DecElLinks} then factors the map into
elementary links above $\Gamma$, whose centers are the
invariant sections described in
Corollary~\ref{cor: description of the orbits for Sbc}.

Finally, in case~\eqref{item:relmax-V}, the base action has
a unique fixed point $q$ and no invariant curve.
Blowing up $q$ gives $\F_1$, whose lifted action has no
fixed point. Thus the only additional base modification
is $\F_1\to\P^2$ or its inverse, giving precisely the
morphisms of Proposition~\ref{prop:UmeF1Contraction}.
This proves exhaustiveness of the list.
The equality of neutral automorphism groups follows
from relative maximality.
\qed

\bibliographystyle{alpha}
\bibliography{biblio}

@article {Fong,
    AUTHOR = {Fong, Pascal},
     TITLE = {Automorphism groups of {$\mathbb{P}^1$}-bundles over geometrically
              ruled surfaces},
   JOURNAL = {J. Lond. Math. Soc. (2)},
  FJOURNAL = {Journal of the London Mathematical Society. Second Series},
    VOLUME = {113},
      YEAR = {2026},
    NUMBER = {6},
     PAGES = {Paper No. e70593, 87},
      ISSN = {0024-6107,1469-7750},
   MRCLASS = {14E07 (14E30 14J26 14J60)},
  MRNUMBER = {5084120},
       DOI = {10.1112/jlms.70593},
       URL = {https://doi.org/10.1112/jlms.70593},
}

@book{Ste68,
 author = {Steinberg, R.},
 title = {Endomorphisms of linear algebraic groups},
 fseries = {Memoirs of the American Mathematical Society},
 series = {Mem. Am. Math. Soc.},
 issn = {0065-9266},
 volume = {80},
 isbn = {978-0-8218-1280-8; 978-1-4704-0029-3},
 year = {1968},
 publisher = {Providence, RI: American Mathematical Society (AMS)},
 language = {English},
 doi = {10.1090/memo/0080},
 zbMATH = {3262382},
 Zbl = {0164.02902}
}

@misc{Mac25,
 author = {Maccan, Matilde},
 title = {Automorphism group of flag varieties with non-reduced stabilizer},
 year = {2025},
 howpublished = {Preprint, {arXiv}:2507.04549 [math.{AG}] (2025)},
 url = {https://arxiv.org/abs/2507.04549},
 arXiv = {arXiv:2507.04549}
}

@article{BFT-MFS,
	author = {Blanc, J{\'e}r{\'e}my and Fanelli, Andrea and Terpereau, Ronan},
	title = {Connected algebraic groups acting on three-dimensional {Mori} fibrations},
	fjournal = {IMRN. International Mathematics Research Notices},
	journal = {Int. Math. Res. Not.},
	issn = {1073-7928},
	volume = {2023},
	number = {2},
	pages = {1572--1689},
	year = {2023},
	language = {English},
	doi = {10.1093/imrn/rnab293},
	zbMATH = {7652768},
	Zbl = {1510.14031}
}

@article{BFT,
 author = {Blanc, J{\'e}r{\'e}my and Fanelli, Andrea and Terpereau, Ronan},
 title = {Automorphisms of {{\(\mathbb{P}^1\)}}-bundles over rational surfaces},
 fjournal = {{\'E}pijournal de G{\'e}om{\'e}trie Alg{\'e}brique. EPIGA},
 journal = {{\'E}pijournal de G{\'e}om. Alg{\'e}br., EPIGA},
 issn = {2491-6765},
 volume = {6},
 pages = {47},
 note = {Id/No 23},
 year = {2022},
 language = {English},
 zbMATH = {7641777},
 Zbl = {1511.14027}
}

@book {Jan03,
    AUTHOR = {Jantzen, Jens Carsten},
     TITLE = {Representations of algebraic groups},
    SERIES = {Mathematical Surveys and Monographs},
    VOLUME = {107},
   EDITION = {Second},
 PUBLISHER = {American Mathematical Society, Providence, RI},
      YEAR = {2003},
     PAGES = {xiv+576},
      ISBN = {0-8218-3527-0},
   MRCLASS = {20G05 (17B10)},
  MRNUMBER = {2015057},
}

@Article{EF1898,
    Author = {Enriques, F. and Fano, G.},
    Title = {{Sui gruppi continui di trasformazioni Cremoniane dello spazio.}},
    FJournal = {{Annali di Matematica Pura ed Applicata. II. Serie}},
    Journal = {{Annali di Mat. (2)}},
    ISSN = {0373-3114; 1618-1891/e},
    Volume = {26},
    Pages = {59--98},
    Year = {1898},
    Language = {Italian},
    DOI = {10.1007/BF02346204},
    Zbl = {28.0598.03}
}

@article {MO67,
    AUTHOR = {Matsumura, Hideyuki and Oort, Frans},
     TITLE = {Representability of group functors, and automorphisms of algebraic schemes},
   JOURNAL = {Invent. Math.},
  FJOURNAL = {Inventiones Mathematicae},
    VOLUME = {4},
      YEAR = {1967},
     PAGES = {1--25},
      ISSN = {0020-9910},
   MRCLASS = {14.55},
  MRNUMBER = {0217090},
MRREVIEWER = {M. Miyanishi},
}

@incollection {MU83,
    AUTHOR = {Mukai, Shigeru and Umemura, Hiroshi},
     TITLE = {Minimal rational threefolds},
 BOOKTITLE = {Algebraic geometry ({T}okyo/{K}yoto, 1982)},
    SERIES = {Lecture Notes in Math.},
    VOLUME = {1016},
     PAGES = {490--518},
 PUBLISHER = {Springer, Berlin},
      YEAR = {1983},
   MRCLASS = {14J30 (14E35 14M20)},
  MRNUMBER = {726439 (85c:14027)},
MRREVIEWER = {Mary Schaps},
       DOI = {10.1007/BFb0099976},
       URL = {http://dx.doi.org/10.1007/BFb0099976},
}

@article {Ume82a,
    AUTHOR = {Umemura, Hiroshi},
     TITLE = {Maximal algebraic subgroups of the {C}remona group of three
              variables. {I}mprimitive algebraic subgroups of exceptional
              type},
   JOURNAL = {Nagoya Math. J.},
  FJOURNAL = {Nagoya Mathematical Journal},
    VOLUME = {87},
      YEAR = {1982},
     PAGES = {59--78},
      ISSN = {0027-7630},
     CODEN = {NGMJA2},
   MRCLASS = {14E07 (14M20)},
  MRNUMBER = {676586 (84b:14005)},
MRREVIEWER = {M. Kh. Gizatullin},
       URL = {http://projecteuclid.org/euclid.nmj/1118786899},
}

@article {Ume80,
    AUTHOR = {Umemura, Hiroshi},
     TITLE = {Sur les sous-groupes alg\'ebriques primitifs du groupe de
              {C}remona \`a trois variables},
   JOURNAL = {Nagoya Math. J.},
  FJOURNAL = {Nagoya Mathematical Journal},
    VOLUME = {79},
      YEAR = {1980},
     PAGES = {47--67},
      ISSN = {0027-7630},
     CODEN = {NGMJA2},
   MRCLASS = {14E05 (14M20 57S25)},
  MRNUMBER = {587409},
MRREVIEWER = {M. Kh. Gizatullin},
       URL = {http://projecteuclid.org/euclid.nmj/1118786168},
}

@article {Ume82b,
    AUTHOR = {Umemura, Hiroshi},
     TITLE = {On the maximal connected algebraic subgroups of the {C}remona
              group. {I}},
   JOURNAL = {Nagoya Math. J.},
  FJOURNAL = {Nagoya Mathematical Journal},
    VOLUME = {88},
      YEAR = {1982},
     PAGES = {213--246},
      ISSN = {0027-7630},
     CODEN = {NGMJA2},
   MRCLASS = {14E07 (14M17 14M20)},
  MRNUMBER = {683251 (84g:14013)},
MRREVIEWER = {M. Kh. Gizatullin},
       URL = {http://projecteuclid.org/euclid.nmj/1118787013},
}

@incollection {Ume85,
    AUTHOR = {Umemura, Hiroshi},
     TITLE = {On the maximal connected algebraic subgroups of the {C}remona
              group. {II}},
 BOOKTITLE = {Algebraic groups and related topics ({K}yoto/{N}agoya, 1983)},
    SERIES = {Adv. Stud. Pure Math.},
    VOLUME = {6},
     PAGES = {349--436},
 PUBLISHER = {North-Holland, Amsterdam},
      YEAR = {1985},
   MRCLASS = {14E07 (14L30)},
  MRNUMBER = {803342 (87d:14008)},
MRREVIEWER = {M. Kh. Gizatullin},
}

@article {Ume88,
    AUTHOR = {Umemura, Hiroshi},
     TITLE = {Minimal rational threefolds. {II}},
   JOURNAL = {Nagoya Math. J.},
  FJOURNAL = {Nagoya Mathematical Journal},
    VOLUME = {110},
      YEAR = {1988},
     PAGES = {15--80},
      ISSN = {0027-7630},
     CODEN = {NGMJA2},
   MRCLASS = {14J30 (14E07 14E30)},
  MRNUMBER = {945907 (89f:14040)},
MRREVIEWER = {Mary Schaps},
       URL = {http://projecteuclid.org/euclid.nmj/1118780959},
}

@phdthesis{Lau,
	Author = {Laurent, Bruno},
	Title = {Courbes et surfaces presque homogènes},
	School = {Université Grenoble Alpes},
	Year = {2018},
	Url = {https://theses.hal.science/tel-01999851},
}

@book {BSU13,
    AUTHOR = {Brion, Michel and Samuel, Preena and Uma, Vikraman},
     TITLE = {Lectures on the structure of algebraic groups and geometric
              applications},
    SERIES = {CMI Lecture Series in Mathematics},
    VOLUME = {1},
 PUBLISHER = {Hindustan Book Agency, New Delhi; Chennai Mathematical
              Institute (CMI), Chennai},
      YEAR = {2013},
     PAGES = {viii+120},
      ISBN = {978-93-80250-46-5},
   MRCLASS = {14L15 (14J60 14K05 14L30 20G15)},
  MRNUMBER = {3088271},
MRREVIEWER = {Andy R. Magid},
}

\end{document}